\documentclass[a4paper,reqno,12pt,twoside]{amsart}
\usepackage[T1]{fontenc}
\usepackage{amssymb}
\usepackage{mathrsfs}
\usepackage{graphicx}
\usepackage{psfrag}

\theoremstyle{plain}
\newtheorem{prop}{Proposition}[section]
\newtheorem{thm}[prop]{Theorem}
\newtheorem{lem}[prop]{Lemma}
\newtheorem{cor}[prop]{Corollary}

\numberwithin{equation}{section}

\theoremstyle{definition}
\newtheorem{dfn}[prop]{Definition}
\newtheorem{ex}[prop]{Example}

\theoremstyle{remark}
\newtheorem*{rmk}{Remark}

\DeclareMathOperator{\sgn}{sgn}

\DeclareMathOperator{\re}{Re}
\DeclareMathOperator{\im}{Im}
\DeclareMathOperator{\supp}{supp}
\begin{document}
\title[Microlocal properties of the range]
{On some microlocal properties of the range of a pseudo-differential operator of principal type}
\author{Jens Wittsten}
\address{Center for Mathematical Sciences\\
Lund University\\
Box 118\\
S-221 00 Lund\\
Sweden}
\email{jens.wittsten@math.lu.se}
\thanks{Research supported in part by the Swedish Research Council}
\begin{abstract}
The purpose of this paper is to obtain microlocal analogues of
results by L. Hörmander about inclusion relations between the ranges of
first order differential operators with coefficients in $C^\infty$
which fail to be locally solvable. Using similar techniques, we shall study the properties of the
range of classical pseudo-differential operators of principal type which fail
to satisfy condition $(\varPsi)$.
\end{abstract}
\maketitle
\section{Introduction}
\noindent
In this paper we shall study the properties of the range of a
classical pseudo-differential operator $P\in \varPsi_{\mathrm{cl}}^m(X)$
that is not locally solvable, where $X$ is a $C^\infty$
manifold of dimension $n$. Here, classical means that the
total symbol of $P$ is an asymptotic sum of homogeneous terms,
\[\sigma_P(x,\xi)=p_m(x,\xi)+p_{m-1}(x,\xi)+\ldots,
\]
where $p_k$ is homogeneous of degree $k$ in $\xi$ and
$p_m$ denotes the principal symbol of $P$. When no confusion can occur
we will simply refer to $\sigma_P$ as the symbol of $P$. We shall restrict our study
to operators of principal type, which means that the Hamilton vector field
$H_{p_m}$ and the radial vector field are linearly independent when $p_m=0$.
We shall also assume that all operators are properly supported, that is, both projections
from the support of the kernel in $X\times X$ to $X$ are proper maps.
For such operators,
local solvability at a compact set $M\subset X$ means that for every $f$ in a
subspace of $C^\infty(X)$ of finite codimension there is a distribution $u$ in
$X$ such that
\begin{equation}\label{eqintrolocsolv}
Pu=f
\end{equation}
in a neighborhood of $M$. We can also define microlocal solvability at a set in the cosphere bundle,
or equivalently, at a conic set in $T^\ast(X)\smallsetminus 0$, the cotangent bundle of $X$ with the
zero section removed. By a conic set $K\subset T^\ast(X)\smallsetminus 0$ we mean a set that is
conic in the fiber, that is,
\[
(x,\xi)\in K \quad \Longrightarrow \quad (x,\lambda \xi)\in K \quad \text{for all }\lambda>0.
\]
If, in addition, $\pi_x(K)$ is compact in $X$, where $\pi_x:T^\ast(X)\to X$ is the projection,
then $K$ is said to be compactly based.
Thus, we say that $P$ is solvable at the compactly based cone $K \subset T^{\ast}(X) \smallsetminus 0$
if there is an integer $N$ such that for every
$f \in H_{(N)}^{\mathrm{loc}}(X)$
there exists a $u \in \mathscr{D}'(X)$ with $K \cap W\! F(Pu-f)=\emptyset$ (see Definition \ref{defrange}).

The famous example due to Hans Lewy ~\cite{le} of the existence of functions $f\in C^\infty(\mathbb{R}^3)$
such that the equation
\[\partial_{x_1} u + i\partial_{x_2} u -2i(x_1+i x_2) \partial_{x_3} u = f
\]
does not have any solution $u\in \mathscr{D}'(\varOmega)$ in any open non-void
subset $\varOmega\subset \mathbb{R}^3$ 
contradicted the assumption that partial differential equations with smooth coefficients
behave as analytic partial differential equations, for which existence of analytic solutions
is guaranteed by the Cauchy-Kovalevsky theorem. This example
led to an extension due to Hörmander ~\cite{ho10,ho11}
in the sense of a necessary condition for a differential equation $P(x,D)u=f$ to have a
solution locally for every $f\in C^\infty$.
In fact (see ~\cite[Theorem $6.1.1$]{ho0}),
if $\varOmega$ is an open set in $\mathbb{R}^n$, and $P$ is a differential
operator of order $m$ with coefficients in $C^\infty(\varOmega)$ such that the differential equation
$P(x,D)u=f$
has a solution $u\in \mathscr{D}'(\varOmega)$ for every $f\in C_0^\infty(\varOmega)$, then
$\{p_m, \overline{p}_m \}$ must vanish at every point $(x,\xi)\in \varOmega \times \mathbb{R}^n$ for which
$p_m(x,\xi)=0$,
where
\[ \{ a , b \}=\sum_{j=1}^n \partial_{\xi_j}a \, \partial_{x_j}b- \partial_{x_j} a
\, \partial_{\xi_j} b
\]
denotes the Poisson bracket.

In addition to his example, Lewy conjectured that differential operators which fail
to have local solutions are essentially uniquely determined by the range. Later
Hörmander ~\cite[Chapter $6.2$]{ho0} proved that
if $P$ and $Q$ are two first order differential operators with coefficients in $C^\infty(\varOmega)$
and in
$C^1(\varOmega)$, respectively, such that
the equation $P(x,D)u=Q(x,D)f$ has a solution $u\in \mathscr{D}'(\varOmega)$
for every $f\in C_0^\infty(\varOmega)$, and $x$ is a point in $\varOmega$ such that
\begin{equation}\label{eq:introeq0.9}
p_1(x,\xi)=0 , \quad \{ p_1 , \overline{p}_1 \}(x,\xi)\neq 0
\end{equation}
for some $\xi\in \mathbb{R}^n$, then there is a constant $\mu$ such that (at the fixed point $x$)
\[ {}^t Q(x,D)=\mu\, {}^tP(x,D)
\]
where ${}^t Q$ and ${}^t P$ are the formal adjoints of $Q$ and $P$. If \eqref{eq:introeq0.9}
holds for a dense set of points
$x$ in $\varOmega$ and if the coefficients of $p_1(x,D)$ do not vanish simultaneously in $\varOmega$, then there is a
function $\mu \in C^1(\varOmega)$ such that
\begin{equation}\label{eq:introeq1}
Q(x,D) u = P(x,D) (\mu u).
\end{equation}
Furthermore, for such an operator $P$ and function $\mu
%\in C^1(\varOmega)
$, the equation $P(x,D)u=\mu P(x,D)f$
has a solution $u\in\mathscr{D}'(\varOmega)$ for every $f\in C_0^\infty(\varOmega)$ if and only if
$p_1(x,D)\mu=0$.

Hörmander also showed that
this result extends to operators of higher order in the following way (see
~\cite[Theorem $6.2.4$]{ho0}). If $P$ is a differential operator
of order $m$ with coefficients in $C^\infty(\varOmega)$ and $\mu$ is a function in $C^m(\varOmega)$ such
that the equation
\[P(x,D)u=\mu P(x,D)f
\]
has a solution $u\in \mathscr{D}'(\varOmega)$ for every $f\in C_0^\infty(\varOmega)$,
then it follows that
\begin{equation*}
\sum_{j=1}^n \partial_{\xi_j}p_m(x,\xi) \partial_{x_j} \mu(x)=0
\end{equation*}
for all $x\in \varOmega$ and $\xi\in \mathbb{R}^n$ such that
\begin{equation}\label{eqintrocond1}
\{ p_m, \overline{p}_m\}(x,\xi) \neq 0, \quad p_m(x,\xi)=0.
\end{equation}
This means that the derivative of $\mu$ must vanish along every
bicharacteristic element
with initial data $(x,\xi)$ giving rise to non-existence of solutions.

If $P$ is a pseudo-differential operator such that $P$ is microlocally elliptic near $(x_0,\xi_0)$,
then there exists a microlocal inverse,
called a parametrix $P^{-1}$ of $P$, such that in a conic neighborhood of $(x_0,\xi_0)$
we have $PP^{-1}=P^{-1}P=\mathrm{Identity}$ modulo smoothing operators. $P$ is then trivially seen to
be microlocally solvable near $(x_0,\xi_0)$, and for any pseudo-differential operator $Q$
we can write $Q=PP^{-1} Q+R=PE+R$ where $R$ is a smoothing operator. When the range of $Q$
is microlocally contained in the range of $P$, we will show the existence of this type of
representation for $Q$ in the case when $P$ is a non-solvable pseudo-differential operator of principal type,
although we will have to content ourselves with a weaker statement concerning the Taylor coefficients
of the symbol of the operator $R$ (see Theorem \ref{bigthm1} for the precise formulation of the result).
Note that when $P$ is solvable but non-elliptic we cannot hope to obtain such a representation
in general; see the remark on page \pageref{rmk:solvablebutnonelliptic1}.

For pseudo-differential operators of principal type, Hörmander ~\cite{ho4} proved
that local solvability in the sense of \eqref{eqintrolocsolv}
implies that $M$ has an open neighborhood $Y$ in $X$ where
$p_m$ satisfies condition $(\varPsi)$, which means that
\begin{multline}\label{intropsi}
\im ap_m \text{ does not change sign from $-$ to $+$} \\
\text{along the oriented bicharacteristics of } \re ap_m
\end{multline}
over $Y$ for any $0\neq a\in C^\infty(T^*(Y)\smallsetminus 0)$.
The oriented bicharacteristics are the positive flow-outs of the Hamilton vector field
$H_{\re ap_m}
%\neq 0
$ on $\re ap_m=0$.
The proof relies on
an idea due to Moyer ~\cite{mo},
and uses the fact that condition \eqref{intropsi} is invariant
under multiplication of $p_m$ with nonvanishing factors,
and conjugation of $P$ with elliptic
Fourier integral operators.

Rather recently Dencker ~\cite{de} proved that condition $(\varPsi)$ is also sufficient for local
and microlocal solvability for operators of principal type. To get local solvability at a point $x_0$, the
strong form of the nontrapping condition at $x_0$,
\begin{equation}\label{eq:intronontrapping}
p_m=0 \quad \Longrightarrow \quad \partial_\xi p_m\neq 0,
\end{equation}
was assumed. This was the original condition for principal type of Nirenberg and Treves ~\cite{nitr},
which is always obtainable microlocally after a canonical transformation.
Thus, we shall 
study
pseudo-differential operators that fail to satisfy condition $(\varPsi)$ in
place of the condition given by \eqref{eqintrocond1}, and
show that such operators are, in analogue with the
inclusion relations
between the ranges of differential operators that fail to be locally solvable,
essentially uniquely determined by the range.
However, note that even though \eqref{eqintrocond1} is a microlocal condition, one obtains
the mentioned local results for differential operators because of the analyticity in $\xi$
of the corresponding symbol. Since this is
not true in general for pseudo-differential operators, our results will be
inherently microlocal.
We will combine the techniques used in ~\cite{ho0} to prove the inclusion relations for
differential operators with
the approach used in ~\cite{ho4} to prove the necessity of condition $(\varPsi)$ for local
solvability of pseudo-differential operators of principal type.

It should be noted that
it is possible to extend these results to certain systems of
pseudo-differential operators. We are currently working on
a generalization to systems of principal type
and constant characteristics, although this is not adressed here.

The author is grateful to Professor Nils Dencker at Lund
University for suggesting the problem that led to the results presented here,
and also for many helpful discussions on the subject.

\section{Non-solvable Operators of Principal Type}

\noindent
Let $X$ be a $C^\infty$
manifold of dimension $n$. In what follows, $C$ will be taken to be a new constant
every time unless stated otherwise. We let $\mathbb{N}=\{0,1,2,\ldots\}$,
and if $\alpha\in\mathbb{N}^n$ is a multi-index $\alpha=(\alpha_1,\ldots,\alpha_n)$,
we let
\[
D_x^\alpha=D_{x_1}^{\alpha_1}\ldots D_{x_n}^{\alpha_n},
\]
where $D_{x_j}=-i\partial_{x_j}$.
We shall also employ the standard notation $f_{(\alpha)}^{(\beta)}(x,\xi)=\partial_x^\alpha\partial_\xi^\beta
f(x,\xi)$ for multi-indices $\alpha,\beta$.

In this section we will follow the outline of Chapter $26$, Section $4$ of ~\cite{ho4}.
Recall that the Sobolev space $H_{(s)}(X)$, $s\in\mathbb{R}$, is a local space,
that is, if $\varphi\in C_0^\infty(X)$ and $u\in H_{(s)}(X)$
then $\varphi u \in H_{(s)}(X)$, and the corresponding operator
of multiplication is continuous. Thus we can define
\[
H_{(s)}^{\mathrm{loc}}(X)=\{u\in \mathscr{D}'(X) :
\varphi u\in H_{(s)}(X), \forall \varphi \in C_0^\infty(X)\}.
\]
This is a Fréchet space, and its dual with respect to the
inner product on $L^2$ is $H_{(-s)}^{\mathrm{comp}}(X)=H_{(-s)}^{\mathrm{loc}}(X)\cap \mathscr{E}'(X)$.
\begin{dfn}\label{defrange}If $K \subset T^{\ast}(X) \smallsetminus 0$ is a compactly based cone we
shall say that the range of $Q\in \varPsi_{\mathrm{cl}}^m(X)$
is microlocally contained in the range of $P\in \varPsi_{\mathrm{cl}}^k(X)$ at $K$ if there
exists an integer $N$ such that for every
$f\in H_{(N)}^{\mathrm{loc}}(X)$, there exists a $u\in \mathscr{D}'(X)$ with $W\! F(Pu-Qf)\cap K = \emptyset$.
\end{dfn}
\noindent If $I\in\varPsi_{\mathrm{cl}}^0(X)$ is the identity on $X$, we obtain from Definition \ref{defrange}
the definition of microlocal solvability for a
pseudo-differential operator (see ~\cite[Definition $26.4.3$]{ho4}) by setting $Q=I$.
Thus, the range of the identity is microlocally contained in the range of $P$ at $K$
if and only if $P$ is microlocally solvable at $K$.
Note also that if $P$ and $Q$ satisfy Definition \ref{defrange} for some integer $N$,
then due to the inclusion
\[H_{(t)}^{\mathrm{loc}}(X)\subset H_{(s)}^{\mathrm{loc}}(X), \quad \textrm{if } s<t,
\]
the statement also holds for any integer $N' \geq N$. Hence $N$ can always be assumed to be positive.
Furthermore, the property is preserved if $Q$ is composed with a properly supported pseudo-differential
operator $Q_1\in \varPsi_{\mathrm{cl}}^{m'}(X)$ from the right.
Indeed,
let $g$ be an arbitrary function in
$H_{(N+m')}^{\mathrm{loc}}(X)$. Then $f=Q_1 g\in H_{(N)}^{\mathrm{loc}}(X)$
since $Q_1$ is continuous
\[
Q_1:H_{(s)}^{\mathrm{loc}}(X)\rightarrow H_{(s-m')}^{\mathrm{loc}}(X)
\]
for every $s\in\mathbb{R}$,
so by Definition \ref{defrange}
there exists a $u\in \mathscr{D}'(X)$ with $W\! F(Pu-Qf)\cap K = \emptyset$.
Hence the range of $QQ_1$
is microlocally contained in the range of $P$ at $K$ with the integer
$N$ replaced by $N+m'$.

The property given by Definition \ref{defrange} is also
preserved under composition of both $P$ and $Q$ with a properly supported pseudo-differential
operator from the left. This follows immediately
from the fact that properly supported
pseudo-differential operators are microlocal, that is,
\[
W\! F(Au)\subset W\! F(u)\cap W\! F(A), \quad u\in \mathscr{D}'(X).
\]
\begin{rmk}\label{rmk:solvabilityineprime}
It should be pointed out that in Definition \ref{defrange} we may always assume that
$f\in H_{(N)}^{\mathrm{comp}}(X)$ and $u\in\mathscr{E}'(X)$ when considering a fixed
cone $K$. In fact, assume
\[
Qf=Pu+g
\]
where $f\in H_{(N)}^{\mathrm{loc}}(X)$ and $u, g\in\mathscr{D}'(X)$
with $W\! F(g)\cap K=\emptyset$, and let $Y\Subset X$ satisfy $K\subset
T^\ast(Y)\smallsetminus 0$.
(We write $Y\Subset X$ when $\overline{Y}$ is compact and contained in $X$.)
Since $P$ and $Q$ are properly supported
we can find $Z_1, Z_2\subset X$ such that $Pv=0$ in $Y$ if $v=0$ in $Z_1$,
and $Qv=0$ in $Y$ if $v=0$ in $Z_2$. We may of course assume that
$Y\Subset Z_j$, $j=1,2$. Fix $\phi_j\in C_0^\infty(X)$ with $\phi_j=1$ on $Z_j$.
Then we have $Pu=P(\phi_1 u)$ and $Qf=Q(\phi_2 f)$ in $Y$, so
\[
\emptyset=W\! F(Qf-Pu)\cap K=W\! F(Q(\phi_2f)-P(\phi_1u))\cap K
\]
where $\phi_1 u$ and $\phi_2f$ have compact support. Hence we may assume that $u\in\mathscr{E}'(X)$
and $f\in H_{(N)}^{\mathrm{comp}}(X)=H_{(N)}^{\mathrm{loc}}(X)\cap\mathscr{E}'(X)$
to begin with. Note that this also implies
$g=Qf-Pu\in\mathscr{E}'(X)$ since $P$ and $Q$ are properly supported.
\end{rmk}

The following easy example will prove useful when discussing inclusion relations
between the ranges of solvable but non-elliptic operators.
\begin{ex}\label{ex:d1solvable}
If $X\subset\mathbb{R}^n$ is open, and $K\subset T^\ast (X)\smallsetminus 0$ is a compactly based cone,
then the range of $D_1=-i\partial / \partial x_1$ is microlocally contained in
the range of $D_2$ at $K$. In fact, this is trivially true since both operators are surjective
$\mathscr{D}'(X)\rightarrow\mathscr{D}'(X)/C^\infty (X)$. To see that for example $D_1$
is surjective we note that
by the remark on page \pageref{rmk:solvabilityineprime}
it suffices to show that
there exists a number $N\in\mathbb{Z}$ such that the equation $D_1 u =f$
has a solution $u\in \mathscr{D}'(X)$ for every
$f\in H_{(N)}^{\mathrm{comp}}(X)=H_{(N)}^{\mathrm{loc}}(X)\cap\mathscr{E}'(X)$.
By ~\cite[Theorem $10.3.1$]{ho2}
this is satisfied for every $N\in\mathbb{Z}$ if $u\in H_{(N+1)}^{\mathrm{loc}}(X)$
is given by $E\ast f$ where $E$ is the regular fundamental solution of $D_1$.
\end{ex}

Just as the microlocal solvability of a pseudo-differential operator $P$
gives an a priori estimate for the adjoint
$P^\ast$, we have the following result for operators satisfying Definition \ref{defrange}.
\begin{lem}\label{lemrange1}Let $K \subset T^{\ast}(X) \smallsetminus 0$ be a compactly based cone. Let
$Q\in \varPsi_{\mathrm{cl}}^m(X)$ and $P\in \varPsi_{\mathrm{cl}}^k(X)$
be properly supported pseudo-differential operators such that
the range of $Q$ is microlocally contained in the range of $P$ at $K$. If $Y\Subset X$ satisfies $K\subset
T^*(Y)$ and if $N$ is the integer in Definition \ref{defrange}, then for every positive integer $\kappa$
we can find a constant $C$, a positive integer $\nu$ and a properly
supported pseudo-differential operator $A$ with $W\! F(A)\cap K= \emptyset$ such that
\begin{equation}\label{rangeeq1}
\|Q^*v\|_{(-N)}\leq C(\|P^*v\|_{(\nu)}+\|v\|_{(-N-\kappa-n)}+\|Av\|_{(0)})
\end{equation}for all $v\in C_0^\infty(Y)$.
\end{lem}
\noindent Since \eqref{rangeeq1} holds for any $\kappa$,
it is actually superfluous to include the dimension $n$ in the norm
$\|v\|_{(-N-\kappa-n)}$. However, for our purposes, it turns out that this is the most convenient formulation.
\begin{proof}We shall essentially adapt the proof of Lemma $26.4.5$ in ~\cite{ho4}.
Let $\| \phantom{i} \|_{(s)}$ denote a norm in $H_{(s)}^{\mathrm{comp}}(X)$
which defines the topology in $H_{(s)}^c(M)=H_{(s)}^{\mathrm{loc}}(X)\cap\mathscr{E}'(M)$
for every compact set $M\subset X$. (The reason we change notation from $H_{(s)}^{\mathrm{comp}}(M)$
to $H_{(s)}^c(M)$ when $M$ is compact
is to signify that $H_{(s)}^c(M)$ is a Hilbert space for each fixed compact set $M$.)
Let $Y\Subset Z \Subset X$,
and take $\chi\in C_0^\infty(X)$ with
$\supp\chi=\overline{Z}$ to be a real valued cutoff function identically equal to $1$ in a neighborhood of $Y$.
Then $\chi Qf\in H_{(N-m)}^c(\overline{Z})$ for all $f\in H_{(N)}^{\mathrm{comp}}(X)$
since $Q$ is properly supported, and we claim that for fixed $f\in H_{(N)}^{\mathrm{comp}}(X)$
we have for some $C$, $\nu$ and $A$ as in the statement of the lemma
\begin{equation}\label{rangeeq2}
|(\chi Q f ,v )|
\leq C(\|P^*v\|_{(\nu)}+\|v\|_{(-N-\kappa-n)}+\|Av\|_{(0)})
\end{equation}
for all $v\in C_0^\infty(Y)$. Indeed, by hypothesis and the remark
on page \pageref{rmk:solvabilityineprime} we can find $u$ and $\tilde{g}$ in $\mathscr{E}'(X)$
with $W\! F(\tilde{g})\cap K=\emptyset$ such that
\[\chi Qf=Qf-(1-\chi)Qf=Pu+\tilde{g}-(1-\chi)Qf.
\]
Since $K\subset T^*(Y)$ and $\chi \equiv 1$ near $Y$ we get
$W\! F((1-\chi)Qf)\cap K=\emptyset$, so $\chi Q f= Pu+g$ for
some $g\in\mathscr{E}'(X)$ with $W\! F(g)\cap K=\emptyset$.
Thus
\[( \chi Q f , v )=(u, P^*v)+(g,v), \quad v\in C_0^\infty(Y).
\]
Now choose properly supported pseudo-differential operators $B_1$ and $B_2$ of order $0$ with $I=B_1+B_2$ and
$W\! F(B_1)\cap W\! F(g)=\emptyset$, $W\! F(B_2)\cap K=\emptyset$ which is possible since $W\! F(g)\cap K=\emptyset$.
Since $g\in \mathscr{E}'(X)$ and $B_1 : \mathscr{E}'(X) \to \mathscr{E}'(X)$ is continuous and microlocal
we get $B_1 g\in C_0^\infty(X)$ so $(B_1 g,v)$ can be estimated by $C\|v\|_{(-N-\kappa-n)}$.
Also, $g\in H_{(-\mu)}^{\mathrm{loc}}(X)$
for some $\mu>0$ so if $B$ is properly supported and elliptic of order $\mu$,
and $B'\in \varPsi_{\mathrm{cl}}^{-\mu}(X)$ is a properly
supported parametrix
of $B$ then
\begin{equation}\label{rangeeq3}
B_2^*v=B'BB_2^*v+LB_2^*v,
\end{equation}
where $L\in \varPsi^{-\infty}(X)$ and both $B'$ and $L$
are continuous $H_{(s)}^{\mathrm{comp}}(X)\to H_{(s+\mu)}^{\mathrm{comp}}(X)$. Hence
\[|(B_2 g , v)|\leq C \|B_2^* v\|_{(\mu)}\leq C(\|BB_2^* v\|_{(0)}+\|B_2^*v\|_{(0)}),
\]
and if we apply the identity \eqref{rangeeq3}
to $\|B_2^*v\|_{(0)}, \|B_2^*v\|_{(-\mu)}, \ldots$ sufficiently many times,
and then recall that $B_2^*$ is properly supported and of order $0$, we obtain
\[|(B_2 g , v)|\leq C(\|BB_2^* v\|_{(0)}+\|v\|_{(-N-\kappa-n)}).
\] Since we chose $B$ to be properly supported this gives \eqref{rangeeq2} with $A=BB_2^*$.

For fixed $\kappa$,
let $V$ be the space $C_0^\infty(Y)$ equipped with the topology defined by the semi-norms $\|v\|_{(-N-\kappa-n)}$,
$\|P^*v\|_{(\nu)}$, $\nu=1,2,\ldots$, and $\|Av\|_{(0)}$ where $A$ is a properly supported
pseudo-differential operator with $K\cap W\! F(A)=\emptyset$. It suffices to use a countable sequence $A_1, A_2, \ldots$
where $A_\nu$ is noncharacteristic of order $\nu$ in a set which increases to $(T^*(X)\smallsetminus 0)\smallsetminus K$
as $\nu \to \infty$. Thus $V$ is a metrizable space. The sesquilinear form $(\chi Q f , v)$ in the product of the
Hilbert space $H_{(N-m)}^c(\overline{Z})$ and the metrizable space $V$ is obviously continuous in $\chi Q f$ for fixed
$v$, and by \eqref{rangeeq2} it is also continuous in $v$ for fixed $f$. Hence it is continuous, which means that for
some $\nu$ and $C$
\[|(\chi Q f ,v )|
\leq C\|Q f\|_{(N-m)}(\|P^*v\|_{(\nu)}+\|v\|_{(-N-\kappa-n)}+\|Av\|_{(0)})
\]
for all $f\in H_{(N)}^{\mathrm{comp}}(X)$ and $v\in C_0^\infty(Y)$.
Now $Q$ is continuous from $H_{(N)}^{\mathrm{comp}}(X)$ to $H_{(N-m)}^{\mathrm{comp}}(X)$
so $\|Q f\|_{(N-m)}\le C\|f\|_{(N)}$. Since
$\chi \equiv 1$ near $Y$ and $(\chi Q)^*=Q^* \chi$
this yields
the estimate
\begin{equation}\label{rangeeq4}
|(f , Q^* v )|
\leq C\| f\|_{(N)}(\|P^*v\|_{(\nu)}+\|v\|_{(-N-\kappa-n)}+\|Av\|_{(0)}).
\end{equation}
For $v\in C_0^\infty(Y)$ and $Q^*$ properly supported we have $Q^*v\in C_0^\infty(X)$, and therefore also
$Q^*v\in H_{(-N)}^{\mathrm{loc}}(X)$. Viewing $Q^*v$ as a functional
on $H_{(N)}^{\mathrm{comp}}(X)$, the dual of $H_{(-N)}^{\mathrm{loc}}(X)$
with respect to the standard inner product on $L^2$,
we obtain
\eqref{rangeeq1} after taking the supremum over all $f\in H_{(N)}^{\mathrm{comp}}(X)$ with $\|f\|_{(N)}=1$.
\end{proof}

We will need the following analogue of ~\cite[Proposition $26.4.4$]{ho4}.
Recall that $\mathcal{H}: T^\ast(Y)\smallsetminus 0\rightarrow T^\ast(X)\smallsetminus 0$
is a canonical transformation if and only if
its graph $C_{\mathcal{H}}$
in the product
$(T^\ast(X)\smallsetminus 0)\times (T^\ast(Y)\smallsetminus 0)$
is Lagrangian
with respect to the difference $\sigma_X-\sigma_Y$ of the
symplectic forms of $T^\ast(X)$ and $T^\ast(Y)$ lifted to
$T^\ast(X)\times T^\ast(Y)=T^\ast(X\times Y)$. This differs in sign
from the symplectic form $\sigma_X+\sigma_Y$ of $T^\ast(X\times Y)$
so it is the \emph{twisted graph}
\[
C_{\mathcal{H}}'=\{(x,\xi,y,-\eta): (x,\xi,y,\eta)\in C_{\mathcal{H}}\}
\]
which is Lagrangian with respect to the standard symplectic structure in
$T^\ast(X\times Y)$.
\begin{prop}\label{prop.26.4.4} Let $K\subset T^\ast(X)\smallsetminus 0$ and
$K'\subset T^\ast(Y)\smallsetminus 0$ be compactly based cones and let $\chi$ be a
homogeneous symplectomorphism from a conic neighborhood of $K'$ to one of $K$ such that
$\chi(K')=K$. Let $A\in I^{m'}(X\times Y, \varGamma')$ and $B\in I^{m''}(Y\times X,(\varGamma^{-1})')$
where $\varGamma$ is the graph of $\chi$, and assume that $A$ and $B$ are properly supported and
non-characteristic at the restriction of the graphs of $\chi$ and $\chi^{-1}$ to $K'$ and to $K$ respectively,
while $W\! F'(A)$ and $W\! F'(B)$ are contained in small conic neighborhoods. Then the range of the pseudo-differential
operator $Q$ in $X$ is microlocally contained in the range of the pseudo-differential
operator $P$ in $X$ at $K$ if and only if the range of the pseudo-differential
operator $BQA$ in $Y$ is microlocally contained in the range of the pseudo-differential
operator $BPA$ in $Y$ at $K'$.
\end{prop}
\begin{proof}
Choose $A_1\in I^{-m''}(X\times Y, \varGamma')$ and $B_1\in I^{-m'}(Y\times X, (\varGamma^{-1})')$
properly supported such that
\begin{align*}
K'\cap W\! F(BA_1-I)& = \emptyset,  & K  \cap W\! F(A_1B-I) = \emptyset, \\
K'\cap W\! F(B_1A-I)& = \emptyset,  & K  \cap W\! F(AB_1-I) = \emptyset.
\end{align*}
Assume that the range of $Q$ is microlocally contained in the range of $P$ at $K$ and
choose $N$ as in Definition \ref{defrange}. Let $g\in H_{(N+m')}^{\mathrm{loc}}(Y)$
and set $f=Ag\in H_{(N)}^{\mathrm{loc}}(X)$. Then we can find $u\in \mathscr{D}'(X)$
such that $K\cap W\! F(Pu-Qf)=\emptyset$. Let $v=B_1u\in \mathscr{D}'(Y)$. Then
\[ W\! F(Av-u)=W\! F( (AB_1-I)u)
\]
does not meet $K$, so $K\cap W\! F(PAv-Qf)=\emptyset$. Recalling that $f=Ag$ this implies
\[K'\cap W\! F(BPAv-BQAg)=\emptyset,
\]
so the range of $BQA$ is microlocally contained in the range of $BPA$ at $K'$. Conversely,
if the range of $BQA$ is microlocally contained in the range of $BPA$ at $K'$ it follows
that the range of $A_1BQAB_1$ is microlocally contained in the range of $A_1BPAB_1$ at $K$.
Since
\[K\cap W\! F(A_1BPAB_1u-A_1BQAB_1f)=K\cap W\! F(Pu-Qf)
\]
this means that the range of $Q$ is microlocally contained in the range of $P$ at $K$, which proves the
proposition.
\end{proof}

Before we can state our main theorem, we need to 
study the geometric situation that occurs when $p$ fails to satisfy condition $(\varPsi)$.
Recall that by ~\cite[Theorem $26.4.12$]{ho4}
we may always assume that the nonvanishing factor in condition \eqref{intropsi}
is a homogeneous function.
We begin with a lemma concerning a reduction of the general case.
\begin{lem}\label{lemexiststrongly}
Let $p$ and $q$ be homogeneous smooth functions on $T^\ast(X)\smallsetminus 0$, and let
$t\mapsto\gamma(t)$, $a\le t\le b$,
be a bicharacteristic interval
of $\re qp$ such that $q(\gamma(t))\ne 0$ for $a\le t\le b$. If
\begin{equation}\label{notintropsi}
\im qp(\gamma(a))<0<\im qp(\gamma(b)),
\end{equation}
then there exists a proper subinterval $[a',b']\subset [a,b]$,
possibly reduced to a point, such that
\begin{itemize}
\item[i)] $\im qp (\gamma(t))= 0$ for $a'\le t\le b'$,
\item[ii)] for every $\varepsilon > 0$
there exist $a' - {\varepsilon} < s_- < a'$ and $b' < s_+ < b' + {\varepsilon}$ such that
$\im qp(\gamma(s_-)) < 0 < \im qp(\gamma(s_+))$.
\end{itemize}
\end{lem}
\noindent If $\gamma(t)$ is defined for $a\le t\le b$ we shall in the sequel say that $\im qp$
changes sign from $-$ to $+$ on $\gamma$ if \eqref{notintropsi} holds.
If $\gamma |_{[a',b']}$ is the restriction of $\gamma$ to $[a',b']$
and i) and ii) hold we shall say that $\im qp$ \emph{strongly}
changes sign from $-$ to $+$ on $\gamma |_{[a',b']}$.
\begin{proof}
It suffices to regard
the case $q=1$, $X=\mathbb{R}^n$, $p$ homogeneous of degree $1$ with $\re p=\xi_1$,
and the bicharacteristic of $\re p$ given by
\begin{equation}\label{eq:reducedbichars}
a\leq x_1 \leq b, \quad x'=(x_2, \ldots, x_n)=0, \quad \xi=\varepsilon_n.
\end{equation}
Here $\varepsilon_n=(0,\ldots,0,1)\in \mathbb{R}^n$,
and we shall
in what follows write $\xi^0$ in place of $\varepsilon_n'$.
The proof of this fact is taken from ~\cite[p. $97$]{ho4}
and is given here for the purpose of reference later, in particular
in connection with Definition \ref{dfn:minimal1} below.

Choose a pseudo-differential operator $Q$ with principal symbol $q$. If we let
$P_1=QP$, then the principal symbol of $P_1$ is $p_1=qp$ so $\im p_1$
changes sign from $-$ to $+$ on the bicharacteristic $\gamma$ of $\re p_1$. Now choose
$Q_1$ to be of order $1-\mathrm{degree}$ $P_1$ with positive, homogeneous principal symbol.
If $p_2$ is the principal symbol of $P_2=Q_1P_1$, it follows that $\re p_1$
and $\re p_2$ have the same bicharacteristics, including orientation,
and since $p_2$ is homogeneous of degree $1$ these can be considered to
be curves on the cosphere bundle $S^\ast(X)$. Moreover, $\im p_1$ and $\im p_2$ have the same
sign, so $\im p_2$ changes sign from $-$ to $+$ along $\gamma\subset S^\ast(X)$. If 
$\gamma$ is a closed curve on $S^\ast(X)$ we can pick an arc that is not closed where
the sign change still occurs. If we assume this to be done,
then ~\cite[Proposition $26.1.6$]{ho4} states that there exists a $C^\infty$
homogeneous canonical transformation $\chi$ from an open conic neighborhood of \eqref{eq:reducedbichars} to
one of $\gamma$ such that $\chi(x_1,0,\varepsilon_n)=\gamma(x_1)$ and $\chi^\ast(\re p_2)=\xi_1$.
Since the Hamilton field is symplectically invariant
it follows that the equations of a bicharacteristic are invariant
under the action of canonical transformations,
that is, $\tilde{\gamma}$ is a bicharacteristic of $\chi^\ast(\re p_2)$ if and only if $\chi(\tilde{\gamma})$
is a bicharacteristic of $\re p_2$. This proves the claim.

In accordance with the notation in ~\cite[p. $97$]{ho4}, let $(x',\xi')=(0,\xi^0)$ and consider
\begin{equation*}
L(0,\xi^0)=\inf\{t-s: a<s<t<b,\, \im p(s,0,\varepsilon_n)<0<\im p(t,0,\varepsilon_n)\}.
\end{equation*}
For every small $\delta>0$ there exist $s_\delta$ and $t_\delta$ such that
$a<s_\delta<t_\delta<b$, $\im p(s_\delta,0,\varepsilon_n)<0<\im p(t_\delta,0,\varepsilon_n)$
and $t_\delta-s_\delta<L(0,\xi^0)+\delta$. Choose a sequence $\delta_j\to 0$
such that the limits $a'=\lim s_{\delta_j}$ and $b'=\lim t_{\delta_j}$ exist.
Then $b'-a'=L(0,\xi^0)$ and in view of \eqref{notintropsi}
we have $a<a'\le b'<b$ by continuity. Moreover, $\im p(t,0,\varepsilon_n) = 0$
for $a'\le t\le b'$. This is clear if $a'=b'$. If on the other hand
$\im p(t,0,\varepsilon_n)$ is, say, strictly positive for some $a'<t<b'$, then
$L(0,\xi^0)\le t-s_{\delta_j}\rightarrow t-a'<b'-a'$, a contradiction.
Thus i) holds.

To prove ii), let $\varepsilon>0$.
After possibly reducing to a subsequence we may assume that
the sequences
$\{s_{\delta_j}\}$ and $\{t_{\delta_j}\}$ given above are monotone increasing
and decreasing, respectively. It then follows by i) that $s_{\delta_j}<a'\le b'<t_{\delta_j}$
for all $j$.
Since $s_{\delta_j}\to a'$ and $t_{\delta_j}\to b'$ we can choose $j$ so that
$a' - {\varepsilon} < s_{\delta_j} < a'$ and $b' < t_{\delta_j} < b' + {\varepsilon}$.
By construction we have $\im p(s_{\delta_j},0,\varepsilon_n)<0<\im p(t_{\delta_j},0,\varepsilon_n)$.
This completes the proof. 
\end{proof}

Although it will not be needed here, we note that if $[a',b']$ is the interval given by Lemma \ref{lemexiststrongly}
and $a'<b'$, then in addition to i) and ii) we also have
\begin{itemize}
\item[iii)] there exists a $\delta > 0$
such that $\im qp(\gamma(s)) \le 0 \le \im qp(\gamma(t))$
for all 
$a' - {\delta} < s < a'$ and $b' < t < b' + {\delta}$.
\end{itemize}
Indeed, the infimum $L(0,\xi^0)=b'-a'$ would otherwise satisfy
$L(0,\xi^0)<\delta$ for every $\delta$ in view of ii), which is a contradiction
when $a'<b'$.

We next recall the definition of a one dimensional bicharacteristic.
\begin{dfn}\label{def1dim}
A one dimensional bicharacteristic of the pseudo-differential operator with homogeneous principal symbol
$p$ is a $C^1$ map $\gamma: I \rightarrow T^\ast(X)\smallsetminus 0$ where $I$ is an interval on
$\mathbb{R}$, such that
\begin{itemize}
\item[(i)] $p(\gamma(t))=0$, $t\in I$,
\item[(ii)] $0\neq \gamma'(t)=c(t)H_p(\gamma(t)) \text{ if } t\in I$
\end{itemize}
for some continuous function $c: I\rightarrow\mathbb{C}$.
\end{dfn}

Let $P$ be an operator of principal type on a $C^\infty$ manifold $X$
with principal symbol $p$, and suppose $p$ fails to satisfy
condition $(\varPsi)$ in $X$. By \eqref{intropsi} there is a function $q$ in $C^\infty(T^\ast(X)\smallsetminus 0)$
such that $\im qp$ changes sign from $-$ to $+$ on a bicharacteristic $\gamma$ of $\re qp$ where $q\neq 0$.
As can be seen in ~\cite[pp. $96-97$]{ho4},
we can then find a compact one dimensional bicharacteristic interval $\varGamma \subset \gamma$ or a characteristic
point $\varGamma \in \gamma$ such that the sign change occurs on bicharacteristics of $\re qp$ arbitrarily
close to $\varGamma$.
What we mean by this
will be clear from the following discussion, although we will not use this terminology in
the sequel. By the proof of Lemma \ref{lemexiststrongly} it suffices to regard
the case $q=1$, $X=\mathbb{R}^n$, $p$ homogeneous of degree $1$ with $\re p=\xi_1$,
and the bicharacteristic of $\re p$ given by \eqref{eq:reducedbichars}.

We shall now study a slightly more general situation is some detail.
If $\gamma = I \times \{w_0\}$, $I= [a,b]$, we shall by $|\gamma|$
denote the usual arc length in $\mathbb{R}^{2n}$, so that $|\gamma|=b-a$.
Furthermore, we will assume that all curves are bicharacteristics of $\re p=\xi_1$, that
is, $w_0=(x',0,\xi')\in\mathbb{R}^{2n-1}$.
We owe parts of this exposition to Nils Dencker ~\cite{de1}.

\begin{lem}\label{lem:minimal01}
Assume that  $\im p$  strongly changes sign from $-$ to $+$ on 
$\gamma =[a,b] \times \{w_0\}$. Then
for any $\delta > 0$ there exist $\varepsilon > 0$, $a
-\delta < s_- < a $ and $b < s_+ < b + \delta$ so that
$\pm \im p(s_\pm,w) > 0$ for any  $|w-w_0| < \varepsilon$.
\end{lem}
\begin{proof}
Since $t \mapsto \im p(t, w_0)$ strongly changes sign on $[a,b]$ we can find
$s_\pm$ satisfying the conditions so that  $\pm \im p(s_\pm,w_0) > 0$. By
continuity we can find $\varepsilon_\pm > 0 $ so that $\pm \im p(s_\pm,w)
> 0$ for any  $|w-w_0| < \varepsilon_\pm$. The lemma now follows if we take $\varepsilon =
\min(\varepsilon_-, \varepsilon_+)$.
\end{proof}

We shall employ the following notation.

\begin{dfn}
Let $\gamma = [a,b] \times \{w_0\}$, and let ${\gamma}_j = [a_j, b_j] \times \{w_j\}$. If
$\liminf_{j \to \infty} a_j \ge a$, $\limsup_{j \to \infty} b_j \le b$ and $\lim_{j\to\infty}
w_j = w_0$, then we shall write ${\gamma}_j \dashrightarrow {\gamma}$ as $j\to\infty$.
If in addition $\lim_{j \to \infty} a_j=a$ and $\lim_{j \to \infty} b_j=b$ then we
shall write $\gamma_j\to\gamma$ as $j\to\infty$.
\end{dfn}
\begin{dfn}\label{def:minimal02}
If $\gamma$ is a bicharacteristic of $\re p=\xi_1$ and there exists a sequence
$\{\gamma_j\}$ of bicharacteristics of $\re p$ such that $\im p$
strongly changes sign from $-$ to $+$ on $\gamma_j$
for all $j$
and $\gamma_j\dashrightarrow\gamma$ as $j\to\infty$, we set
\begin{equation}\label{newdefoflp}
L_p({\gamma}) = \inf_{\{\gamma_j\}}\{\liminf_{j\to\infty} |\gamma_j| : \gamma_j \dashrightarrow \gamma \,
\text{ as }j\to\infty\},
\end{equation}
where the infimum is taken over all such sequences.
We shall write $L_p({\gamma})\ge 0$ to signify the existence of such a sequence $\{\gamma_j\}$.
\end{dfn}
\begin{rmk}\label{connectionntoL0}
The definition of $L_p(\gamma)$ corresponds to what is denoted by $L_0$ in ~\cite[p. $97$]{ho4},
when $\gamma = [a,b] \times \{w_0\}$ is given by \eqref{eq:reducedbichars} and
\begin{equation}\label{pchangesongamma}
\im p(a,w_0)<0<\im p(b,w_0).
\end{equation}
To prove this claim, we begin by showing that $L_p(\gamma)\le L_0$,
after having properly defined $L_0$. 
To this end,
let $\tilde{\gamma}=[\tilde{a},\tilde{b}]\times\{\tilde{w}\}$ be a bicharacteristic
of $\re p$ such that $\im p$ changes sign on $\tilde{\gamma}$.
For $w$ close to $w_0$ we set
\[
\mathcal{L}_p(\tilde{\gamma},w)=\inf\{t-s: \tilde{a}<s<t<\tilde{b},\ \im p(\tilde{a},w)<0<\im p(\tilde{b},w)\}.
\]
(Using
the notation in ~\cite[p. $97$]{ho4} we would have $\mathcal{L}_p(\gamma,w)=L(x',\xi')$
if $w=(x',0,\xi')$.) Then
\[
L_0=\liminf_{w\to w_0}\mathcal{L}_p(\gamma,w).
\]
By an adaptation of the arguments in ~\cite[p. $97$]{ho4}
it follows from the definition of $L_0$
that we can find 
a sequence $\{\gamma_j\}$ of bicharacteristics of $\re p$ with
$\gamma_j=[a_j,b_j]\times \{w_j\}$ such that
\[
\im p (a_j,w_j)<0<\im p (b_j,w_j)\quad\text{for all }j,
\] 
where $\lim w_j=w_0$ and the limits $a_0=\lim a_j$ and $b_0=\lim b_j$
exist, belong to the interval $(a,b)$
and satisfy $b_0-a_0=L_0$. 
If we for each $j$ apply Lemma \ref{lemexiststrongly} to
$\gamma_j$ we obtain a sequence of bicharacteristics $\varGamma_j\subset\gamma_j$
of $\re p$ such that
$\im p$ strongly changes sign from $-$ to $+$ on $\varGamma_j$,
where $|\varGamma_j|=\mathcal{L}_p(\gamma_j,w_j)<|\gamma_j|$.
Clearly $\varGamma_j\dashrightarrow\gamma$ as $j\to\infty$.
Since $a<a_j\le b_j<b$ if $j$ is sufficiently large it
follows that for such $j$ we have $\mathcal{L}_p(\gamma,w_j)\le\mathcal{L}_p(\gamma_j,w_j)$
by definition.
This implies
\begin{equation}\label{existenceofGammaj}
\begin{aligned}
L_0&=\liminf_{w\to w_0}\mathcal{L}_p(\gamma,w)\le\liminf_{j\to\infty}\mathcal{L}_p(\gamma,w_j)\\
&\le\liminf_{j\to\infty}|\varGamma_j|
\le\limsup_{j\to\infty}|\varGamma_j|\le\lim_{j\to\infty}|\gamma_j|=L_0,
\end{aligned}
\end{equation}
so $|\varGamma_j|\to L_0$ as $j\to\infty$. Thus $L_p(\gamma)\le L_0$.

For the reversed inequality, suppose
$\{\tilde{\gamma}_j\}$
is any sequence satisfying the properties of Definition \ref{def:minimal02},
with $\tilde{\gamma}_j=[\tilde{a}_j,\tilde{b}_j]\times\{\tilde{w}_j\}$.
By assumption we have
$\im p(\tilde{a}_j,\tilde{w}_j)=\im p(\tilde{b}_j,\tilde{w}_j)=0$ for all $j$,
which together with \eqref{pchangesongamma} and a
continuity argument implies the existence of a positive integer $j_0$ such that
\begin{equation*}
a<\tilde{a}_j\le \tilde{b}_j<b\quad\text{for all }j\ge j_0.
\end{equation*}
If $\tilde{\gamma}_{j,\delta}=[\tilde{a}_j-\delta,\tilde{b}_j+\delta]\times\{\tilde{w}_j\}$,
this means that
for small $\delta>0$ and sufficiently large $j$
we have
\[
\mathcal{L}_p(\gamma,\tilde{w}_j)\le\mathcal{L}_p(\tilde{\gamma}_{j,\delta},\tilde{w}_j).
\]
Since $\im p$ strongly changes sign from $-$ to $+$ on $\tilde{\gamma}_j$,
the infimum in the right-hand side exists for every $\delta>0$, and
is bounded from above by $\tilde{b}_j-\tilde{a}_j+2\delta$.
Taking the limit as $\delta\to 0$ yields $\mathcal{L}_p(\gamma,\tilde{w}_j)
\le |\tilde{\gamma}_j|$. Since $\tilde{w}_j\to w_0$ as $j\to\infty$ the definition
of $L_0$ now gives
\begin{equation}\label{boundonlp}
L_0\le\liminf_{j\to\infty}\mathcal{L}_p(\gamma,\tilde{w}_j)\le\liminf_{j\to\infty}
|\tilde{\gamma}_j|,
\end{equation}
and since the sequence $\{\tilde{\gamma}_j\}$ was arbitrary, we obtain
$L_0\le L_p(\gamma)$ by Definition \ref{def:minimal02}.
This proves the claim.
\end{rmk}

When no confusion can occur we will
omit the dependence on $p$ in Definition \ref{def:minimal02}. 
We note that
if $L_p(\gamma)$ exists, then $L_p(\gamma)\le |\gamma|$ by definition. Also,
if $\im p$ strongly changes sign from $-$ to $+$ 
on ${\gamma}$ then Lemma \ref{lem:minimal01} implies that the conditions of
Definition \ref{def:minimal02} are satisfied.
This proves the first part of the following result.
\begin{cor}\label{cor:minimal03}
Let $\gamma=[a,b]\times\{w_0\}$ be a bicharacteristic of $\re p=\xi_1$.
If $\im p$ strongly changes sign from $-$ to $+$ on
${\gamma}$
then $0\le L_p({\gamma}) \le
|\gamma|$. Moreover, for every $\delta,\varepsilon>0$ there exists a bicharacteristic
$\tilde{\gamma}=\tilde{\gamma}_{\delta,\varepsilon}$ of $\re p$ with
\[
\tilde{\gamma}=[\tilde{a},\tilde{b}]\times\{\tilde{w}\},
\quad a-\varepsilon<\tilde{a}\le\tilde{b} < b+\varepsilon,
\quad |\tilde{w}-w_0|<\varepsilon,
\]
such that $\im p$ strongly changes sign from $-$ to $+$ on $\tilde{\gamma}$
and $|\tilde{\gamma}|<L_p(\gamma)+\delta$.
\end{cor}
\begin{proof}
The existence of the sequence $\{\varGamma_j\}$ in the preceding remark can 
after some adjustments be used
to prove the second part of Corollary \ref{cor:minimal03}, but we prefer the following direct proof.

Given $\delta>0$ we can by Definition \ref{def:minimal02}
find a sequence $\gamma_j=[a_j,b_j]\times\{w_j\}$ of bicharacteristics
of $\re p$ such that $\gamma_j\dashrightarrow\gamma$ as $j\to\infty$,
$\im p$ strongly changes sign from $-$ to $+$
on $\gamma_j$ and $\liminf_{j\to\infty}|\gamma_j|<L(\gamma)+\delta$.
After reducing to a subsequence we may assume
$|\gamma_j|<L(\gamma)+\delta$ for all $j$. We have $\liminf_{j\to\infty} a_j\ge a$ so for
every $\varepsilon$ there exists a $j_1(\varepsilon)$ such that
$a_j>a-\varepsilon$ for all $j\ge j_1$. Similarly
there exists a $j_2(\varepsilon)$ such that $b_j<b+\varepsilon$
for all $j\ge j_2$. Also, $w_j\to w_0$ as $j\to\infty$ so there exists a $j_3(\varepsilon)$
such that $|w_j-w_0|<\varepsilon$ for all $j\ge j_3$. Hence we can take $\tilde{\gamma}=\gamma_{j_0}$
where $j_{0}=\max (j_1,j_2,j_3)$.
\end{proof}

Consider now
the general case when
$\im qp$ changes sign from $-$ to $+$ on a bicharacteristic
$\gamma\subset T^\ast(X)\smallsetminus 0$
of $\re qp$ where $q\neq 0$,
that is, \eqref{notintropsi} holds. In view of
the proof of Lemma \ref{lemexiststrongly} 
we can by means of \eqref{newdefoflp} define a minimality property of a subset of the curve $\gamma$
in the following sense.

\begin{dfn}\label{dfn:minimal1}
Let $I\subset\mathbb{R}$ be a compact interval possibly reduced to a point and
let $\tilde{\gamma}:I\to T^\ast(X)\smallsetminus 0$
be a characteristic point or a compact one dimensional bicharacteristic interval
of the homogeneous function $p\in C^\infty(T^\ast(X)\smallsetminus 0)$.
Suppose that there exists a function $q\in C^\infty(T^\ast(X)\smallsetminus 0)$
and a $C^\infty$
homogeneous canonical transformation $\chi$ from an open conic neighborhood $V$ of
\[
\varGamma=\{(x_1,0,\varepsilon_n): x_1\in I\}
\subset T^\ast (\mathbb{R}^n)
\]
to an open conic neighborhood
$\chi(V)\subset T^\ast(X)\smallsetminus 0$
of $\tilde{\gamma}(I)$ such that
\begin{itemize}
\item[(i)] $\chi(x_1,0,\varepsilon_n)=\tilde{\gamma}(x_1)$ and $\re \chi^\ast(qp )=\xi_1$
in $V$,
\item[(ii)] $L_{\chi^\ast(qp)}(\varGamma)=|\varGamma|$.
\end{itemize}
Then we say that $\tilde{\gamma}(I)$ is a minimal characteristic point
or a minimal bicharacteristic interval
if $|I|=0$ or $|I|>0$,
respectively.
\end{dfn}

The definition of the arclength is of course dependent of the choice of Riemannian
metric on $T^\ast (\mathbb{R}^n)$. However, since we are only using the arclength
to compare curves where one is contained within the other and both are parametrizable through condition (i),
the results here and Definition
\ref{dfn:minimal1} in particular are independent of the chosen metric.
By choosing a Riemannian metric on $T^\ast (X)$, one
could therefore define the minimality property given by Definition
\ref{dfn:minimal1} through the corresponding arclength in $T^\ast (X)$
directly, although there, the notion of convergence of curves
is somewhat trickier. We shall not pursue this any further.

Note that condition (i) implies that
$q\ne 0$ and $\re H_{qp}\ne 0$ on $\tilde{\gamma}$,
and that by definition,
a minimal bicharacteristic interval
is a compact one dimensional bicharacteristic interval.
Moreover,
if $\im qp$ changes sign from $-$ to $+$ on a bicharacteristic
$\gamma\subset T^\ast(X)\smallsetminus 0$
of $\re qp$ where $q\neq 0$, then we can always find a minimal characteristic point $\tilde{\gamma}
\in\gamma$ or a minimal bicharacteristic interval
$\tilde{\gamma}\subset\gamma$.
In view of the proof of Lemma \ref{lemexiststrongly},
this follows from the conclusion of
the extensive remark beginning on page \pageref{connectionntoL0}
together with \eqref{existenceofGammaj}.
The following proposition shows that this continues to hold even when
the assumption \eqref{notintropsi} is relaxed
in the sense of Definition \ref{def:minimal02}.
We will state this result only in the (very weak) generality needed here.

\begin{prop}\label{prop:minimal03}
Let $\gamma=[a,b]\times \{w_0\}$ be a bicharacteristic of $\re p=\xi_1$,
and assume that $L(\gamma)\ge 0$.
Then there exists a minimal characteristic point $\varGamma\in\gamma$ of $p$ or
a minimal bicharacteristic interval $\varGamma\subset\gamma$ of $p$ of length $L(\gamma)$
if $L(\gamma)=0$ or $L(\gamma)>0$, respectively.
If $\varGamma=[a_0,b_0]\times \{w_0\}$ and $a_0<b_0$, that is, $L(\gamma)>0$, then
\begin{equation}\label{eq:minimal04}
\im p_{(\alpha)}^{(\beta)}(t,w_0)=0
\end{equation}
for all $\alpha, \beta$ with
$\beta_1=0$ if
$a_0\le t\le b_0$.
Conversely, if $\gamma$ is a minimal characteristic point or
a minimal bicharacteristic interval
then $L(\gamma)=|\gamma|$.
\end{prop}

For the proof we shall need the following lemma.

\begin{lem}\label{prop:minimal05}
Let $\gamma$ and $\gamma_j$, $j\ge 1$, be bicharacteristics of $\re p=\xi_1$,
and assume that $\im p$ strongly changes sign from $-$ to $+$ on
${\gamma}_j$ for each $j$.
If ${\gamma}_j \dashrightarrow {\gamma}$ as $j\to\infty$ then $L({\gamma}) \le
\liminf_{j \to \infty} L({\gamma}_j)$.
\end{lem}
\begin{proof}Let $\gamma_j=[a_j,b_j]\times \{w_j\}$ and $\gamma=[a,b]\times \{w_0\}$.
Since $\im p$ strongly changes sign from $-$ to $+$ on
${\gamma}_j$ we can by
Corollary \ref{cor:minimal03} for each $j$ find
a bicharacteristic
$\tilde{\gamma}_{j}=[\tilde{a}_{j},\tilde{b}_{j}]\times \{\tilde{w}_{j}\}$
of $\re p$ with
\[
a_j-1/j<\tilde{a}_j\le\tilde{b}_j < b_j+1/j,
\quad |\tilde{w}_j-w_j|<1/j,
\]
such that $\im p$ strongly changes sign from $-$ to $+$ on $\tilde{\gamma}_j$
and $|\tilde{\gamma}_j|<L(\gamma_j)+1/j$.
Now $|\tilde{w}_{j} - w_0| \le |\tilde{w}_{j} - w_j| + |w_j - w_0|$, and since 
$\liminf_{j \to
\infty} \tilde{a}_{j} \ge \liminf_{j \to \infty} (a_{j}-1/j) \ge a$ and
correspondingly for $\tilde{b}_{j}$,
we find
that $\tilde{\gamma}_{j} \dashrightarrow \gamma$ as $j\to\infty$.
Thus
\[
L({\gamma}) \le \liminf_{j \to \infty} |\tilde{\gamma}_{j}| \le
\liminf_{j \to \infty}( L({\gamma}_j) + 1/j)
\]
which completes the proof.
\end{proof}

\begin{proof}[Proof of Proposition \ref{prop:minimal03}]
We may without loss of generality assume that $w_0=(0,\varepsilon_n)\in\mathbb{R}^{2n-1}$.
The last statement is then an immediate consequence of Definition \ref{dfn:minimal1}.
To prove the theorem it then also suffices to show that we can find
a characteristic point $\varGamma\in\gamma$ of $p$, or a compact one dimensional bicharacteristic interval
$\varGamma\subset\gamma$ of $p$ of length $L(\gamma)$, with the property
that in any neighborhood of $\varGamma$ there is a bicharacteristic of $\re p$
where $\im p$ strongly changes sign from $-$ to $+$. 
This is done by adapting the arguments in ~\cite[p. $97$]{ho4},
which also yields \eqref{eq:minimal04}.

For small $\delta>0$ we can find $\varepsilon(\delta)$ with $0<\varepsilon<\delta$ such that
$L(\tilde{\gamma})>L(\gamma)-\delta/2$ for any bicharacteristic $\tilde{\gamma}=[\tilde{a},\tilde{b}]
\times \{\tilde{w}\}$ with $a-\varepsilon<\tilde{a}\le\tilde{b}< b+\varepsilon$
and $|\tilde{w}-w_0|<\varepsilon$
such that $\im p$ strongly changes sign from $-$ to $+$ on $\tilde{\gamma}$.
Indeed, otherwise there would exist a $\delta>0$ such that for each
(sufficiently large) $k$ we can find
a bicharacteristic $\gamma_k=[a_k,b_k]\times \{w_k\}$ with $a-1/k< a_k\le b_k <b+1/k$ and $|w_k-w_0|<1/k$
such that $\im p$ strongly changes sign from $-$ to $+$ on $\gamma_k$ and $L(\gamma_k)\le L(\gamma)-\delta/2$.
This implies that $\gamma_k\dashrightarrow\gamma$ as $k\to\infty$, so by Lemma
\ref{prop:minimal05} we obtain
\[
L(\gamma)\le \liminf_{k\to\infty} L(\gamma_k)\le L(\gamma)-\delta/2,
\]
a contradiction.
Since $L(\gamma)\ge 0$ we have
by Corollary \ref{cor:minimal03}
for some $|w_\delta-w_0|<\varepsilon$ and
$a-\varepsilon< a_\delta\le b_\delta< b+\varepsilon$ with $w_\delta=(x_\delta',0,\xi_\delta')$
that $\im p$ strongly changes sign from $-$ to $+$ on
the bicharacteristic
$\gamma_\delta=[a_\delta,b_\delta]\times \{w_\delta\}$, and
$|\gamma_\delta|<L(\gamma)+\delta/4$.
Thus,
\begin{equation}\label{eq:minimal04.5}
L(\gamma)-\delta /2 < |\gamma_\delta | < L(\gamma)+\delta /4.
\end{equation}
We claim that $\im p$ and
all derivatives with respect to $x'$ and $\xi'$ must vanish at
$(t,w_\delta)$ if $a_\delta+\delta<t<b_\delta-\delta$. Indeed, by Lemma
\ref{lem:minimal01} we can find a $\rho>0$, $a_\delta-\delta/4<s_-<a_\delta$ and
$b_\delta<s_+<b_\delta+\delta/4$ such that
\[
\im p(s_-,w)<0<\im p(s_+,w)\quad\text{for all }|w-w_\delta|<\rho.
\]
If $\im p$ and
all derivatives with respect to $x'$ and $\xi'$ do not vanish at
$(t,w_\delta)$ if $a_\delta+\delta<t<b_\delta-\delta$, then we can choose
$w=(x',0,\xi')$ so that $|w-w_\delta|<\rho$, $|w-w_0|<\varepsilon$
and $\im p(t,w)\neq 0$ for some $a_\delta+\delta<t<b_\delta-\delta$.
It follows that the required sign change
of $\im p(x_1,w)$
must occur
on one of the intervals $(s_-,t)$ and $(t,s_+)$, which are shorter than $L(\gamma)-\delta/2$.
This contradiction proves the claim.

Now choose a sequence $\delta_j\to 0$ as $j\to\infty$ such that
$\lim a_{\delta_j}$ and $\lim b_{\delta_j}$ exist. If we denote these limits
by $a_0$ and $b_0$, respectively, then $L(\gamma)=b_0-a_0$
by \eqref{eq:minimal04.5}, and \eqref{eq:minimal04} holds
if $a_0<b_0$. In particular, if $a_0<b_0$ then
\[
H_{p}(\gamma(t))=(1+i\partial \im p(\gamma(t))/\partial \xi_1)\gamma'(t), \quad a_0\le t \le b_0,
\]
so if $\varGamma=\{(t,w_0): t\in I\}$, $I=[a_0,b_0]$ then $\varGamma$
is a compact one dimensional bicharacteristic interval of $p$ with
the function $c$ in Definition \ref{def1dim} given by
\[
c(t)=(1+i\partial \im p(\varGamma(t))/\partial \xi_1)^{-1}.
\]
This completes the proof.
\end{proof}

Proposition \ref{prop:minimal03} allows us to make some additional comments
on the implications of Definition \ref{dfn:minimal1}.
With the notation in the definition, we note that
condition (ii)
implies that
there exists a sequence $\{\varGamma_j\}$
of bicharacteristics of $\re \chi^\ast (qp)$ on which $\im\chi^\ast (qp)$ strongly changes
sign from $-$ to $+$,
such that $\varGamma_j\to\varGamma$
as $j\to\infty$.
By our choice of terminology, the sequence $\{\varGamma_j\}$
may simply be a sequence of points when $L(\varGamma)=0$. Conversely, if
$\{\varGamma_j\}$
is a point sequence then $L(\varGamma)=0$.
Also note that if $\tilde{\gamma}(I)$ is minimal, and
condition (i) in Definition \ref{dfn:minimal1} is satisfied
for some other choice of maps $q',\chi'$, then condition (ii)
also
holds for $q',\chi'$;
in other words,
\[
L_{\chi^\ast(qp)}(\varGamma)=|\varGamma|=L_{(\chi')^\ast(q'p)}(\varGamma).
\]
This follows
by an application of Proposition \ref{prop:minimal03}
together with ~\cite[Lemma $26.4.10$]{ho4}.
It is then also clear that $\tilde{\gamma}(I)$ is a minimal characteristic point
or a minimal bicharacteristic interval of 
the homogeneous function $p\in C^\infty(T^\ast(X)\smallsetminus 0)$ if and only if
$\varGamma(I)$ is a minimal characteristic point
or a minimal bicharacteristic interval of 
$\chi^\ast(qp)\in C^\infty(T^\ast(\mathbb{R}^n)\smallsetminus 0)$
for any maps $q$ and $\chi$ satisfying condition (i) in Definition \ref{dfn:minimal1}.

The proof of ~\cite[Theorem $26.4.7$]{ho4}
stating that condition $(\varPsi)$ is necessary for local solvability
relies on the imaginary part of the principal symbol
satisfying \eqref{eq:minimal04}. By Proposition \ref{prop:minimal03}, it is clear that \eqref{eq:minimal04}
holds on a minimal bicharacteristic interval $\varGamma$
in the case $q=1$, $\re p=\xi_1$. However, we shall
require the fact that we can find bicharacteristics arbitrarily
close to $\varGamma$ for which the following stronger result is applicable, at least if $\im p$ does not depend on $\xi_1$
as is the case for the standard normal form. This will be made precise below.
\begin{prop}\label{prop:minimal04}
Let $p=\xi_1+i\im p$. Assume that  $\im p$ strongly changes sign  from $-$ to $+$ on $\gamma=[a,b]\times \{w\}$ and that
$L({\gamma}) \ge |\gamma| -{\varrho}$ for some
$0<\varrho<|\gamma|/2$.
If $\im p$ does not depend on $\xi_1$ then
for any ${\kappa} > \varrho$ we find that $\im p$ vanishes 
identically in a neighborhood of $I_{\kappa} \times \{w\}$, where $I_{\kappa} =
[a+ {\kappa}, b - {\kappa}]$.
\end{prop}
\noindent The statement would of course be void if the hypotheses hold only for
$\varrho\ge |\gamma|/2$, for then $I_\kappa=\emptyset$.
\begin{proof}
If the statement is false, there exists a ${\kappa} > 0$ so that $\im p \not\equiv 0$
near  $I_{\kappa} \times \{w\}$. Thus there exists a sequence $(s_j,w_j)
\dashrightarrow I_{\kappa} \times \{w\}$ such that $\im p(s_j,w_j) \ne 0$ for all $j$. Since 
$\im p$ does not depend on $\xi_1$ we can choose $w_j$ to have $\xi_1$ coordinate equal to zero
for all $j$,
so that $(s_j,w_j)$ is contained in a bicharacteristic of $\re p$.
We may choose a
subsequence so that for some $ s \in I_{\kappa}$ we have $|s_j-s| \to 0$ and $|w_j- w| \to 0$
monotonically,  and either $\im p(s_j,w_j) > 0$ or $-\im p(s_j,w_j) > 0$ for all
$j$. We shall consider the case
with positive sign, the negative case works similarly.

Choose $\delta < ({\kappa} - \varrho)/3$ and use Lemma \ref{lem:minimal01}. We find that there
exists $a -\delta < s_- < a $ and $\varepsilon > 0$ such that  
$\im p(s_-,v) < 0$ for any  $|v-w| < \varepsilon$. Choose $k > 0$ so that
$|s_j- s| < \delta$ and $|w_j-w| < \varepsilon$ when $j > k$. Then $t
\mapsto \im p(t,w_j)$
changes sign from $-$ to $+$ on $I_j = [s_-, s_j]$, which has length
\[
|I_j| = s_j - s_- \le |s_j -s| + s -a + a -s_-  < |\gamma| - {\kappa} + 2 \delta
< |\gamma| - \varrho - {\delta}.
\]
If we for each $j$ apply Lemma \ref{lemexiststrongly}
to $I_j\times\{w_j\}$ and let $j \to \infty$ we obtain a contradiction
to the hypothesis $L(\gamma)\ge |\gamma|-\varrho$.
\end{proof}

Note that one could state Proposition \ref{prop:minimal04} without
the condition that the imaginary part is independent of $\xi_1$. The
invariant statement would then be that the restriction of the imaginary part
to the characteristic set of the real part vanishes in a neighborhood of $\gamma$.

The fact that Proposition \ref{prop:minimal04}
assumes that $\im p$ strongly changes sign from $-$ to $+$ on $\gamma$
means that the conditions are not in general satisfied 
when $\gamma$ is a minimal bicharacteristic interval.
As mentioned above, we will instead show that
arbitrarily close to a minimal bicharacteristic interval one can always find
bicharacteristics 
for which Proposition 
\ref{prop:minimal04} is applicable.
Before we state the results we introduce a helpful
definition
together with
some (perhaps contrived but illustrative) examples.
\begin{dfn}\label{dfn:minimal2}
A minimal bicharacteristic
interval $\varGamma=[a_0,b_0]\times \{w_0\}\subset T^\ast(\mathbb{R}^n)\smallsetminus 0$
of the homogeneous function $p=\xi_1+i\im p$ of degree $1$
is said to be $\varrho$-minimal if there exists a
$\varrho\ge 0$ such that $\im p$ vanishes in a neighborhood of
$[a_0+\kappa,b_0-\kappa]\times \{w_0\}$
for any $\kappa>\varrho$.
\end{dfn}
By a $0$-minimal bicharacteristic interval $\varGamma$ we thus mean a minimal bicharacteristic interval such
that the imaginary part vanishes in a neighborhood of any proper closed subset of $\varGamma$. Note that
this does not hold for minimal bicharacteristic intervals in general.
\begin{ex}\label{ex:graph1}
Let $f\in C^\infty(\mathbb{R})$ be given by
\begin{equation}\label{eq:graph1}
f(t)=\left\{
\begin{array}{ll}
-e^{-1/t^2} & \textrm{if $t<0$}\\
0 & \textrm{if $0\le t\le 2$}\\
e^{-1/(t-2)^2} & \textrm{if $t>2$}
\end{array}
\right.
\end{equation}
and let $\phi\in C^\infty(\mathbb{R})$ be a smooth cutoff function with $\supp\phi =[0,2]$ such
that $\phi>0$ on $(0,2)$.
If $\xi=(\xi_1,\xi')$ then
\[
p_1(x,\xi)=\xi_1+i|\xi'|(f(x_1)+x_2\phi(x_1))
\]
is homogeneous of degree $1$.
If we write $x=(x_1,x_2,x'')$ then
for any fixed $(x'',\xi')\in\mathbb{R}^{n-2}\times\mathbb{R}^{n-1}$
with $\xi'\ne 0$
we find that
$\{(x_1,x_2,x'',0,\xi') : x_1 = a , x_2=c\}$ is a minimal characteristic point of $p_1$ if
$c\ge 0$ and $a=0$ or if $c\le 0$ and $a=2$.
Note that if $\xi'\ne 0$ then
$\im p_1$ changes sign from $-$ to $+$ on the bicharacteristic
$\gamma(x_1)=\{(x_1,0,x'',0,\xi')\}$
of $\re p_1$, but that none of the points
$\{\gamma(x_1): 0<x_1<2\}$
are minimal characteristic points.\footnote{If the factor
$x_2$ in $\im p_1$ is raised to the power 3 for example,
then it turns out that $\{\gamma(x_1): 0<x_1<2\}$ is
a one dimensional bicharacteristic interval of $p_1$, and
not only a bicharacteristic of the real part. It is obviously not
minimal though, nor
does it contain any minimal characteristic points.}
On the other hand, if
$f$ is given by \eqref{eq:graph1} let
\begin{equation*}
h(x,\xi')=\left\{
\begin{array}{ll}
|\xi'|f(x_1-1)e^{1/x_2} & \textrm{if $x_2<0$}\\
0 & \textrm{if $x_2=0$}\\
|\xi'|f(x_1)e^{-1/x_2} & \textrm{if $x_2>0$}
\end{array}
\right.
\end{equation*}
be the imaginary part of $p_2(x,\xi)$. If $\re p_2=\xi_1$ then
$p_2$ is homogeneous of degree $1$ and
\[
\varGamma_c=\{(x_1,x_2,x'',0,\xi'): x_2=c,x_1\in I_c\}
\]
is a minimal bicharacteristic interval of $p_2$ for any $(x'',\xi')\in\mathbb{R}^{n-2}\times\mathbb{R}^{n-1}$
with $\xi'\ne 0$ if $c\ge 0$ and $I_c=[0,2]$ or if $c\le 0$ and $I_c=[1,3]$.
Moreover, if $c\lessgtr 0$ then $\varGamma_c$ is a $0$-minimal
bicharacteristic interval. However, there is no $\varrho>0$ such that
the minimal bicharacteristic interval
$\varGamma=\{(x_1,0,x'',0,\xi'): 0\le x_1\le 2\}$
is $\varrho$-minimal. The same holds for the minimal bicharacteristic
interval $\tilde{\varGamma}=\{(x_1,0,x'',0,\xi'): 1\le x_1\le 3\}$.
Figure \ref{fig:nollställe1} shows a cross-section of the characteristic sets of
$\im p_1$ and $\im p_2$.
\end{ex}
\begin{figure}
\psfrag{q1}[][]{\small $x_1$}
\psfrag{q2}[][]{\small $x_2$}
\psfrag{q3}[][]{$-$}
\psfrag{q4}[][]{$+$}
\psfrag{q5}[][]{$0$}
\centering
\includegraphics[width=2.3in]{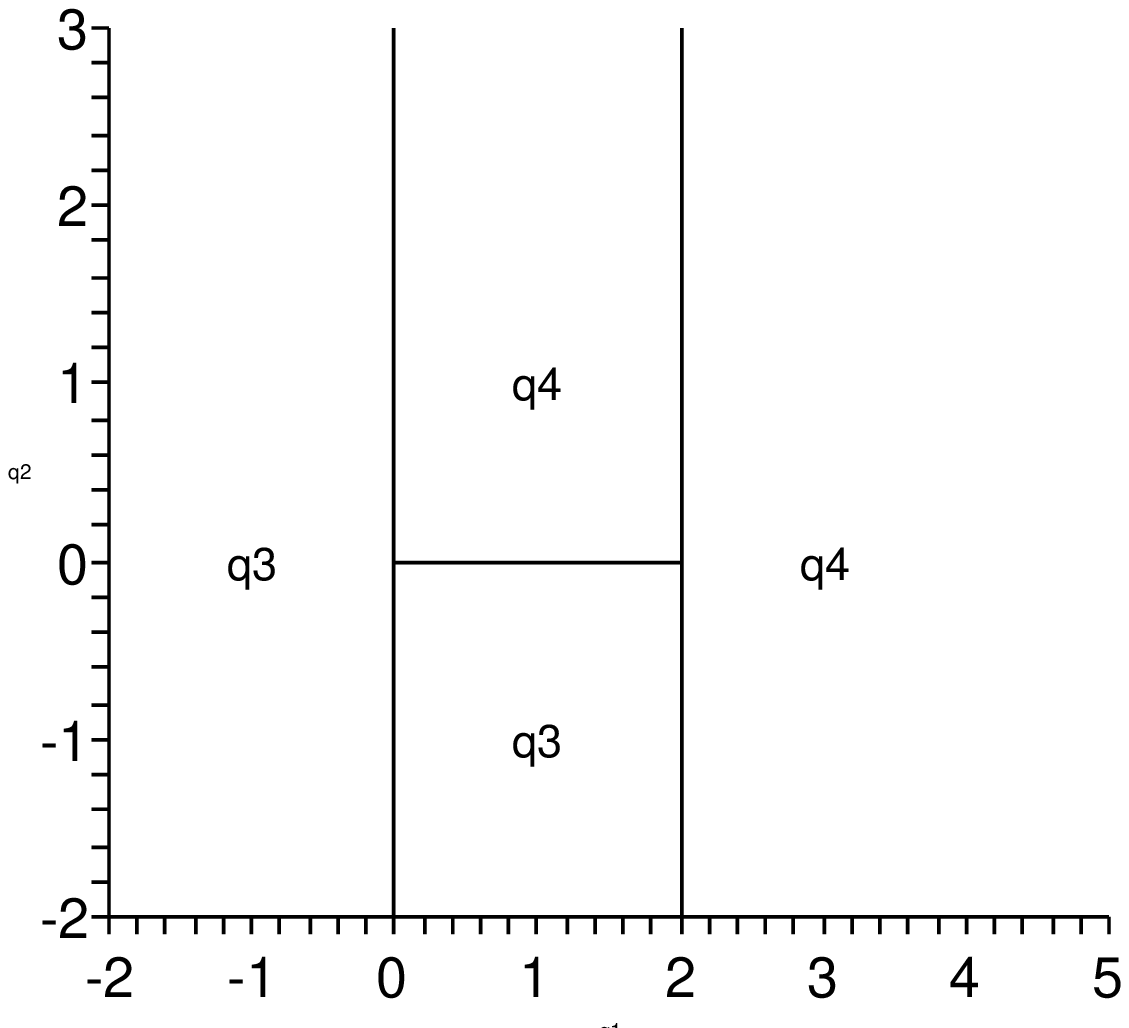}%
\hspace{0.2in}%
\includegraphics[width=2.3in]{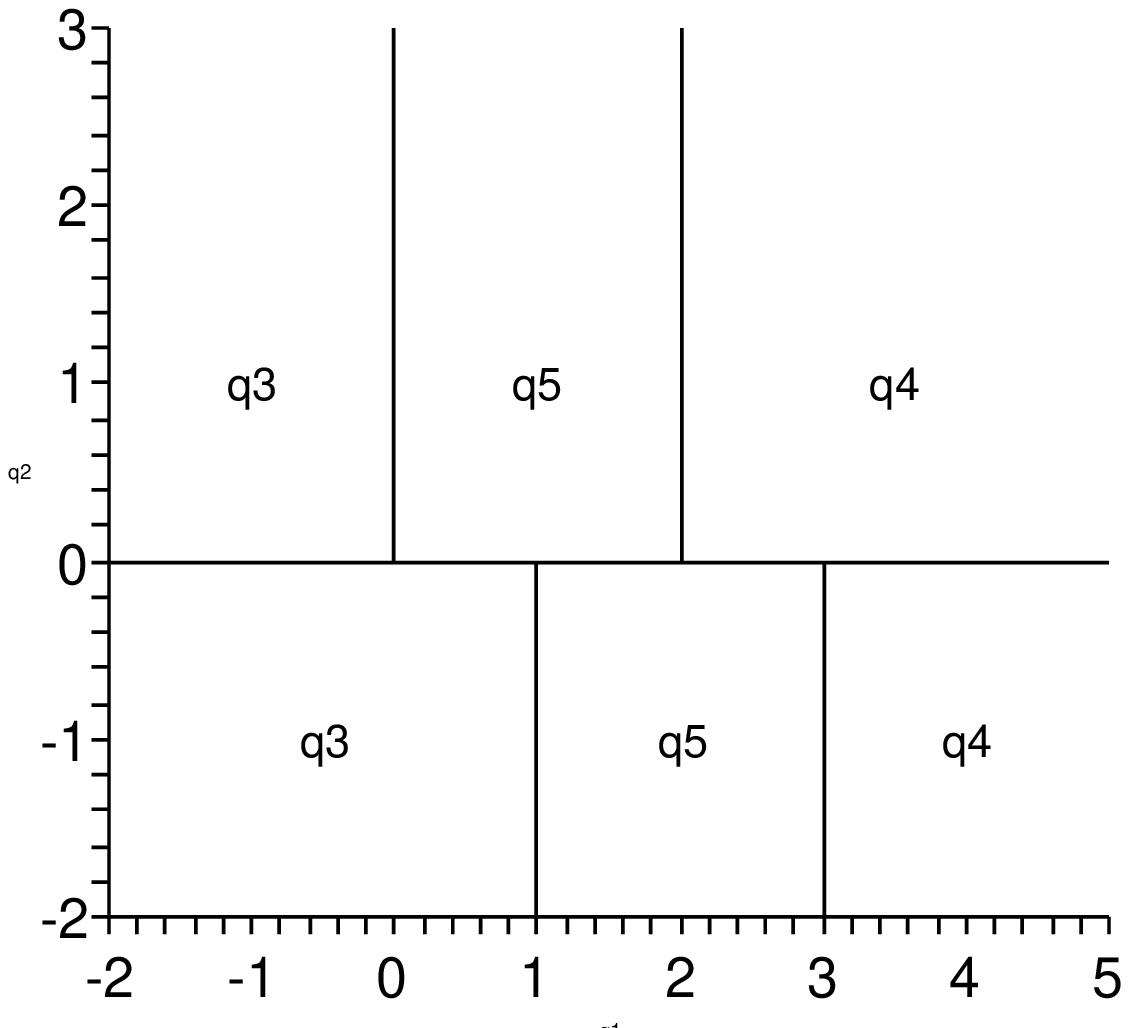}
\caption{Cross-sections of the characteristic sets of $\im p_1$ and $\im p_2$, respectively.}\label{fig:nollställe1}
\end{figure}
\begin{lem}\label{cor:minimal06}
Let $p=\xi_1+i\im p$, and assume that $L(\gamma)>0$ and that $\im p$ does not depend on $\xi_1$. Then
one can find $\widetilde {\gamma}_j \subset {\gamma}_j \dashrightarrow
{\gamma}$ such that
$|\widetilde{\gamma}_j| \to L({\gamma})$,  $\im p$ strongly changes sign  from $-$ to $+$ on
${\gamma_j}$ and $\im p$ vanishes in a neighborhood of
$\widetilde{\gamma}_j$.
\end{lem}
\noindent Note that the conditions imply that $\widetilde {\gamma}_j \dashrightarrow
{\gamma}$ as $j\to\infty$.
\begin{proof}
Choose ${\gamma}_{j} \dashrightarrow {\gamma}$ when
$j \to \infty$ as in the proof of Proposition
\ref{prop:minimal03},
so that  $\im p$ strongly changes sign from $-$ to $+$ on
${\gamma}_{j}$ and $L({\gamma}) = \lim_{j \to \infty}
|{\gamma}_{j}|$. By Lemma \ref{prop:minimal05} and Corollary
\ref{cor:minimal03} we have
\[
L({\gamma}) \le \liminf_{j \to \infty} L({\gamma}_j) \le \liminf_{j
\to \infty} |{\gamma}_j| = L({\gamma}).
\]
Thus we can for every ${\varepsilon} > 0$ choose $j$ so that
$|L({\gamma}) - |{\gamma}_j|| < {\varepsilon}$ and $|L({\gamma}_j) -
|{\gamma}_j|| < {\varepsilon}$.
If we choose $\varepsilon<L(\gamma)/5$ then
\[
2\varepsilon<(L(\gamma)-\varepsilon)/2<|\gamma_j|/2.
\]
Hence, if ${\gamma}_j = [a_j, b_j]\times
w_j$ then by using Proposition \ref{prop:minimal04} on ${\gamma}_j$
we find that $\im p$ vanishes identically in a neighborhood of $\widetilde {\gamma}_j = [a_j +
2{\varepsilon}, b_j -2{\varepsilon} ] \times \{w_j\}$.
Now choose a sequence $\varepsilon_k\to 0$ as $k\to\infty$.
Then $\widetilde{\gamma}_{j(k)} \subset {\gamma}_{j(k)}$
and assuming as we may that $j(k)>j(k')$ if $k>k'$
we obtain $|\widetilde {\gamma}_{j(k)}| \to L({\gamma})$
as $k\to\infty$, which completes the proof.
\end{proof}

If $\varGamma\subset\gamma$ is a minimal bicharacteristic interval
in $T^\ast(\mathbb{R}^n)\smallsetminus 0$ of the homogeneous function $p=\xi_1+i\im p$
of degree $1$, where the imaginary part is independent of $\xi_1$,
then by Definition \ref{dfn:minimal1} and Proposition \ref{prop:minimal03}
we have $0<|\varGamma|=L(\varGamma)$.
By the proof of Lemma \ref{cor:minimal06} there
exists a sequence $\gamma_j\to\varGamma$ of bicharacteristics of
$\re p$ such that $\im p$ strongly changes sign from $-$ to $+$ on $\gamma_j$
and vanishes identically in a neighborhood of a subinterval $\tilde{\gamma}_j
\subset\gamma_j$. Moreover, $\tilde{\gamma}_j\to\varGamma$ as $j\to\infty$.
By Lemma \ref{prop:minimal05}
we have $L(\gamma_j)>0$ for sufficiently large $j$, so
according to Proposition \ref{prop:minimal03} we can for each
such
$j$ find a minimal bicharacteristic interval $\varGamma_j\subset\gamma_j$.
We have $\gamma_j\to\varGamma$ as $j\to\infty$ and since
\begin{align*}
|\varGamma|=L(\gamma)&\le\liminf_{j\to\infty}L(\gamma_j)=\liminf_{j\to\infty}|\varGamma_j|\\
&\le\limsup_{j\to\infty}|\varGamma_j|
\le\lim_{j\to\infty}|\gamma_j|=|\varGamma|,
\end{align*}
it follows that $\varGamma_j\to\varGamma$ as $j\to\infty$.
Since also $\tilde{\gamma}_j\subset\gamma_j$
and $\tilde{\gamma}_j\to\varGamma$ as $j\to\infty$, the intersection $\tilde{\gamma}_j\cap\varGamma_j$
must be nonempty for large $j$.
For such $j$ it follows that $\tilde{\gamma}_j$ must
be a proper subinterval of $\varGamma_j$, for if not,
this would contradict the fact that $\varGamma_j$
is a minimal bicharacteristic interval.
Hence we can find a sequence $\{\varrho_j\}$ of positive
numbers with $\varrho_j\to 0$ as $j\to\infty$, such that $\varGamma_j$ is a
$\varrho_j$-minimal bicharacteristic interval. We have thus
proved the following theorem, which concludes our study of the bicharacteristics.
\begin{thm}\label{thm:minimal07}
If $\varGamma$ is a minimal bicharacteristic interval
in $T^\ast(\mathbb{R}^n)\smallsetminus 0$ of the homogeneous function $p=\xi_1+i\im p$
of degree $1$, where the imaginary part is independent of $\xi_1$, then
there exists a sequence $\{\varGamma_j\}$ of $\varrho_j$-minimal
bicharacteristic intervals of $p$ such that $\varGamma_j\to\varGamma$
and $\varrho_j\to 0$ as $j\to\infty$.
\end{thm}

We can now state our main theorem, which yields necessary conditions
for inclusion relations between the ranges of operators which fail
to be microlocally solvable.

\begin{thm}\label{bigthm1}
Let $K \subset T^{\ast}(X) \smallsetminus 0$ be a compactly based cone. Let
$P\in \varPsi_{\mathrm{cl}}^k(X)$ and $Q\in \varPsi_{\mathrm{cl}}^{k'}(X)$
be properly supported pseudo-differential operators such that
the range of $Q$ is microlocally contained in the range of $P$ at $K$,
where $P$ is an operator of principal type
in a conic neighborhood of $K$.
Let $p_k$ be the homogeneous principal symbol
of $P$, and let $I=[a_0,b_0]\subset \mathbb{R}$ be a compact interval possibly reduced
to a point. Suppose that $K$ contains a conic neighborhood of $\gamma(I)$, where
$\gamma:I\to T^\ast(X)\smallsetminus 0$ is either
\begin{enumerate}
\item[(a)] a minimal characteristic point of $p_k$,
or
\item[(b)] a minimal bicharacteristic interval of $p_k$
with injective regular projection in $S^\ast (X)$.
\end{enumerate}
Then there exists a pseudo-differential operator
$E\in \varPsi_{\mathrm{cl}}^{k'-k}(X)$ such that the terms in the asymptotic sum of the symbol of $Q-PE$  
have vanishing Taylor coefficients at $\gamma(I)$.
\end{thm}

Note that the hypotheses of Theorem \ref{bigthm1} imply
that $P$ is not solvable at the cone $K$. Indeed,
solvability at $K\subset T^\ast(X)\smallsetminus 0$ implies solvability
at any smaller closed cone, 
and in view of Definition \ref{dfn:minimal1} it follows
by ~\cite[Theorem $26.4.7 '$]{ho4}
together with ~\cite[Proposition $26.4.4$]{ho4} that $P$ is
not solvable at the cone generated by $\gamma(I)$.
Conversely, suppose that $P$ is an operator of principal type that is
not microlocally solvable
in any neighborhood of a point $(x_0,\xi_0)\in T^\ast(X)\smallsetminus 0$.
Then the principal symbol $p_k$ fails to satisfy condition \eqref{intropsi} in every
neighborhood of $(x_0,\xi_0)$ by ~\cite[Theorem $1.1$]{de}. In view
of the alternative version of condition \eqref{intropsi} given by
~\cite[Theorem $26.4.12$]{ho4}, it is then easy to see using
~\cite[Theorem $21.3.6$]{ho3} and ~\cite[Lemma $26.4.10$]{ho4}
that $(x_0,\xi_0)$ is a minimal characteristic point of $p_k$,
so Theorem \ref{bigthm1} applies 
there.

We also mention that if $P$ is of principal type and $\gamma$ is a minimal
bicharacteristic interval of the principal
symbol $p_k$ contained in a curve
along which $p_k$ fails to satisfy condition \eqref{intropsi},
then $\gamma$ has injective regular projection in $S^\ast (X)$ by
the proof of ~\cite[Theorem 26.4.12]{ho4}.

\begin{rmk}\label{rmk:solvablebutnonelliptic1}
As pointed out in the introduction, we cannot hope to obtain a result such
as Theorem \ref{bigthm1} for solvable non-elliptic operators in general.
Indeed, Example \ref{ex:d1solvable} shows
that if $X\subset\mathbb{R}^n$ is open, and $K\subset T^\ast (X)\smallsetminus 0$ is a compactly based cone,
then the range of $D_2$ is microlocally contained in
the range of $D_1$ at $K$.
If there were to exist a pseudo-differential operator $e(x,D)\in\varPsi_{\mathrm{cl}}^0(X)$
such that all the terms in the symbol of $R(x,D)=D_2-D_1 \circ e(x,D)$ have vanishing Taylor coefficients
at a point $(x_0,\xi_0)\in K$ contained in a bicharacteristic of the principal symbol
$\sigma(D_1)=\xi_1$ of $D_1$, then in particular this would hold for
the principal symbol
\begin{equation*}\label{eq:solvablebutnonelliptic1}
\sigma(R)(x,\xi)=\xi_2-\xi_1 e_0(x,\xi),
\end{equation*}
if $e_0$ denotes the principal symbol of $e(x,D)$.
However, taking the $\xi_2$ derivative of the equation above and
evaluating at $(x_0,\xi_0)$ then immediately yields the contradiction $0=1$ since
$(x_0,\xi_0)$ belongs to the hypersurface $\xi_1=0$.
\end{rmk}

In the proof of the theorem
we may assume that $P$ and $Q$ are operators of order $1$.
In fact, the discussion following Definition \ref{defrange}
shows that if the conditions of Theorem \ref{bigthm1} hold
and $Q_1\in \varPsi_{\mathrm{cl}}^{k-k'}(X)$ and
$Q_2\in \varPsi_{\mathrm{cl}}^{1-k}(X)$ are properly supported,
then the range of $Q_2QQ_1\in \varPsi_{\mathrm{cl}}^{1}(X)$
is microlocally contained in the range of $Q_2P\in \varPsi_{\mathrm{cl}}^{1}(X)$ at $K$.
If the theorem holds for operators of the same order $k$ then there exists an
operator $E\in\varPsi_{\mathrm{cl}}^{0}(X)$
such that all the terms in the asymptotic expansion of the symbol of $QQ_1-PE$
have vanishing Taylor coefficients at $\gamma(I)$. If we choose $Q_1$ to be elliptic,
then we can find a parametrix $Q_1^{-1}$ of $Q_1$ so that
\[
Q-PEQ_1^{-1}\equiv (QQ_1-PE)\circ Q_1^{-1} \quad \text{mod }\varPsi^{-\infty}(X)
\]
has symbol
\begin{equation}\label{eq:claimindepofxi1}
\sigma_{A\circ Q_1^{-1}}(x,\xi)\sim \sum \partial_\xi^\alpha \sigma_A(x,\xi) D_x^\alpha
\sigma_{Q_1^{-1}}(x,\xi)/\alpha!
\end{equation}
with $A=QQ_1-PE$. Clearly, all the terms in the asymptotic expansion of the symbol of $Q-PEQ_1^{-1}$
then have vanishing Taylor coefficients at $\gamma(I)$, and $E_1=EQ_1^{-1}\in\varPsi_{\mathrm{cl}}^{k'-k}(X)$
so the theorem holds with $E$ replaced by $E_1$. If the theorem holds for operators of order $1$
we can choose $Q_2$ elliptic and use the same argument to show that if
all the terms in the asymptotic expansion of the symbol of $Q_2QQ_1-Q_2PE$
have vanishing Taylor coefficients at $\gamma(I)$, then the same holds for
\[
Q-PEQ_1^{-1}\equiv Q_2^{-1}\circ (Q_2QQ_1-Q_2PE)\circ Q_1^{-1} \quad \text{mod }\varPsi^{-\infty}(X),
\]
where $Q_2^{-1}$ is a parametrix of $Q_2$.
Here we use the fact that if $\gamma(I)$ is a minimal characteristic
point or a minimal bicharacteristic interval of the principal symbol
of $P$, then this also holds for the principal symbol of $Q_2P$
by Definition \ref{dfn:minimal1}.

For pseudo-differential operators, the property
that all terms in the asymptotic expansion of the total symbol have vanishing
Taylor coefficients is preserved
under conjugation with Fourier integral operators
associated with a canonical transformation (see Lemma \ref{symbolvanishafterconjugation}
in the appendix).
Thus we will be able to prove Theorem \ref{bigthm1} by local arguments and an application of
Proposition \ref{prop.26.4.4}.

Let $\gamma:I\to T^\ast(X)\smallsetminus 0$, $I=[a_0,b_0]\subset\mathbb{R}$, be the map given
by Theorem \ref{bigthm1}.
By
using ~\cite[Theorem $21.3.6$]{ho3} or ~\cite[Theorem $26.4.13$]{ho4} when $\gamma$ is a characteristic point
or a one dimensional bicharacteristic, respectively, we can find a $C^\infty$ canonical transformation $\chi$ from
a conic neighborhood of $\varGamma=\{(x,\varepsilon_n): x_1\in I, \, x'=0\}$ in $T^\ast(\mathbb{R}^n)\smallsetminus 0$
to a conic neighborhood of $\gamma(I)$ in $T^\ast(X)\smallsetminus 0$ and a $C^\infty$ homogeneous function
$b$ of degree $0$
with no zero on $\gamma(I)$ such that $\chi(x_1,0,\varepsilon_n)=\gamma(x_1)$, $x_1\in I$,
and
\begin{equation}\label{eq:normal1}
\chi^\ast(bp_1)=\xi_1+if(x,\xi')
\end{equation}
where $f$ is real valued, homogeneous of degree $1$ and independent of $\xi_1$.
Thus, by the hypotheses of Theorem \ref{bigthm1}
one can in any neighborhood of $\varGamma$
find an interval in the $x_1$ direction where $f$ changes sign from $-$ to $+$ for
increasing $x_1$. Also, if $I$ is an interval then $f$ vanishes of infinite order on $\varGamma$
by \eqref{eq:minimal04}, and by Theorem \ref{thm:minimal07}
there exists a sequence $\{\varGamma_j\}$ of $\varrho_j$-minimal
bicharacteristics of $\chi^\ast(bp_1)$ such that
$\varrho_j\to 0$ and $\varGamma_j\to\varGamma$ as $j\to\infty$.

The existence of the canonical transformation $\chi$ together with
Proposition \ref{prop.26.4.4} implies that we can find Fourier integral operators $A$ and $B$
such that the range of $BQA$ is microlocally contained in the range of $BPA$ at a cone $K'$ containing
$\varGamma$, where the principal symbol of $BPA$ is
given by \eqref{eq:normal1}. 
In view of Lemma \ref{symbolvanishafterconjugation} in Appendix \ref{appendix1} we may therefore
reduce the proof to the case $P, Q\in\varPsi_{\mathrm{cl}}^1(\mathbb{R}^n)$ and the principal symbol $p$ of
$P$ given by \eqref{eq:normal1}. In accordance with the notation in Proposition
\ref{prop.26.4.4} we will assume that the range of $Q$
is microlocally contained in the range of $P$ at a cone $K$ containing
$\varGamma$, thus renaming $K'$ to $K$.
If
\[\sigma_{Q}=q_1+q_{0}+\ldots
\]
is the asymptotic sum of homogeneous terms of the symbol of $Q$, we can then
use the Malgrange preparation theorem (see ~\cite[Theorem $7.5.6$]{ho1}) to find
$e_{0}, r_1\in C^\infty$ near $\varGamma$ such that
\[q_1(x,\xi)=(\xi_1+if(x,\xi'))e_{0}(x,\xi)+r_1(x,\xi'),
\]
where $r_1$ is independent of $\xi_1$. Restricting to $|\xi|=1$
and extending by homogeneity
we can make $e_{0}$ and $r_1$ homogeneous of degree $0$ and $1$, respectively.
The term of degree $1$ in the symbol of $Q-P\circ e_{0}(x,D)$ is $r_1(x,\xi')$. 
Again, by
Malgrange's preparation theorem we can find $e_{-1}, r_{0}\in C^\infty$ near $\varGamma$ such that
\begin{align*}
q_{0}(x,\xi) & - \sigma_{0}(P\circ e_{0}(x,D))(x,\xi)\\
& =(\xi_1+if(x,\xi'))e_{-1}(x,\xi)+r_{0}(x,\xi'),
\end{align*}
where $e_{-1}$ and $r_{0}$ are homogeneous of degree $-1$ and $0$, respectively, and
$r_{0}$ is independent of $\xi_1$.
The term of degree $0$ in the symbol of
\[
Q-P\circ e_{0}(x,D)-P\circ e_{-1}(x,D)
\]
is $r_0(x,\xi')$.
Repetition of the argument allows us
to write
\begin{equation}\label{eq:bqadef}
Q=P\circ E+R(x,D_{x'})
\end{equation}
where $\sigma_R(x,\xi')=r_{1}(x,\xi')+r_{0}(x,\xi')+\ldots$
is an asymptotic sum of homogeneous terms, all independent of $\xi_1$.
Thus $R(x,D_{x'})$ is a pseudo-differential operator in the $n-1$ variables
$x'$ depending on $x_1$ as a parameter.
Furthermore, the range of $R(x,D_{x'})$ is microlocally contained in the
range of $P$ at $K$. Indeed, 
suppose $N$ is the integer given by Definition \ref{defrange}.
If $g\in H_{(N)}^{\mathrm{loc}}(\mathbb{R}^n)$, then $Rg=PEg-Qg=Pv-Qg$
for some $v\in \mathscr{D}'(\mathbb{R}^n)$, and there exists a $u\in \mathscr{D}'(\mathbb{R}^n)$
such that
\[K\cap W\! F(Qg-Pu)=\emptyset.
\]
Hence,
\[W\! F(P(v-u)-Rg)
\]
does not meet $K$, so the range of $R$ is microlocally contained in the range of $P$ at $K$.
We claim that under the assumptions of Theorem \ref{bigthm1}, this implies that all terms in the asymptotic sum
of the symbol of the
operator $R(x,D_{x'})$ in \eqref{eq:bqadef} have vanishing Taylor coefficients at $\varGamma$,
thus proving Theorem \ref{bigthm1}.
The proof of this claim will be based on the
two theorems stated below.
As we have seen, the principal symbol $p$ of $P$ may be assumed to have
the normal form given by \eqref{eq:normal1}. By means of Theorem \ref{mainthm2} below,
we shall also use the fact that an
even simpler normal form exists near a point where $p=0$ and $\{\re p,\im p\}\ne 0$.
To prove these two theorems, we will use techniques that actually require
the lower order terms of $P$ to be independent of $\xi_1$ near $\varGamma$. However, we claim that this may
always be assumed. In fact, Malgrange's preparation theorem implies that
\[
p_0(x,\xi)=a(x,\xi)(\xi_1+if(x,\xi'))+b(x,\xi')
\]
where $a$ is homogeneous of degree $-1$ and $b$ homogeneous of degree $0$, as
demonstrated in the construction of the operators $E$ and $R$ above. The
term of degree $0$ in the symbol of $(I-a(x,D))P$ is equal to $b(x,\xi')$.
Repetition of the argument implies that
there exists a classical operator $\widetilde{a}(x,D)$ of order $-1$ such that
$(I-\widetilde{a}(x,D))P$ has principal symbol
$\xi_1+if(x,\xi')$ and all lower order terms are independent of $\xi_1$.
The microlocal property of pseudo-differential
operators immediately implies that the range of $(I-\widetilde{a}(x,D))Q$ is microlocally contained in the range
of $(I-\widetilde{a}(x,D))P$ at $K$.
Hence, if there are operators $E$ and $R$ with
\[
R=(I-\widetilde{a}(x,D))Q-(I-\widetilde{a}(x,D))P E
\]
such that all terms in the asymptotic expansion of the symbol
of $R$ have vanishing
Taylor coefficients at $\varGamma$, then this also holds for
the symbol of $Q-PE\equiv (I-\widetilde{a}(x,D))^{-1}R$ mod $\varPsi^{-\infty}$,
since this property is preserved under composition with elliptic
pseudo-differential operators by \eqref{eq:claimindepofxi1}.

\begin{thm}\label{mainthm2}
Suppose that in a conic neighborhood $\varOmega$ of
\[\varGamma'=\{(0,\varepsilon_n)\}\subset T^*(\mathbb{R}^n)\smallsetminus 0
\]
$P$ has the form $P=D_1+ix_1D_n$ and
the symbol of $R(x,D_{x'})$ is given by the asymptotic sum
\[\sigma_R=\sum_{j=0}^\infty r_{1-j}(x,\xi')
\]
with $r_{1-j}$ homogeneous of degree $1-j$
and independent of $\xi_1$. If there exists a compactly based cone $K\subset T^*(\mathbb{R}^n)\smallsetminus 0$
containing $\varOmega$
such that the range of $R$ is microlocally contained in the range of $P$ at $K$, 
then all the terms in the asymptotic sum of the symbol of $R$ have vanishing Taylor coefficients at $\varGamma'$.
\end{thm}
\begin{thm}\label{mainthm1}
Suppose that in a conic neighborhood $\varOmega$ of
\[\varGamma'=\{(x_1,x',0,\xi'): a\leq x_1 \leq b\}\subset T^*(\mathbb{R}^n)\smallsetminus 0
\]
the principal symbol of $P$ has the form
\[p(x,\xi)=\xi_1+if(x,\xi')
\]
where $f$ is real valued and homogeneous of degree $1$, and suppose that if $b>a$
then $f$ vanishes of infinite order on $\varGamma'$ and there exists a $\varrho\ge 0$ such that
for any $\varepsilon>\varrho$ one can
find a neighborhood of
\begin{equation}\label{eq:mainthm11}
\varGamma_\varepsilon'=\{(x_1,x',0,\xi'): a+\varepsilon \leq x_1 \leq b-\varepsilon\}
\end{equation}
where $f$ vanishes identically.
Suppose also that
\begin{equation}\label{eq:mainthm1}
f(x,\xi')=0 \quad \Longrightarrow \quad \partial f (x,\xi') / \partial x_1 \leq 0
\end{equation}
in $\varOmega$ and that in
any neighborhood of $\varGamma'$ one can find an interval in the $x_1$ direction where $f$
changes sign from $-$ to $+$
for increasing $x_1$. Furthermore, suppose that in $\varOmega$
the symbol of $R(x,D_{x'})$ is given by the asymptotic sum
\[\sigma_R=\sum_{j=0}^\infty r_{1-j}(x,\xi')
\]
with $r_{1-j}$ homogeneous of degree $1-j$
and independent of $\xi_1$. If
the lower order terms $p_0,p_{-1},\ldots$ in the symbol of $P$
are independent of $\xi_1$ near $\varGamma'$, and
there exists a compactly based cone $K\subset T^*(\mathbb{R}^n)\smallsetminus 0$
containing $\varOmega$
such that the range of $R$ is microlocally contained in the range of $P$ at $K$, 
then all the terms in the asymptotic sum of the symbol of $R$ have vanishing Taylor coefficients on
$\varGamma_\varrho'$ if $a<b$, and at $\varGamma'$ if $a=b$.
\end{thm}

Assuming these results for the moment, we can now show how Theorem \ref{bigthm1} follows.
\begin{proof}[End of Proof of Theorem \ref{bigthm1}.]
Recall that
\[
\varGamma=\{(x_1,0,\varepsilon_n):  a_0\le x_1\le b_0\}\subset T^\ast(\mathbb{R}^n)\smallsetminus 0.
\]
By what we have shown, it suffices to regard the case $Q=PE+R$, where we may assume that the conditions of Theorem
\ref{mainthm1} are all satisfied in a conic neighborhood $\varOmega$ of $\varGamma$,
with the exception of \eqref{eq:mainthm1} and the condition
concerning the existence of a neighborhood of \eqref{eq:mainthm11} in which
$f$ vanishes identically when $a_0<b_0$.
We consider three cases.

i) $\varGamma$ is an interval.
We then claim that condition \eqref{eq:mainthm1} 
imposes no restriction.
Indeed, if there is no neighborhood of $\varGamma$ in which \eqref{eq:mainthm1} holds,
then there exists a sequence $\{\gamma_j\}=\{(t_j,x_j',0,\xi_j')\}$ such that
$a_0\le\liminf t_j\le\limsup t_j\le b_0$, $(x_j',\xi_j')\to (0,\xi^0)\in \mathbb{R}^{2n-2}$
and
\begin{equation}\label{eq:rmk1}
f(t_j,x_j',\xi_j')=0, \quad \partial f(t_j,x_j',\xi_j') / \partial x_1 >0
\end{equation}
for each $j$.
By \eqref{eq:rmk1} we can choose a sequence $0<\delta_j\to 0$ such that
\[f(t_j-\delta_j,x_j',\xi_j')<0<f(t_j+\delta_j,x_j',\xi_j').
\]
In view of Definition \ref{def:minimal02} we must therefore have $L(\varGamma)=0$.
Since $\varGamma$ is minimal, this implies that $|\varGamma|=0$ so $\gamma_j\to\varGamma$.
Thus, if there is no neighborhood of $\varGamma$ in which \eqref{eq:mainthm1} holds, then $\varGamma$
is a point, and we will in this case use the existence of the sequence $\{\gamma_j\}$ satisfying \eqref{eq:rmk1}
to reduce
the proof of Theorem \ref{bigthm1} to Theorem \ref{mainthm2},
as demonstrated in case iii) below. In the present case however, $\varGamma$ is assumed to be
an interval, so there exists a
neighborhood $\mathcal{U}$ of $\varGamma$ in which \eqref{eq:mainthm1} holds.
We may assume that $\mathcal{U}\subset\varOmega$ and since $f$ is homogeneous of degree $1$
we may also assume that $\mathcal{U}$ is conic.

By Theorem \ref{thm:minimal07}, there exists a sequence $\{\varGamma_j\}$ of
$\varrho_j$-minimal bicharacteristic intervals
such that $\varrho_j\to 0$ and $\varGamma_j\to\varGamma$ as $j\to\infty$. For sufficiently
large $j$ we have $\varGamma_j\subset\mathcal{U}$. Hence, if
\[
\varGamma_j=\{(x_1,x_j',0,\xi_j'): a_j \leq x_1 \leq b_j\}
\]
then all the terms in the
asymptotic sum of the symbol of $R$ have vanishing Taylor coefficients on
\[
\varGamma_{\varrho_j}=\{(x_1,x_j',0,\xi_j'):  a_j+\varrho_j \leq x_1 \leq b_j-\varrho_j\}
\]
by Theorem \ref{mainthm1}. Since $\varGamma_{\varrho_j}\to\varGamma$
as $j\to\infty$, and all the terms in the
asymptotic sum of the symbol of $R$ are smooth functions, it follows that
all the terms in the
asymptotic sum of the symbol of $R$ have vanishing Taylor coefficients on
$\varGamma$. This proves Theorem \ref{bigthm1} in this case.

ii) $\varGamma$ is a point and condition \eqref{eq:mainthm1} holds. Then all the terms in the
asymptotic sum of the symbol of $R$ have vanishing Taylor coefficients on
$\varGamma$ by Theorem \ref{mainthm1}, so Theorem
\ref{bigthm1} follows.

iii) $\varGamma$ is a point and
\eqref{eq:mainthm1} is false. Let
$\{\gamma_j\}$ be the sequence satisfying \eqref{eq:rmk1}.
We then have $\{\re p, \im p\}(\gamma_j)>0$ and $p(\gamma_j)=0$
for each $j$ since $\gamma_j=(t_j,x_j',0,\xi_j')$.
For fixed $j$ we may assume that
$\gamma_j=(0,\eta)$ and use ~\cite[Theorem $21.3.3$]{ho3} to find a canonical
transformation $\chi$ together with Fourier integral operators $A, B, A_1$ and $B_1$
as in Proposition \ref{prop.26.4.4} such that $\chi(0,\varepsilon_n)=\gamma_j$,
and $BPA=D_1+ix_1D_n$ in a conic neighborhood $\varOmega$ of $\{(0,\varepsilon_n)\}$.
Repetition of the arguments above allows us to write
\begin{equation}\label{eq:rmk2}
BQA=BPAE+R(x,D_{x'}),
\end{equation}
where the range of $R$ is microlocally contained in the range of $BPA$ at some
compactly based cone $K'$ containing $\varOmega$ with $\chi(K')=K$. As before,
$E$ and $R$ have classical symbols. Then all the terms
in the asymptotic expansion of the symbol of
$R$ have vanishing Taylor coefficients at $\{(0,\varepsilon_n)\}$
by Theorem \ref{mainthm2}, and therefore all the terms
in the asymptotic expansion of the symbol of $A_1RB_1$
have vanishing Taylor coefficients at $\gamma_j$ by Lemma \ref{symbolvanishafterconjugation}
in the appendix.
Since the Fourier integral operators are chosen so that
\begin{align*}
K  \cap W\! F(A_1B-I) = \emptyset, \quad
K  \cap W\! F(AB_1-I) = \emptyset,
\end{align*}
we have
\begin{align*}
\emptyset & =K\cap W\! F(A_1BQAB_1-A_1BPAEB_1-A_1RB_1) \\
& = K\cap W\! F(Q-PAEB_1-A_1RB_1)
\end{align*}
in view of \eqref{eq:rmk2}.
Hence, all the terms
in the asymptotic expansion of the symbol of
\begin{equation}\label{eq:opej}
Q-PE_1=A_1RB_1 + S, \quad \quad W\! F(S)\cap K=\emptyset,
\end{equation}
have vanishing Taylor coefficients at $\gamma_j$
if $E_1=AEB_1$. (Strictly speaking,
the change of base variables $\gamma_j\mapsto (0,\eta)$ should be represented in
\eqref{eq:opej} by conjugation of a linear transformation $\kappa:\mathbb{R}^n\rightarrow \mathbb{R}^n$,
but this could be integrated in the Fourier integral operators
$A_1$ and $B_1$ so it has been left out since it will not affect the arguments below.)
It is clear that $E_1\in\varPsi_{\mathrm{cl}}^{0}(\mathbb{R}^n)$.

We have now shown that for each $j$ there exists an operator
$E_j\in\varPsi_{\mathrm{cl}}^{0}(\mathbb{R}^n)$ such that
all the terms
in the asymptotic expansion of the symbol of $Q-PE_j$
have vanishing Taylor coefficients at $\gamma_j$.
To construct the operator $E$ in Theorem \ref{bigthm1}, we do the following.
For each $j$, denote the symbol of $E_j$ by
\[
e^j(x,\xi)\sim\sum_{l=0}^\infty e_{-l}^j(x,\xi)
\]
where $e_{0}^j(x,\xi)$ is the principal part,
and $e_{-l}^j(x,\xi)$ is homogeneous of degree $-l$.
If $q$ is the principal symbol of $Q$, then by Proposition \ref{appthm26} in the appendix
there exists a function $e_0\in C^\infty(T^\ast(\mathbb{R}^n)\smallsetminus 0)$,
homogeneous of degree $0$, such that $q-p e_0$ has vanishing Taylor coefficients at $\varGamma$.

This argument can be repeated for lower order terms. Indeed, if
$\sigma_Q=q+q_0+\ldots$, then the term of degree $0$
in $\sigma_{Q-PE_j}$ is
\[
\sigma_0(Q-PE_j)=\tilde{q}_j-pe_{-1}^j,
\]
where (see equation \eqref{eq:vectorfields4} below)
\begin{align*}
\tilde{q}_j(x,\xi)&=q_0(x,\xi)
-p_0(x,\xi) e_0^j(x,\xi)
-\sum_k \partial_{\xi_k} p(x,\xi) D_{x_k}e_0^j(x,\xi) .
\end{align*}
We can write
\[
p(x,\xi)e_{-1}^j(x,\xi)=p(x,\xi/|\xi|)e_{-1}^j(x,\xi/|\xi|),
\]
so that $\tilde{q}_j(x,\xi)$, $p(x,\xi/|\xi|)$ and $e_{-1}^j(x,\xi/|\xi|)$ are
all homogeneous of degree $0$. Since
\[
\partial_x^\alpha\partial_\xi^\beta e_0(\varGamma)=\lim_{j\to\infty}
\partial_x^\alpha\partial_\xi^\beta e_0^j(\gamma_j)
\]
it follows by Proposition \ref{appthm26}
that there is a function $g\in C^\infty(T^\ast(\mathbb{R}^n)\smallsetminus 0)$,
homogeneous of degree $0$, such that
\begin{align*}
q_0(x,\xi)&
-p_0(x,\xi) e_0(x,\xi)
-\sum_k \partial_{\xi_k} p(x,\xi) D_{x_k} e_0(x,\xi)\\
& - p(x,\xi/|\xi|)g(x,\xi)
\end{align*}
has vanishing Taylor coefficients at $\varGamma$. Putting $e_{-1}(x,\xi)=|\xi|^{-1}g(x,\xi)$
we find that
\[
\partial_x^\alpha\partial_\xi^\beta e_{-1}(\varGamma)=\lim_{j\to\infty}
\partial_x^\alpha\partial_\xi^\beta e_{-1}^j(\gamma_j),
\]
and that
\begin{align*}
\sigma_0(Q-P\circ e_0(x,D)-P\circ e_{-1}(x,D))
\end{align*}
has vanishing Taylor coefficients at $\varGamma$.
Continuing this way we successively
obtain functions $e_m(x,\xi)\in
C^\infty(T^\ast(\mathbb{R}^n)\smallsetminus 0)$, homogeneous of degree $m$ for $m\le 0$, such that
\[
\sigma_Q-(\sum_{m=0}^M e_{-m})\sigma_P \quad \text{mod }S_{\mathrm{cl}}^{-M}
\]
has vanishing Taylor coefficients at $\varGamma$. If we let $E$ have symbol
\[
\sigma_E(x,\xi)\sim \sum_{m=0}^\infty (1-\phi(\xi))e_{-m}(x,\xi)
\]
with $\phi\in C_0^\infty$ equal to $1$ for $\xi$ close to $0$,
then $E\in\varPsi_{\mathrm{cl}}^{0}(\mathbb{R}^n)$ and all terms in the
asymptotic expansion of the symbol of $Q-PE$ have vanishing
Taylor coefficients at $\varGamma$.
This completes the proof of Theorem \ref{bigthm1}.
\end{proof}
\begin{rmk}
Instead of reducing to the study
of the normal form $P=D_{x_1}+ix_1D_{x_n}$ when condition \eqref{eq:mainthm1}
does not hold, as in case iii) above, one could show that the terms in the
asymptotic expansion of the operator $R$
given by \eqref{eq:bqadef} has vanishing Taylor coefficients at every point
in the sequence $\{\gamma_j\}$ satisfying \eqref{eq:rmk1} using techniques
very similar to those used to prove Theorem \ref{mainthm1}. Theorem \ref{bigthm1}
would then follow by continuity, but the proof of the analogue of
Theorem \ref{mainthm2} would be more involved.
In particular, we would have to construct a phase function $w$ solving the eiconal equation
\[
\partial w / \partial x_1 - if(x,\partial w/ \partial x')=0
\]
approximately instead of explicitly (confer the proofs of
Theorems \ref{mainthm1} and \ref{mainthm2}, respectively). For fixed $j$ this could be accomplished by
adapting the approach in ~\cite{ho0,ho17}
(for a brief discussion, see ~\cite[p. $83$]{ho40})
where one has $f=0$ and $\partial f/\partial x_1 >0$ at $(0,\xi^0)$
instead of at $\gamma_j$.
\end{rmk}

We shall now show how our results relates to the ones referred to in the introduction,
beginning with \eqref{eq:introeq1}. There, it sufficed to have the coefficients of $P$ and $Q$
in $C^\infty$ and $C^1$, respectively. However, in order for Theorem \ref{bigthm1} to qualify,
we must require both $P$ and $Q$ to have smooth coefficients. On the other hand,
we shall only require the equation $Pu=Qf$ to be microlocally solvable (at an appropriate cone $K$)
as given
by Definition \ref{defrange}.
Note that if $P$ is a first order differential operator on an open set $\varOmega\subset\mathbb{R}^n$,
such that
the principal symbol $p$ of $P$ satisfies condition \eqref{eqintrocond1} at a point
$(x,\xi)\in T^\ast(\varOmega)\smallsetminus 0$,
then either $\{\re p, \im p\}>0$ at $(x,\xi)$, or $\{\re p, \im p\}>0$ at $(x,-\xi)$.
(The order of the operator is not important; the statement is still true for a differential operator of order $m$,
since the Poisson bracket is
then homogeneous of order $2m-1$.)
Assuming the former, this implies that $(x,\xi)$ satisfies condition $(\mathrm{a})$ in Theorem \ref{bigthm1}
by an application of ~\cite[Theorem $21.3.3$]{ho3} and Lemma \ref{lem:minimal01}.
In order to keep the formulation of the following
result as simple as possible, we will
assume that there exists a compactly based cone $K\subset T^\ast(\varOmega)\smallsetminus 0$
with non-empty interior such that $K$
contains the appropriate point $(x,\pm\xi)$, and 
such that the equation $Pu=Qf$ is microlocally solvable at $K$.
This is clearly the case if the equation $Pu=Qf$ is locally solvable in $\varOmega$ in the
weak sense suggested by \eqref{eqintrolocsolv}.
\begin{cor}\label{cor:jmfvectorfield}
Let $\varOmega\subset\mathbb{R}^n$ be open, and let
$P(x,D)$ and $Q(x,D)$ be two first order differential operators with
coefficients in $C^\infty (\varOmega)$. Let $p$ be the principal symbol of $P$,
and let
$x_0$ be a point in $\varOmega$ such that
\begin{equation}\label{eq:coreq1}
p(x_0,\xi_0)=0, \quad \{\re p, \im p\}(x_0,\xi_0)>0
\end{equation}
for some $\xi_0\in\mathbb{R}^n$. If $K\subset T^\ast(\varOmega)\smallsetminus 0$
is a compactly based cone containing $(x_0,\xi_0)$ such that the range of $Q$
is microlocally contained in the range of $P$ at $K$, then there exists a
constant $\mu$ such that (at the fixed point $x_0$)
\begin{equation}\label{eq:coreq2}
Q^\ast(x_0,D)=\mu P^\ast (x_0,D)
\end{equation}
where $Q^\ast$ and $P^\ast$ are
the adjoints of $Q$ and $P$.
\end{cor}
\begin{proof}
By \eqref{eq:coreq1}, $P\in\varPsi_{\mathrm{cl}}^1(\varOmega)$ is an operator of principal type
microlocally near $(x_0,\xi_0)$.
$P$ and $Q$ therefore satisfy the hypotheses of Theorem \ref{bigthm1},
and in view of the discussion above regarding the point $(x,\xi)$ we find
that there exists an operator $E\in\varPsi_\mathrm{cl}^0(\varOmega)$
such that all the terms in the asymptotic expansion of the symbol of $Q-PE$
has vanishing Taylor coefficients at $(x_0,\xi_0)$. By the discussion
following equation \eqref{specialeq:asexpforadjoint} on page \pageref{specialeq:asexpforadjoint}
below, it follows that the same must hold for the adjoint $Q^\ast-E^\ast P^\ast$.
If we let $Q^\ast$ and $P^\ast$ have symbols $\sigma_{Q^\ast}(x,\xi)=q_1(x,\xi)+q_0(x)$ and
$\sigma_{P^\ast}(x,\xi)=p_1(x,\xi)+p_0(x)$,
then $E^\ast P^\ast$ has principal symbol $e_0 p_1$
if $\sigma_{E^\ast}=e_0+e_{-1}+\ldots$ denotes the symbol of $E^\ast$.
Hence
\[
\partial q_1(x_0,\xi_0)/\partial \xi_k=e_0(x_0,\xi_0)\partial p_1(x_0,\xi_0)/\partial \xi_k,
\quad 1\le k\le n,
\]
for $p_1(x_0,\xi_0)=\overline{p(x_0,\xi_0)}=0$. Since $q_1$ and $p_1$
are polynomials in $\xi$ of degree $1$, this means that at the fixed point $x_0$
we have
$q_1(x_0,\xi)=\mu p_1(x_0,\xi)$ for $\xi\in\mathbb{R}^n$ where the constant
$\mu$ is given by the value of $e_0$ at $(x_0,\xi_0)$.
Moreover,
\begin{equation}\label{eq:corvectorfields1}
\begin{aligned}
0&=\partial_{\xi_j} \partial_{\xi_k}q_1(x_0,\xi_0)\\
&=\partial_{\xi_j} e_0(x_0,\xi_0)\partial_{\xi_k} p_1(x_0,\xi_0)
+\partial_{\xi_k} e_0(x_0,\xi_0)\partial_{\xi_j} p_1(x_0,\xi_0).
\end{aligned}
\end{equation}
By
assumption, the coefficients of $p(x,D)$ do not vanish simultaneously,
so the same is true for $p_1(x,D)$.
Hence $\partial_{\xi_j} p_1(x_0,\xi_0)\ne 0$ for some $j$. Assuming this holds for
$j=1$, we find by choosing $j=k=1$ in \eqref{eq:corvectorfields1} that
$\partial_{\xi_1} e_0(x_0,\xi_0)=0$. But this immediately yields
\[
\partial_{\xi_k} e_0(x_0,\xi_0)=-\partial_{\xi_1} e_0(x_0,\xi_0) \partial_{\xi_k}p_1(x_0,\xi_0)
/\partial_{\xi_1}p_1(x_0,\xi_0)=0
\]
for $2\le k\le n$. Now
\[
\sigma_{E^\ast P^\ast}(x,\xi)\sim\sum \frac{1}{\alpha !}
\partial_\xi^\alpha \sigma_{E^\ast} \, D_x^\alpha (p_1(x,\xi)+p_0(x)),
\]
and since we have a bilinear map
\[
S_{\mathrm{cl}}^{m'} / S^{-\infty}\times S_{\mathrm{cl}}^{m''} / S^{-\infty} \ni
(a,b)\mapsto a\# b \in S_{\mathrm{cl}}^{m'+m''} / S^{-\infty}
\]
with
\[
(a\#b)(x,\xi)\sim\sum \frac{1}{\alpha !}
\partial_\xi^\alpha a(x,\xi) \, D_x^\alpha b(x,\xi),
\]
we find that the term of order $0$ in the symbol of $E^\ast P^\ast$ is
\begin{equation}\label{eq:vectorfields4}
\begin{aligned}
\sigma_0(E^\ast P^\ast)(x,\xi)&=e_{-1}(x,\xi)p_1(x,\xi)+e_0(x,\xi)p_0(x)\\
&\phantom{=}+\sum_{k=1}^n
\partial_{\xi_k}e_0(x,\xi) \, D_k p_1(x,\xi).
\end{aligned}
\end{equation}
Since $\partial_{\xi_k}e_0$ and $p_1$ vanish at $(x_0,\xi_0)$ we find that
$q_0(x_0)=\mu p_0(x_0)$ at the fixed point $x_0$, which completes the proof.
\end{proof}
Having proved this result, we immediately obtain the following after making
the obvious adjustments to
~\cite[Theorem $6.2.2$]{ho0}. The fact that we require higher regularity on the coefficients of
$Q$ then yields higher regularity on the propertionality factor. Since the
proof remains the same, it is omitted.
\begin{cor}\label{cor:jmfvectorfield2}
Let $\varOmega\subset\mathbb{R}^n$ be open, and let
$P(x,D)$ and $Q(x,D)$ be two first order differential operators with
coefficients in $C^\infty (\varOmega)$. Let $p$ be the principal symbol of $P$,
and assume that the coefficients of $p(x,D)$ do not vanish simultaneously in $\varOmega$.
If for a dense set of points $x$ in $\varOmega$ one can find $\xi\in\mathbb{R}^n$ such that
\eqref{eq:coreq1} is fulfilled, and
if for each $(x,\xi)$ there is a compactly based cone $K\subset T^\ast(\varOmega)\smallsetminus 0$
containing $(x,\xi)$ such that the range of $Q$
is microlocally contained in the range of $P$ at $K$, then there exists a
function $e\in C^\infty(\varOmega)$ such that
\begin{equation}\label{eq:coreq3}
Q(x,D)u\equiv P(x,D)(e u). 
\end{equation}
\end{cor}

In stating
Corollary \ref{cor:jmfvectorfield2}
we could replace the assumption that the coefficients of $p(x,D)$ do not vanish simultaneously in $\varOmega$
with the condition that $P$ is of principal type. 
Indeed, if $dp\ne 0$ then by a canonical transformation we find that condition
\eqref{eq:intronontrapping} holds. Since $p\ne 0$ implies $\partial_\xi p\ne 0$ by the Euler homogeneity
equation we then have $\partial_\xi p\ne 0$ everywhere, that is,
the coefficients of $p(x,D)$ do not vanish simultaneously in $\varOmega$.
The converse is obvious.

As shown in Example \ref{ex:commutator} below,
we also recover the result for higher order differential operators
mentioned in the introduction as a special case of the following corollary
to Theorem \ref{bigthm1},
although we again need to assume higher regularity
in order to apply our results.
\begin{prop}\label{thm:commutator}
Let $X$ be a smooth manifold, $P\in\varPsi_\mathrm{cl}^k(X)$ and
$Q\in\varPsi_\mathrm{cl}^{k'}(X)$ be properly supported
such that the range of $Q\circ P$ is microlocally contained in the
range of $P$ at a compactly based cone $K\subset T^\ast (X)\smallsetminus 0$.
Let $p$ and $q$ be the principal symbols of $P$ and $Q$, respectively,
and assume that $P$ is of principal type microlocally near $K$.
If $\gamma : I\rightarrow T^\ast (X)\smallsetminus 0$ is a minimal characteristic
point or a minimal bicharacteristic interval of $p$ contained in $K$
then it follows that
\begin{equation*}
H_p^m (q)
=0
\end{equation*}
for all $(x,\xi)\in \gamma(I)$ and $m\ge 1$.
\end{prop}
\noindent Here $H_p^m(q)$ is defined recursively
by $H_p (q)=\{p,q\}$ and $H_p^m(q)=\{p,H_p^{m-1}(q)\}$ for $m\ge 2$.
\begin{proof}
First note that if the range of $Q\in\varPsi_\mathrm{cl}^{k'}(X)$
is microlocally contained in the
range of $P\in\varPsi_\mathrm{cl}^k(X)$ at $K$
and both operators are properly supported, then it follows that
the range of $Q\circ P$ is microlocally contained in the
range of $P$ at $K$. (The converse is not true in general.)
Indeed, let $N$ be the integer given by Definition
\ref{defrange}, and let $f\in H_{(N+k)}^{\mathrm{loc}}(X)$.
Since $P:H_{(N+k)}^{\mathrm{loc}}(X)\rightarrow H_{(N)}^{\mathrm{loc}}(X)$
is continuous, we have $g=Pf\in H_{(N)}^{\mathrm{loc}}(X)$. Thus, there exists
a $u\in\mathscr{D}'(X)$ such that
\[\emptyset =K\cap W\! F(Qg-Pu)=K\cap W\! F(QPf-Pu),
\]
so the conditions of Definition
\ref{defrange} are satisfied with $N$ replaced with $N+k$.

Let $(x,\xi)\in\gamma(I)$. The range of $PQ$ is easily seen to be
microlocally contained in the range of $P$ for any properly supported
pseudo-differential operator $Q$. The assumptions of the proposition
therefore imply that the range of the commutator
\begin{equation}\label{eq:commrecurs}
R_1=P\circ Q-Q \circ P\in\varPsi_{\mathrm{cl}}^{k+k'-1}(X)
\end{equation}
is microlocally contained in the range of $P$ at $K$.
Hence, by Theorem
\ref{bigthm1} there exists
an operator $E\in\varPsi_{\mathrm{cl}}^{k'-1}(X)$ such that, in particular, the principal symbol
of $R_1-PE$ vanishes at $(x,\xi)$. If $e$ is the principal symbol of $E$, homogeneous of
degree $k'-1$, then the principal symbol of $PE$ satisfies $p(x,\xi)e(x,\xi)=0$
since $p\circ\gamma=0$. Since the principal symbol
of $R_1$ is
\begin{align*}
\sigma_{k+k'-1}(R_1)&=\frac{1}{i}\{p,q\},
\end{align*}
the result follows for $m=1$.

Let $R_m$ be defined recursively by $R_m=[P,R_{m-1}]$ for $m\ge 2$ with $R_1$
given by \eqref{eq:commrecurs}. Arguing by induction, we conclude in view of
the first paragraph of the proof
that the range of $R_m$ is microlocally contained in the
range of $P$ at $K$ for $m=1,2\ldots$ since this holds for $R_1$.
Assuming the proposition holds for some $m\ge 1$, we can repeat the
arguments above to show that the principal symbol of $R_{m+1}$
must vanish at $(x,\xi)$. Since the principal symbol of $R_{m+1}$
equals $\frac{1}{i}\{p,H_p^m(q)\}$, this completes the proof.
\end{proof}

\begin{ex}\label{ex:commutator}
Let $\varOmega\subset\mathbb{R}^n$ be open, $P(x,D)$ be a differential operator
of order $m$ with coefficients in $C^\infty(\varOmega)$,
and let $\mu$ be a function
in $C^\infty(\varOmega)$ such that the equation
\[
P(x,D)u=\mu P(x,D)f
\]
has a solution $u\in\mathscr{D}'(\varOmega)$ for every $f\in C_0^\infty(\varOmega)$.
If $p$ is the principal symbol of $P$ then it follows that
\begin{equation}\label{eq:coreq4}
\sum_{j=1}^n \partial_{\xi_j}p(x,\xi) D_{x_j} \mu(x)=0
\end{equation}
for all $x\in \varOmega$ and $\xi\in \mathbb{R}^n$ such that
\begin{equation}\label{eq:coreq5}
\{ p, \overline{p}\}(x,\xi) \neq 0, \quad p(x,\xi)=0.
\end{equation}
Indeed, if
$(x,\xi)$ satisfies \eqref{eq:coreq5} then we may assume that
\[
\{\re p,\im p\}(x,\xi)=-\frac{1}{2i}\{ p, \overline{p}\}(x,\xi)>0
\]
since otherwise we just regard
$(x,-\xi)$ instead as per the remarks preceding Corollary \ref{cor:jmfvectorfield}.
By the same discussion it is also clear that
$(x,\xi)$ is a minimal characteristic point of $p$.
Now the conditions above imply that there exists a compactly based cone $K\subset T^\ast(\varOmega)
\smallsetminus 0$
containing $(x,\xi)$ such that the range of $\mu P$ is microlocally contained in the range of $P$
at $K$.
By condition \eqref{eq:coreq5}
$P$ is of
principal type near $(x,\xi)$,
so Proposition \ref{thm:commutator} implies that $\{p,\mu\}=0$ at $(x,\xi)$,
that is,
\begin{equation*}
\sum_{j=1}^n \partial_{\xi_j}p(x,\xi) \partial_{x_j} \mu(x)
-\partial_{x_j}p(x,\xi) \partial_{\xi_j} \mu(x)=0.
\end{equation*}
Since $\mu$ is independent of $\xi$ we find that
\eqref{eq:coreq4}
holds at $(x,\xi)$.
By homogeneity it then also holds at $(x,-\xi)$.
\end{ex}

\section{Proof of Theorem \ref{mainthm2}}\label{sec:proofmainthm2}

\noindent Throughout this section we assume that the hypotheses of Theorem \ref{mainthm2}
hold. We shall prove the theorem by using Lemma \ref{lemrange1} on
approximate solutions of the equation $P^\ast v =0$ concentrated near $\varGamma'=\{(0,\varepsilon_n)\}$.
We take as starting point the construction on ~\cite[p. $103$]{ho4}, but some
modifications need to be made in particular to the amplitude function $\phi$,
so the results there concerning the estimates
for the right-hand side of \eqref{rangeeq1} cannot be used immediately. To obtain the desired estimates
we will instead
have to use ~\cite[Lemma $26.4.15$]{ho4}.
Set
\begin{equation}\label{specialeq:defapprox}
v_\tau (x)=\phi(x) e^{i\tau w(x)}
\end{equation}
where
\[w(x)=x_n+i(x_1^2+x_2^2+\ldots+x_{n-1}^2+(x_n+ix_1^2/2)^2)/2
\]
satisfies $P^*w=0$ and $\phi\in C_0^\infty(\mathbb{R}^n)$.
By the Cauchy-Kovalevsky theorem we can solve $D_1\phi-ix_1D_n\phi=0$ in a neighborhood of $0$
for any analytic initial data $\phi(0,x')=f(x')\in C^\omega(\mathbb{R}^{n-1})$;
in particular we are free to specify the Taylor coefficients of $f(x')$ at $x'=0$. We
take $\phi$ to be such a solution. If need be
we can reduce the support of $\phi$ by multiplying by a smooth cutoff function $\chi$
where $\chi$ is equal to $1$ in some
smaller neighborhood of $0$ so that $\chi\phi$ solves the equation there. We assume this to be done
and note that if $\supp\phi$ is small enough then
\begin{equation}\label{specialeq:defapprox1}
\im w(x)\geq|x|^2/4, \quad x\in \supp\phi.
\end{equation}
Since
\[
d \re w(x)=-x_1x_n d x_1 + (1-x_1^2/2) d x_n
\]
we may similarly assume that
$d \re w(x)\ne 0$ in the support of $\phi$. We then have the following result.

\begin{lem}\label{special:lemest1}
Suppose $P=D_1+ix_1D_n$ and let $v_\tau$
be defined by \eqref{specialeq:defapprox}. Then $\phi$ and $w$
can be chosen so that for any
$f\in C^\omega(\mathbb{R}^{n-1})$ and
any positive integers $k$ and $m$
we have $\phi(0,x')=f(x')$ in a neighborhood of $(0,0)$,
$\tau^k\|P^\ast v_\tau \|_{(m)}\rightarrow 0$ as $\tau\to\infty$, and
\begin{equation}\label{specialeq:lemest1}
\|v_\tau\|_{(-m)}\leq C_m\tau^{-m}.
\end{equation}
If $\tilde{\varGamma}$ is the cone generated by
\[
\{ (x,w'(x)): x\in \supp\phi, \ \im w(x)=0\}
\]
then $\tau^k v_\tau \rightarrow 0$ in $\mathscr{D}_{\tilde{\varGamma}}'$ as $\tau \to \infty$,
hence $\tau^k A v_\tau \rightarrow 0$ in $C^\infty(\mathbb{R}^n)$ if $A$ is a
pseudo-differential operator with $W\! F(A)\cap\tilde{\varGamma}=\emptyset$.
\end{lem}
\noindent Here $\mathscr{D}_{\tilde{\varGamma}}'(X)=\{u\in \mathscr{D}'(X): W\! F(u)\subset \tilde{\varGamma}\}$,
equipped with the topology given by all the seminorms on $\mathscr{D}'(X)$ for the weak topology,
together with all seminorms of the form
\[P_{\phi, V, N}(u)=\sup_{\xi\in V}|\widehat{\phi u}(\xi)|(1+|\xi|)^N
\]
where $N\geq 0$, $\phi \in C_0^\infty(X)$, and $V\subset \mathbb{R}^n$ is a closed cone
with $(\supp\phi \times V)\cap \tilde{\varGamma}=\emptyset$. Note that
$u_j\rightarrow u$ in $\mathscr{D}_{\tilde{\varGamma}}'(X)$ is equivalent to
$u_j\rightarrow u$ in $\mathscr{D}'(X)$ and $Au_j\rightarrow Au$ in $C^\infty$ for
every properly supported pseudo-differential operator $A$ with $\tilde{\varGamma}\cap W\! F(A)=\emptyset$
(see the remark following ~\cite[Theorem $18.1.28$]{ho3}).
\begin{proof}
We observe that $\tau^k P^\ast v_\tau = \tau^k(P^\ast \phi)e^{i\tau w}\rightarrow 0$
in $C_0^\infty(\mathbb{R}^n)$ for any $k$ as $\tau\to\infty$, if $w$ and $\phi$
are chosen in the way given above. Hence
$\tau^k\|P^\ast v_\tau \|_{(m)}\rightarrow 0$ for any positive integers $k$ and $m$.
In view of \eqref{specialeq:defapprox1} and the fact that $d \re w\ne 0$
in the support of $\phi$ we can apply
~\cite[Lemma $26.4.15$]{ho4} to $v_\tau$. This immediately yields \eqref{specialeq:lemest1}
and also that $\tau^k v_\tau \rightarrow 0$ in $\mathscr{D}_{\tilde{\varGamma}}'$ as $\tau \to \infty$,
which proves the lemma.
\end{proof}

We are now ready to proceed with a tool that will be instrumental in proving Theorem \ref{mainthm1}. The
idea is based on techniques found in ~\cite{ho0}.

Let $R$ be the operator given by Theorem \ref{mainthm2}.
By assumption there exists a compactly based cone $K\subset T^\ast(\mathbb{R}^n)\smallsetminus 0$
such that the range of $R$ is microlocally contained in the range of $P$ at $K$.
If $N$ is the integer given by Definition \ref{defrange}, let $H(x)\in C_0^\infty(\mathbb{R}^n)$ and set
\begin{equation}\label{eq:htaudef}
h_\tau(x)=\tau^{-N}H(\tau x).
\end{equation}
Since $\hat{h}_\tau(\xi)=\tau^{-N-n}\hat{H}(\xi/\tau)$ it is clear that
for $\tau \geq 1$ we have $h_\tau\in H_{(N)}(\mathbb{R}^n)$ and $\|h_\tau\|_{(N)}\leq C\tau^{-n/2}$.
In particular, $\|h_\tau\|_{(N)}\leq C$ for $\tau \geq 1$
where the constant depends on $H$ but not on $\tau$.
Now denote by $I_\tau$
the integral
\begin{equation}\label{specialeq:integraldef}
I_\tau=\tau^n\int H(\tau x) R^\ast v_\tau(x) \, dx=\tau^{N+n}(R^\ast v_\tau, \overline{h_\tau}),
\end{equation}
where $R^\ast$ is the adjoint of $R$.
For any $\kappa$ we then have by the second equality and Lemma \ref{lemrange1} that
\begin{align*}
|I_\tau|&\le \tau^{N+n}\|h_\tau\|_{(N)} \|R^\ast v_\tau\|_{(-N)}\\
&\le C_{\kappa}\tau^{N+n}(\|P^*v_\tau\|_{(\nu)}+\|v_\tau\|_{(-N-\kappa-n)}+\|Av_\tau\|_{(0)})
\end{align*}
for some positive integer $\nu$ and properly
supported pseudo-differential operator $A$ with $W\! F(A)\cap K= \emptyset$.
By Lemma \ref{special:lemest1} this implies
\begin{equation}\label{specialeq:integraldef1}
|I_\tau | \leq C_\kappa\tau^{-\kappa}
\end{equation}
for any positive integer $\kappa$ if $\tau$ is sufficiently large.

Recall that $R(x,D_{x'})$ is a pseudo-differential operator in $x'$ depending on
$x_1$ as a parameter. Its symbol is given by the asymptotic sum
\[
\sigma_R(x,\xi')=r_{1}(x,\xi')+r_{0}(x,\xi')+\ldots
\]
where $r_{-j}(x,\xi')$ is homogeneous of degree $-j$ in $\xi'$.
The symbol of $R^\ast$ has the asymptotic expansion
\[
\sigma_{R^\ast}=\sum \partial_\xi^\alpha D_x^\alpha \overline{\sigma_R(x,\xi')}/\alpha !
\]
which shows that $R^\ast$ is also a pseudo-differential operator in $x'$ depending on $x_1$
as a parameter. If we sort the terms above with respect to homogeneity we can write
\begin{equation}\label{specialeq:asexpforadjoint}
\sigma_{R^\ast}=q_1(x,\xi')+q_{0}(x,\xi')+\ldots
\end{equation}
where $q_{-j}$ is homogeneous of order $-j$, $q_1(x,\xi')=\overline{r_1(x,\xi')}$
and
\[
q_{0}(x,\xi')=\overline{r_{0}(x,\xi')}+\sum_{k=2}^n \partial_{\xi_k}D_{x_k}\overline{r_{1}(x,\xi')}.
\]
A moments reflection shows that if all the terms in 
\eqref{specialeq:asexpforadjoint} have vanishing Taylor coefficients at some point $(x,\xi')$,
then the same must hold for $\sigma_R$.

Our goal is to show that if $q_{-j\,(\alpha)}^{(\beta)}(0,\xi^0)$ does not vanish for all $j\geq -1$ and
all $\alpha,\beta\in \mathbb{N}^n$, then \eqref{specialeq:integraldef1} cannot hold.
For this purpose, we introduce a total well-ordering $>_t$ on the Taylor
coefficients by means of an ordering of the indices $(j,\alpha,\beta)$ as follows.
\begin{dfn}\label{defordering}
Let $\alpha_i,\beta_i\in \mathbb{N}^n$ and $j_i\ge -1$ for $i=1,2$. We say that
\begin{align*}
q_{-j_1\,(\alpha_1)}^{(\beta_1)}(0,\xi^0) & >_t q_{-j_2\,(\alpha_2)}^{(\beta_2)}(0,\xi^0) \quad \text{if}\\
j_1+|\alpha_1|+|\beta_1| & > j_2+|\alpha_2|+|\beta_2|.
\end{align*}
To ``break ties'', we say that if $j_1+|\alpha_1|+|\beta_1| = j_2+|\alpha_2|+|\beta_2|$ then
\[
q_{-j_1\,(\alpha_1)}^{(\beta_1)}(0,\xi^0) >_t q_{-j_2\,(\alpha_2)}^{(\beta_2)}(0,\xi^0)
\quad \text{if }|\beta_2| > |\beta_1|.
\]
Note the reversed order. If also $|\beta_1|=|\beta_2|$ then
we use a monomial ordering on the $\beta$ index to ``break ties''.
Recall that this is any relation $>$ on $\mathbb{N}^n$
such that $>$ is a total well-ordering on $\mathbb{N}^n$ and
$\beta_1>\beta_2$ and $\gamma\in\mathbb{N}^n$ implies $\beta_1+\gamma>\beta_2+\gamma$.
Having come this far, the actual order turns out not to matter for the proof of
Theorem \ref{mainthm2}, but it will have bearing on the proof of Theorem \ref{mainthm1}.
Which monomial ordering we use on the $\beta$ index will not be important, but for completeness let
us choose lexiographic order since this will be used at a later stage in the definition.
Here we by lexiographic order refer to
the usual one, corresponding to the variables being ordered $x_1>\ldots >x_n$.
That is to say, if $\alpha_i\in\mathbb{N}^n, i=1,2$, then
$\alpha_1>_{lex}\alpha_2$ if, in the vector difference $\alpha_1-\alpha_2\in\mathbb{Z}^n$,
the leftmost nonzero entry is positive.
Thus, if $j_1+|\alpha_1|+|\beta_1| = j_2+|\alpha_2|+|\beta_2|$ and $\beta_1=\beta_2$,
then we first say that
\begin{equation}\label{eq:totalordeing1}
q_{-j_1\,(\alpha_1)}^{(\beta_1)}(0,\xi^0) >_t q_{-j_2\,(\alpha_2)}^{(\beta_2)}(0,\xi^0)
\quad \text{if }|\alpha_2| > |\alpha_1|
\end{equation}
and then
use lexiographic order on the $n$-tuples $\alpha$ 
to ``break ties'' at this stage.
Using
the lexiographic order on both multi-indices (separately)
we get
\[
q_1 <_t q_1^{(\varepsilon_n)} <_t\ldots <_t q_1^{(\varepsilon_1)} <_t q_{1 (\varepsilon_n)}
<_t \ldots <_t q_{1 (\varepsilon_1)} <_t q_0 <_t \ldots
\]
\end{dfn}
As indicated above we will prove Theorem \ref{mainthm2} by
a contradiction argument, so in the sequel we let $\kappa$ denote an integer such that
\begin{equation}\label{eq:defkappa}
j+|\alpha|+|\beta|<\kappa
\end{equation}
if $q_{-j\,(\alpha)}^{(\beta)}(0,\xi^0)$ is the first nonvanishing Taylor coefficient
with respect to the ordering $>_t$. Since $j\ge -1$ we will thus have $\kappa\ge 0$.

To simplify notation, we shall in what follows write $t$ instead of $x_1$ and $x$ instead of $x'$.
Then $v_\tau$ takes the form
\[v_\tau(t,x)=\phi(t,x)e^{i\tau w(t,x)},
\]
where
\begin{equation}\label{eqdefwt}
w(t,x)= x_{n-1}+i(t^2+x_1^2+\ldots + x_{n-2}^2+(x_{n-1}+it^2/2)^2)/2.
\end{equation}
We shall as before use the notation $\xi^0=(0,\ldots,0,1)\in \mathbb{R}^{n-1}$
when in this context.
To interpret the integral $I_\tau$ we will need a formula for how
$R^\ast(t,x,D)$ acts on the functions $v_\tau$.
This is given by the following lemma, where the parameter $t$ has been suppressed
to simplify notation.
\begin{lem}[{~\cite[Lemma $26.4.16$]{ho4}}]\label{speciallemaction}
Let $q(x,\xi)\in S^\mu(\mathbb{R}^{n-1}\times \mathbb{R}^{n-1})$,
let $\phi \in C_0^\infty (\mathbb{R}^{n-1})$, $w \in C^\infty (\mathbb{R}^{n-1})$,
and assume that $\im w>0$ except at a point
$y$ where $w'(y)=\eta \in \mathbb{R}^{n-1}\smallsetminus 0$ and $\im w''$ is positive definite. Then
\begin{equation}\label{specialeqaction} |q(x,D)(\phi e^{i\tau w} )-
\sum_{|\alpha|<k} q^{(\alpha)}(x, \tau \eta) (D-\tau\eta)^\alpha (\phi e^{i\tau w})/\alpha ! |
\leq C_k \tau^{\mu-k/2}
\end{equation} for $\tau>1$ and $k=1, 2 , \ldots$ .
\end{lem}
\noindent An inspection of the proof of ~\cite[Lemma $26.4.16$]{ho4} shows that
the result is still applicable if $\im w>0$ everwhere. This is also used without mention in
~\cite{ho4} when proving the necessity of condition ($\varPsi$). Thus the statement holds
if $\im w>0$ except \emph{possibly} at a point
$y$ where $w'(y)=\eta \in \mathbb{R}^{n-1}\smallsetminus 0$ and $\im w''$ is positive definite.
We will also use this fact, but we have refrained from
altering the statement of the lemma.

Note that if $q$ is homogeneous of degree $\mu$, then the sum in \eqref{specialeqaction} consists (apart from
the factor $e^{i\tau w}$) of terms which are homogeneous in $\tau$ of degree $\mu, \mu-1, \ldots$ .
The terms of degree $\mu$ are those in
\begin{equation}\label{eq:specialeqaction1}
\phi \sum q^{(\alpha)}(x,\tau \eta)(\tau w'(x)-\tau \eta)^\alpha /\alpha !
\end{equation}
which is the Taylor expansion
at $\tau \eta$ of $q(x,\tau w')$.
In this way one can give meaning to the expression $q(x,\tau w')$ even though
$q(x,\xi)$ may not be defined for complex $\xi$.
The terms
of degree $\mu-1$ where $\phi$ is differentiated are similarly
\[ \sum_{k=1}^{n-1} q^{(k)}(x,\tau w'(x)) D_k \phi
\]
where $q^{(k)}$
should be replaced by the Taylor expansion at
$\tau \eta$ representing the value at $\tau w'(x)$, as in \eqref{eq:specialeqaction1}.
In the present case we have
\[w_x'(t,x)-\xi^0=ix-(t^2/2)\xi^0,
\]
so the expression $q_{-j}(t,x,w_x'(t,x))$ is given meaning if it is replaced by a finite
Taylor expansion
\[\sum_{\beta}q_{-j}^{(\beta)}(t,x,\xi^0)(w_x'(t,x)-\xi^0)^\beta/|\beta|! 
\]
of sufficiently high order.

Using the classicality of $R^\ast$ we have
\[
\sigma_{R^\ast}(t,x,\xi)-\sum_{j=-1}^{M} q_{-j}(t,x,\xi)\in \varPsi_{\mathrm{cl}}^{-M-1}(\mathbb{R}^{n}),
\]
so there is a symbol $a\in S_{\mathrm{cl}}^{-M-1}(\mathbb{R}^{n}\times\mathbb{R}^{n-1})$ such that
\[
a(t,x,D)=R^\ast(t,x,D)-\sum_{j=-1}^{M} q_{-j}(t,x,D) \quad \text{mod }\varPsi^{-\infty}(\mathbb{R}^{n}).
\]
By \eqref{specialeq:defapprox1} and \eqref{eqdefwt} it is clear that $w$ satisfies the conditions of
Lemma \ref{speciallemaction}, so
\begin{align*}
a(t,x,D)
v_\tau & =
a(t, x, \tau \xi^0)  v_\tau
+\mathcal{O}( \tau^{
-M-3/2})\\
& =\tau^{
-M-1} a(t, x, \xi^0)  v_\tau
+\mathcal{O}( \tau^{
-M-3/2})
\end{align*}
which implies that $|a(t,x,D) v_\tau|\leq C\tau^{-M-1}$. If we for each $-1\leq j\leq M$ write
\begin{equation*} |q_{-j}(t,x,D)
v_\tau-
\sum_{|\alpha|<k_j} q_{-j}^{(\alpha)}(t, x, \tau \xi^0) (D_x-\tau\xi^0)^\alpha v_\tau/\alpha ! |
\leq C_{k_j} \tau^{-j-k_j/2}
\end{equation*}
with $k_j=2M-2j+1$, then
\begin{align*}
R^\ast(t,x,D) 
v_\tau & =
\sum_{j=-1}^{M}\sum_{|\alpha|<k_j}
q_{-j}^{(\alpha)}(t, x, \tau \xi^0) (D_x-\tau\xi^0)^\alpha v_\tau/\alpha !\\
&\phantom{=} +\mathcal{O}( \tau^{
-M-1/2}).
\end{align*}
Now recall the discussion above
regarding the homogeneity of the terms in \eqref{specialeqaction},
and choose $M\geq \kappa$, where $\kappa$ is an integer satisfying \eqref{eq:defkappa}. Then
\begin{align*}
R^\ast(t,x,D)  
v_\tau & = 
e^{i\tau w}
\sum_{j=-1}^{M} \ \sum_{|\alpha|\leq 2M-2j} q_{-j}^{(\alpha)}(t,x,\tau w_x'(t,x) )D^\alpha \phi  \\
&= 
e^{i\tau w}
\sum_{j=-1}^{M} \sum_{|\alpha|\leq 2M-2j} \tau^{-j-|\alpha|} q_{-j}^{(\alpha)}(t,x, w_x'(t,x) )D^\alpha \phi  \\
&= 
e^{i\tau w}
\sum_{J=-1}^M \tau^{-J}\lambda_J(t,x)
\end{align*} 
with an error of order $\mathcal{O}(\tau^{-\kappa-1/2})$, where
\begin{equation}\label{eqdeflambdaspecial}
\lambda_J(t,x)= \sum_{j+|\alpha|=J}q_{-j}^{(\alpha)}(t,x, w_x'(t,x) )D^\alpha \phi
\quad \text{for }j\geq -1.
\end{equation}
As before,
$q_{-j}^{(\alpha)}(t,x, w_x'(t,x) )$ should be replaced by
a finite Taylor expansion at $\xi^0$ of sufficiently high order
representing the value at $w_x'(t,x)$. 
In view of \eqref{specialeq:integraldef}, this yields
\begin{align*}
I_\tau & =\tau^n\int H(\tau t,\tau x) e^{i\tau w(t,x)} \Big( \sum_{J=-1}^\kappa \tau^{-J}\lambda_J(t,x)
+\mathcal{O}(\tau^{-\kappa-1/2})\Big)dt \, dx.
\end{align*}

After the
change of variables
$(\tau t, \tau x) \mapsto (t,x)$
we find that
\begin{equation}\label{eqconv1}
\begin{aligned}I_\tau & 
=\int H( t, x) e^{i\tau w(t/\tau,x/\tau)} \Big( \sum_{J=-1}^\kappa \tau^{-J}\lambda_J(t/\tau,x/\tau)
\\
& \phantom{=}
\qquad \qquad \qquad \qquad +\mathcal{O}(\tau^{-\kappa-1/2})\Big)dt \, dx.
\end{aligned}
\end{equation}
To illustrate how we will proceed to prove Theorem \ref{mainthm2} by contradiction, let
us for the moment assume that $q_1(0,0,\xi^0)\neq 0$, where
$\xi^0=(0,\ldots,0,1)\in\mathbb{R}^{n-1}$.
Since
\begin{equation}\label{eq:lambdaminus1}
\begin{aligned}
\lambda_{-1}(t/\tau,x/\tau)& =\phi(t/\tau,x/\tau)
\sum_{\beta}q_1^{(\beta)}(t/\tau,x/\tau,\xi^0)\\
&\phantom{=} \times (w_x'(t/\tau,x/\tau)-\xi^0)^\beta/|\beta|!
\end{aligned}
\end{equation}
where
\begin{equation}\label{problembarn}
w_x'(t/\tau,x/\tau)-\xi^0=ix/\tau-(t^2/(2\tau^2))\xi^0=\mathcal{O}(\tau^{-1}),
\end{equation}
and \eqref{eqdefwt} implies
that
$\tau w(t/\tau,x/\tau)  \to x_{n-1}$ as $\tau \to \infty$, we obtain
\[\lim_{\tau \to \infty}I_\tau /\tau=\int H(t,x) e^{ix_{n-1}}\phi(0,0)q_1(0,0,\xi^0)dt \, dx.
\]
Since we may
choose $\phi\neq0$ at the origin, the limit above will then not be
equal to $0$ for a suitable choice of $H$. 
However, this contradicts
\eqref{specialeq:integraldef1}.

Now assume that $\partial_t^{k_0} q_{-j_0 (\alpha_0)}^{(\beta_0)}(0,0,\xi^0)$ is the first
nonvanishing Taylor coefficient with respect to the ordering $>_t$, and let
\begin{equation}\label{eq:orderfirsttaylorcoef1}
m=j_0+k_0+|\alpha_0|+|\beta_0|
\end{equation}
so that $m<\kappa$ by \eqref{eq:defkappa}.
Note that $\alpha_0,\beta_0\in\mathbb{N}^{n-1}$ and that the integer
$k_0$ accounts for derivatives in $t$ while there is no corresponding term for
derivatives in the Fourier transform of $t$ since the $q_{-j}$ are independent
of this variable. Note also that since $j_0$ is permitted to be $-1$, we
have $0\le k_0,|\alpha_0|,|\beta_0|\le m+1$.

To use our assumption we will for each term $q_{-j}^{(\beta+\gamma)}(t/\tau,x/\tau,\xi^0)$
in the
Taylor expansion of $q_{-j}^{(\gamma)}(t/\tau,x/\tau,w_x'(t/\tau,x/\tau))$
(as it appears in \eqref{eqdeflambdaspecial}) at $\xi^0$
need to consider Taylor expansions
in $t$ and $x$ at the origin.
Note that for given $j$ and $\gamma$, it suffices to consider finite
Taylor expansions of $q_{-j}^{(\gamma)}$ of order $\kappa-j-|\gamma|$
by \eqref{eqconv1} and \eqref{problembarn}.
For each $j$ and $\gamma$ we thus write
\begin{multline*}
q_{-j}^{(\gamma)}(t/\tau,x/\tau,w_x'(t/\tau,x/\tau)) = \sum_{ k+|\alpha|+|\beta|\leq \kappa-j-|\gamma|}
(\partial_t^k q_{-j \,(\alpha)}^{(\beta+\gamma)})(0,0,\xi^0)\\
\times \tau^{-k-|\alpha|}t^k x^\alpha ( w_x'(t/\tau,x/\tau)-\xi^0)^\beta/(k! |\alpha|! |\beta|!)
+\mathcal{O}(\tau^{-\kappa-1+j+|\gamma|}),
\end{multline*}
where $( w_x'(t/\tau,x/\tau)-\xi^0)^\beta$ should be interpreted by means of \eqref{problembarn}.
As we shall see, the term $(t^2/(2\tau^2))\xi^0$ will not pose any problem,
since it is $\mathcal{O}(\tau^{-2})$.
We have
\begin{multline*}
\lambda_{J}(t/\tau,x/\tau) =\sum_{j+|\gamma|=J}\ \sum_{ k+|\alpha|+|\beta|\leq \kappa-J}
(\partial_t^k q_{-j \,(\alpha)}^{(\beta+\gamma)})(0,0,\xi^0)D^\gamma \phi(t/\tau,x/\tau)\\
\times \tau^{-k-|\alpha|}t^k x^\alpha ( w_x'(t/\tau,x/\tau)-\xi^0)^\beta/(k! |\alpha|! |\beta|!)
+\mathcal{O}(\tau^{-\kappa-1+J})
\end{multline*}
where $-1\le j \le J$. If we are only interested in terms of
order $\tau^{-m}$ in \eqref{eqconv1}, we can
use the assumption that $\partial_t^k q_{-j \,(\alpha)}^{(\beta+\gamma)}(0,0,\xi^0)=0$
for all $-1\leq j+k+|\alpha|+|\beta|+|\gamma|<m$
to let the term $(t^2/(2\tau^2))\xi^0$ from \eqref{problembarn} be
absorbed by the error term in the expression above.
This yields
\begin{multline*}
\sum_{J=-1}^m \tau^{-J}
\lambda_{J}(t/\tau,x/\tau) = \sum_{ j+k+|\alpha|+|\beta|+|\gamma|=m}
(\partial_t^k q_{-j \,(\alpha)}^{(\beta+\gamma)})(0,0,\xi^0)\\
\times D^\gamma \phi(t/\tau,x/\tau) \tau^{-m}t^k x^\alpha (ix)^\beta/(k! |\alpha|! |\beta|!)
+\mathcal{O}(\tau^{-m-1}),
\end{multline*}
where we use $J=j+|\gamma|$ together with the fact that we get a factor $\tau^{-|\beta|}$ from
$( w_x'(t/\tau,x/\tau)-\xi^0)^\beta$ by \eqref{problembarn}. Thus,
\begin{multline*}
\lim_{\tau \to \infty}\tau^mI_\tau =\int H(t,x) e^{ix_{n-1}} \Big\{
\sum_{ j+k+|\alpha|+|\beta|+|\gamma|=m}
 t^k x^\alpha (ix)^\beta\\
\times (\partial_t^k q_{-j \,(\alpha)}^{(\beta+\gamma)})(0,0,\xi^0)D^\gamma \phi(0,0) /(k! |\alpha|! |\beta|!)
\Big\}
dt \, dx.
\end{multline*}
Now choose $\phi$ such that $D^{\beta_0} \phi(0,0)=1$, but $D^\gamma \phi(0,0)=0$ for all other $\gamma$
such that $|\gamma|\leq |\beta_0|$.
This is possible by the discussion following \eqref{specialeq:defapprox}. By \eqref{eq:orderfirsttaylorcoef1}
and our choice of the ordering
$>_t$, we have
$\partial_t^k q_{-j \,(\alpha)}^{(\beta+\beta_0)}(0,0,\xi^0)=0$ for all
$\beta$ such that $|\beta|>0$ as long as
$j+k+|\alpha|+|\beta|+|\beta_0|=m$.
Hence, with this choice of $\phi$, the last expression
takes the form
\begin{equation}\label{eq:orderfirsttaylorcoef2}
\begin{aligned}
\lim_{\tau \to \infty}\tau^mI_\tau =\int & H(t,x) e^{ix_{n-1}} \Big\{ 
\sum_{ j+k+|\alpha|+|\beta_0|= m}
 t^k x^\alpha \\
& \ \ \times (\partial_t^k q_{-j \,(\alpha)}^{(\beta_0)})(0,0,\xi^0) /(k! |\alpha|!)
\Big\}
dt \, dx,
\end{aligned}
\end{equation}
where as usual $j$ is allowed to be $-1$ so that $j\in[-1,m-|\beta_0|]$
in \eqref{eq:orderfirsttaylorcoef2}.
Now some of the Taylor coefficients in \eqref{eq:orderfirsttaylorcoef2}
may be zero, in particular, the expression may
well contain Taylor coefficients that
preceed $\partial_t^{k_0} q_{-j_0 \,(\alpha_0)}^{(\beta_0)}(0,0,\xi^0)$, and
those are by assumption zero.
However, we claim that if at least one of the Taylor coefficients above
are nonzero, then we may choose $H$ so that the limit is nonzero. Indeed, if that were
not the case then the expression within brackets in \eqref{eq:orderfirsttaylorcoef2}
would be a polynomial with infinitely many zeros, and thus it would have to have vanishing
coefficients. Since this violates our assumption, we conclude that the limit is nonzero.
However, this contradicts \eqref{specialeq:integraldef1}, which proves Theorem \ref{mainthm2}.

\section{Proof of Theorem \ref{mainthm1}}

\noindent
In this section we shall give the proof of Theorem \ref{mainthm1}, using
ideas taken from ~\cite{ho0} together with the approach used to prove
~\cite[Theorem $26.4.7'$]{ho4}. As in the previous section, we aim to use Lemma \ref{lemrange1} to estimate
the operator $R(x,D_{x'})$ on approximate solutions of the equation $P^\ast v=0$,
concentrated near
\begin{equation}\label{eq:defgammaprime}
\varGamma'=\{(x_1,x',0,\xi'): x_1\in I' \}\subset T^*(\mathbb{R}^n)\smallsetminus 0.
\end{equation}
The proofs will be similar, but the situation is more complicated
now which will affect the construction of the approximate solutions.
We will also have to make some adjustments to the proof of
~\cite[Theorem $26.4.7'$]{ho4} to make it work, so a lot of the details will have to be
revisited. Note that our approximate solutions will also differ slightly from the ones used to prove
~\cite[Theorem $26.4.7'$]{ho4}, so
although we will refer directly to results in ~\cite{ho4} whenever possible, the formulation
of some of these results will be affected. For a more complete description of the
approximate solutions, we refer the reader to ~\cite{ho40} or ~\cite{ho4} where their construction is
carried out in greater detail. When proving Theorem \ref{mainthm1} we may without loss of
generality assume that $x'=0,\xi'=\xi^0$ in \eqref{eq:defgammaprime}. In accordance with the notation in
the proof of Theorem \ref{bigthm1}, we shall therefore
throughout this section refer to $\varGamma'$ simply by $\varGamma$, and we will let $I'=[a_0,b_0]$.

To simplify notation we shall in what follows write $t$ instead of $x_1$ and $x$ instead of $x'$.
If $N$ is the integer given by Definition \ref{defrange}, and $n$ is the dimension,
the approximate solutions $v_\tau$ will be taken of the form
\begin{equation}\label{apsol} v_\tau (t,x)= \tau^{N+n}e^{i\tau w (t,x)}\sum_0^M \phi_j(t,x)\tau^{-j}.
\end{equation}
Here
$\phi_0, \phi_1, \ldots$ are amplitude functions,
and $w$ is a phase function that should satisfy the eiconal equation
\begin{equation}\label{eiconaleq1}
\partial w / \partial t - i f(t,x , \partial w / \partial x)=0
\end{equation} 
approximately, where $f$ is the imaginary part of the principal symbol of $P$.
We take $w$ of the form
\begin{equation}\label{eq:defw}
w(t,x)=w_0(t)+\langle x-y(t),\eta (t) \rangle +\sum_{2\leq |\alpha|\leq M} w_\alpha (t)(x-y(t))^\alpha /|\alpha|!
\end{equation}
where $M$ is a large integer to be determined later, and $x=y(t)$ is a smooth
real curve. When discussing the functions
$w_\alpha$ we shall permit us to use the
notation $\alpha=(\alpha_1, \ldots , \alpha_s)$ for a sequence of $s=|\alpha|$ indices between $1$ and the dimension
$n-1$ of the $x$ variable. $w_\alpha$ will be symmetric in these indices.
If we take $\eta(t)$ to be real valued and make sure the matrix $(\im w_{j k})$ is positive definite then
$\im w$ will have a strict minimum when $x=y(t)$ as a function of the $x$ variables.

On the curve $x=y(t)$ the eiconal equation \eqref{eiconaleq1} is reduced to
\begin{equation}\label{eq:wzero1}
w_0'(t)=\langle y'(t), \eta(t) \rangle + i f(t, y(t) , \eta(t) ),
\end{equation}
which is the only equation where $w_0$ occurs. Hence it can be used to determine
$w_0$ after $y$ and $\eta$ have been chosen. In particular
\begin{equation}\label{eq:wzero2}
d \im w_0 (t) / d t = f( t, y(t), \eta(t) ).
\end{equation}

In the proof of Theorem \ref{mainthm2} we could solve the corresponding
eiconal equation explicitly. Here this is not possible, so
our goal will instead be to make \eqref{eiconaleq1} valid apart from an error of
order $M+1$ in $x-y(t)$. Note that $f(t,x,\xi)$ is not defined for complex $\xi$, but since
\[\partial w(t,x)/\partial x_j-\eta_j(t)=\sum w_{\alpha , j}(t)(x-y(t))^\alpha / |\alpha|!
\]
\eqref{eiconaleq1}) is given meaning if $f(t,x,\partial w / \partial x)$ is replaced by the finite Taylor
expansion
\begin{equation}\label{eq:ftaylorappox1}
\sum_{|\beta|\leq M} f^{(\beta)}(t,x,\eta (t))(\partial w(t,x) / \partial x - \eta(t))^\beta / |\beta|!.
\end{equation}
To compute the coefficient of $(x-y(t))^\alpha$ in \eqref{eq:ftaylorappox1}
we just have to consider the terms with $|\beta|\leq |\alpha|$.
Since
\begin{equation*}
\begin{aligned}
\partial w / \partial t
& = w_0'-\langle y', \eta \rangle + \langle x-y,\eta'\rangle
+\sum_{2\le |\alpha|\le M} w_\alpha ' (t) (x-y)^\alpha / |\alpha|! \\
& \phantom{= \ }-\sum_k \sum_{1\leq |\alpha| \leq M-1} w_{\alpha, k}
(t)(x-y)^\alpha d y_k / d t / |\alpha|!,
\end{aligned}
\end{equation*}
the first order terms in the equation \eqref{eiconaleq1} give
\begin{equation}\label{eq:firstorderterms1}
\begin{aligned}
d \eta_j / d t & - \sum_k w_{j k} (t) d y_k / d t \\
& = i ( f_{(j)}(t, y , \eta ) 
+ \sum_k  f^{(k)} (t,y,\eta) w_{j k}(t)).
\end{aligned}
\end{equation}
Note that this is a system of $2n$ equations
\begin{equation}\tag*{$(\ref{eq:firstorderterms1})'$}
d \eta_j / d t - \sum_k \re w_{j k}(t) d y_k / d t =
-\sum_k \im w_{j k}(t) f^{(k)}(t, y,\eta),
\end{equation}
\begin{equation}\tag*{$(\ref{eq:firstorderterms1})''$}
\sum_k \im w_{j k}(t) d y_k / d t=-f_{(j)}(t,y,\eta)-\sum_k \re w_{j k}(t) f^{(k)}(t,y,\eta),
\end{equation}
since $y$ and $\eta$ are real, and
under the assumption that $\im w_{j k}$ is positive definite these equations
can be solved for $d y / d t$
and $d \eta / d t$. We observe that at a point where $f=d f = 0$ they just mean
that $d y / d t = d \eta / d t =0$.

When $2\leq |\alpha| \leq M$ we obtain a differential equation
\begin{equation}\label{eq:higherorderterms1}
d w_\alpha / dt-\sum_k w_{\alpha , k} d y_k / d t= F_\alpha(t,y,\eta, \{w_\beta\})
\end{equation}
from \eqref{eiconaleq1}. Here $F_\alpha$ is a linear combination of the derivatives of $f$
of order $|\alpha|$ or less, multiplied with polynomials in $w_\beta$ with $2\leq |\beta|\leq
|\alpha|+1$. Of course, when $|\alpha|=M$ the sum on the left-hand side of \eqref{eq:higherorderterms1}
should be dropped, and $\beta$ should satisfy $|\beta|\leq |\alpha|$ instead. Altogether
$(\ref{eq:firstorderterms1})'$, $(\ref{eq:firstorderterms1})''$ and \eqref{eq:higherorderterms1}
form a quasilinear system of differential equations with as many equations as unknowns.
Hence we have local solutions with prescribed initial data. 
According to ~\cite[pp. $105-106$]{ho4}
we can find a $c>0$ such that the equations \eqref{eq:firstorderterms1} and
\eqref{eq:higherorderterms1} with initial data
\begin{equation}\label{eq:indata1}
w_{j k}= i \delta_{j k}, \quad w_\alpha=0 \quad \text{when $2< |\alpha|\leq M$, $t=(a_0+b_0)/2$}
\end{equation}
\begin{equation}\label{eq:indata2}
y=x, \quad \eta=\xi \quad \text{when $t=(a_0+b_0)/2$}
\end{equation}
have a unique solution in $(a_0-c,b_0+c)$ for all $x$, $\xi$ with
$|x|+|\xi-\xi^0|<c$. (Here $\delta_{j k}$ is the Kronecker $\delta$.) Moreover,
\begin{enumerate}
\item[i)] $(\im w_{j k}-\delta_{j k}/2)$ is positive definite,
\item[ii)] the map
\[ (x,\xi,t)\mapsto (y,\eta,t); \quad |x|+|\xi-\xi^0|<c, \ a_0-c<t<b_0+c
\]
\end{enumerate}
is a diffeomorphism.

In the range $X_c$ of the map ii) we let $v$ denote the image of the vector field
$\partial / \partial t$ under the map. Thus $v$ is the tangent vector field of the
integral curves, and when $f=df=0$ we have $v=\partial / \partial t$. By assumption
$f=0$ implies $\partial f / \partial t\leq 0$ in a neighborhood of $\varGamma$ (see \eqref{eq:mainthm1}), so if $c$
is small enough this also holds in $X_c$. An application of ~\cite[Lemma $26.4.11$]{ho4}
now yields that $f$ must have a change of sign from $-$ to $+$ along
an integral curve of $v$ in $X_c$, for otherwise there would be no such sign change
for increasing $t$ and fixed $(x,\xi)$, and that contradicts the hypothesis in Theorem
\ref{mainthm1}. By \eqref{eq:wzero2} this means that $\im w_0(t)$ will start decreasing and
end increasing, so the minimum is attained at an interior point. We can normalize the minimum
value to zero and have then for a suitable interval of $t$ that $\im w_0>0$ at the end points
and $\im w_0=0$ at some interior point.
Since $\re w_0$ is given by \eqref{eq:wzero1} we can
at this interior point also normalize the value of $\re w_0$ to zero. This completes the proof
of ~\cite[Lemma $26.4.14$]{ho4}. However, in order to prove Theorem \ref{mainthm1}
when $a_0<b_0$ we shall need the following stronger result.

\begin{lem}\label{lem26.4.14}
Assume that the hypotheses of Theorem \ref{mainthm1} are fulfilled, the variables being denoted
$(t,x)$ now. Then given $M\in\mathbb{N}$ we can find
\begin{enumerate}
\item[i)] a curve $t\mapsto (t, y(t), 0 , \eta(t))
\in \mathbb{R}^{2n}$, $a' \leq t \leq b'$ as close to $\varGamma$ as desired,
\item[ii)] $C^\infty$ functions $w_\alpha (t)$, $2\leq |\alpha| \leq M$, with $(\im w_{jk}-\delta_{jk}/2)$
positive definite when $a'\leq t \leq b'$,
\item[iii)] a function $w_0(t)$ with $\im w_0(t)\geq 0$, $a'\leq t \leq b'$, $\im w_0(a')>0$,
$\im w_0(b')>0$ and $\re w_0(c')=\im w_0(c')=0$ for some $c'\in (a', b')$
\end{enumerate}
such that \eqref{eq:defw}
is a formal solution to \eqref{eiconaleq1} with an error of order $\mathcal{O}(|x-y(t)|^{M+1})$.
If $a_0<b_0$ then $\mathrm{iii)}$ can be improved in the sense that if
$\varrho\ge 0$ is the number given by Theorem \ref{mainthm1},
then we can for any $\varepsilon>\varrho$ find
\begin{enumerate}
\item[iii)$'$] a function $w_0(t)$ with $\im w_0(t)\geq 0$, $a'\leq t \leq b'$, $\im w_0(a')>0$,
$\im w_0(b')>0$ and $\re w_0(t)=\im w_0(t)=0$ for all $t\in [a_0+\varepsilon, b_0-\varepsilon]$.
\end{enumerate}
\end{lem}
\begin{proof}
In view of ~\cite[Lemma $26.4.14$]{ho4} we only need to prove iii)$'$.

Let $\varepsilon>\varrho$, and let $I_\varepsilon=[a_0+\varepsilon, b_0-\varepsilon]$.
By the hypotheses of Theorem \ref{mainthm1}, there is a neighborhood $\mathcal{U}$ of
\[\varGamma_\varepsilon=\{(t,0,0,\xi^0 ) : t\in I_\varepsilon\}
\]
where $f$ vanishes identically. Take $\delta>0$ sufficiently small so that
\[t\in I_\varepsilon, \ |x|+|\xi-\xi^0|<\delta \quad \Longrightarrow
\quad (t,x,0,\xi)\in \mathcal{U}.
\]
As above we can find $c>0$ such that
the equations \eqref{eq:firstorderterms1} and
\eqref{eq:higherorderterms1} with initial data \eqref{eq:indata1} and \eqref{eq:indata2}
have a unique solution in $(a_0-c,b_0+c)$ for all $x$, $\xi$ with
$|x|+|\xi-\xi^0|<c$. Since
the map
\[ (x,\xi,t)\mapsto (y,\eta,t); \quad |x|+|\xi-\xi^0|<c, \ a_0-c<t<b_0+c
\]
is a diffeomorphism, we can choose $c$ small enough so that
if $(y,\eta,t)$ is in the range $X_c$ of this map, then
$|y|+|\eta-\xi^0|<\delta$. As we have seen, 
$f$ must change sign from $-$ to $+$
along an
integral curve of $v$ in $X_c$ if $c$ is small enough,
where in $X_c$ we denote by $v$ the image of the vector field $\partial / \partial t$ under the map.
Let this integral curve be given by
\[\gamma(t)=(t,y(t),0,\eta(t))\in \mathbb{R}^{2n}, \ a'\leq t \leq b', 
\]
for some choice of $a'$, $b'$ such that $a_0-c<a'$, $b'<b_0+c$ and
\[f(a',y(a'),\eta(a'))<0<f(b',y(b'),\eta(b')).
\]
Recall that at a point where $f=d f =0$ the equations
$(\ref{eq:firstorderterms1})'$ and $(\ref{eq:firstorderterms1})''$
imply that $d y / d t = d \eta / d t = 0$. Since $f$ vanishes identically
on $\gamma$ for $t\in I_\varepsilon$ and the function $w_0$ is determined by
\eqref{eq:wzero1}, this proves the lemma after a suitable normalization.
\end{proof}
Note that
if $\varGamma$ is a point then by
Lemma \ref{lem26.4.14}
we can obtain a sequence $\{\gamma_j\}$ of curves
\[
\gamma_j(t)=(t, y_j(t),0,\eta_j(t)), \quad a_j'\leq t \leq b_j',
\]
approaching $\varGamma$ which implies that
at $t=c_j'$ we have
\[
(c_j', y_j(c_j'),0,\eta_j(c_j'))\rightarrow \varGamma \quad \text{as } j\rightarrow \infty
\]
in $T^\ast(\mathbb{R}^n)\smallsetminus 0$, where $c_j'$ is the point where
$\re w_{0 j}=\im w_{0 j}=0$. 
Similarly, if $\varGamma$ is an interval
and $\varrho\ge 0$ is the number given by Theorem \ref{mainthm1},
then for any point $\omega$ in the interior of $\varGamma_{\varrho}$
we can use Lemma \ref{lem26.4.14}
to obtain a sequence $\{\gamma_j\}$ of curves
approaching $\varGamma$ and a sequence $\{w_{0j}\}$ of functions
such that for each $j$ there exists a point $\omega_j\in\gamma_j$
with $\omega_j=\gamma_j(t_j)$
which can be chosen so that $\re w_{0 j}(t_j)=\im w_{0 j}(t_j)=0$ and
$\omega_j\rightarrow \omega$ as $j\to \infty$.
This will be crucial in proving Theorem \ref{mainthm1}. Our strategy is to show
that all the terms in the asymptotic sum of the symbol of $R$ have vanishing Taylor coefficients at
$\omega_j$, or at $(c_j', y_j(c_j'),0,\eta_j(c_j'))$ when $\varGamma$ is a point. Theorem \ref{mainthm1}
will then follow by continuity. In what follows we will suppress the index $j$
to simplify notation.

Let $K$ and $\varOmega$ be the cones given by Theorem \ref{mainthm1}, and suppose that the function $w$
given by \eqref{eq:defw} is a formal solution
to \eqref{eiconaleq1} with an error of order $\mathcal{O}(|x-y(t)|^{M+1})$ in
a neighborhood $Y$ of
\[\{(t,0): a_0\leq t \leq  b_0 \}\subset \mathbb{R}^n
\]
with $K\subset T^*(Y)$, such that $\im w>0$ in $Y$ except on a compact non-empty subset $T$ of the curve $x=y(t)$,
with $(t_0,y(t_0))\in T$ and $w =0$ on $T$. We want to show that
all the terms in the asymptotic sum of the symbol of $R$ have vanishing Taylor coefficients at
$(t_0,y(t_0),0,\eta(t_0))$.
By part i) of Lemma \ref{lem26.4.14}
we can choose $w$ so that
\begin{equation}\label{eq:gammazero}
\varGamma_0=\{ (t,x,\partial w(t,x) / \partial t , \partial w(t,x) / \partial x ): (t,x)\in T \}
\end{equation}
is contained in $\varOmega$. This is done to ensure that if
$A$ is a given pseudo-differential operator with
wavefront set contained in the complement
of $K$, then $W\! F(A)$ does not meet the cone generated by $\varGamma_0$.

We now turn our attention to the amplitude functions $\phi_j$.
With the exception of $\phi_0$ which will be of great interest to us,
we will not be very thorough in describing them. Suffice it to say that these functions can be chosen in
such a way that if 
$P^\ast$ is the adjoint of $P$ then
\begin{equation}\label{P^*est1}\|P^\ast v_\tau\|_{(\nu)} \leq C\tau^{N+n+\nu+(1-M)/2}
\end{equation}
where $M$ is the number given by \eqref{apsol}. The procedure begins
by setting
\[ \phi_0(t,x)= \sum_{|\alpha|<M} \phi_{0\alpha}(t)(x-y(t))^\alpha
\]
with $y(t)$ as above, and having $\phi_{0\alpha}$ satisfy a certain linear system of ordinary differential equations
\begin{equation}\label{eq:constructphis}
D_t \phi_{0\alpha} + \sum_{|\beta|<M} a_{\alpha \beta}\phi_{0\beta} =0.
\end{equation}
In the same way we then successively choose $\phi_j$ and obtain
\eqref{P^*est1}. The precise details can be found in ~\cite[pp. $87-89$]{ho40},
or in ~\cite[pp. $107-110$]{ho4}.
Note that we
for any positive integer $J<M$ can solve the equations 
that determine $\phi_0$ so that at the point $(t_0,y(t_0))\in T$ we have
$D_x^\alpha \phi_0(t_0,y(t_0))=0$ for all $|\alpha|\leq J$ except for one index $\alpha$, $|\alpha|=J$.
This will be important later on.
Note also that the estimate \eqref{P^*est1} is not affected if the functions $\phi_j$
are multiplied by a cutoff function in $C_0^\infty (Y)$ which is $1$ in a neighborhood of $T$.
Since the $\phi_j$
will be irrelevant outside of $Y$ for large $\tau$ by construction,
we can in this way choose them to be supported in $Y$ so that $v_\tau \in C_0^\infty(Y)$.

Having completed the construction of the approximate solutions, we
are now ready to start to follow the proof of Theorem \ref{mainthm2}.
To get the estimates for the
right-hand side of \eqref{rangeeq1} when $v$ is an approximate solution, we shall need
the following two results. The first, corresponding to Lemma
\ref{special:lemest1}, is taken from ~\cite{ho4}.
Observe that here it is stated for our approximate
solutions which differ from those in ~\cite{ho4} by a factor of $\tau^{N+n}$, which explains the
difference in appearance.
Note also that although we will not use the lower bound for the approximate solutions,
that estimate is included so as not to alter the statement. 
\begin{lem}[{~\cite[Lemma $26.4.15$]{ho4}}]\label{lemest1}
Let $X\subset \mathbb{R}^n$ be open,
and let $v_\tau$ be defined by \eqref{apsol} where $w\in C^\infty(X)$, $\phi_j \in C_0^\infty
(X)$, $\im w \geq 0$ in $X$ and $d \re w \neq 0$. For any positive integer $m$ we then have
\begin{equation}\label{apsolestm}
\|v_\tau\|_{(-m)}\leq C\tau^{N+n-m}, \qquad \tau>1.
\end{equation}
If $\im w(t_0,x_0)=0$ and $\phi_0(t_0,x_0)\ne 0$ for some $(t_0,x_0)\in X$ then
\[
\|v_\tau\|_{(-m)}\ge c\tau^{N+n/2-m},\quad \tau>1,
\]
for some $c>0$.
If $\tilde{\varGamma}$ is the cone generated by
\[
\{ 
(t,x,\partial_t w(t,x) , \partial_x w(t,x) ):
(t,x)\in  \bigcup_{j} \supp\phi_j ,
\ \im w(t,x)=0 \}
\]
then $\tau^kv_\tau \to 0$ in $\mathscr{D}_{\tilde{\varGamma}}'$ as $\tau \to \infty$,
hence $\tau^k A v_\tau \to 0$ in $C^\infty (\mathbb{R}^n)$, if $A$ is a pseudo-differential operator
with $W\! F(A)\cap \tilde{\varGamma}=\emptyset$, and $k$ is any real number.
\end{lem}
\begin{prop}\label{intest1} Assume
that the hypotheses of Theorem \ref{mainthm1} are fulfilled, the variables being denoted
$(t,x)$ now, and let $v_\tau$ be given by \eqref{apsol}, where
$w\in C^\infty(Y)$, $\phi_j\in C_0^\infty(Y)$, $\im w\geq 0$ in $Y$ and $d \re w\neq 0$.
Here $Y$ is a neighborhood of $\{(t,0): a_0\leq t \leq b_0\}$ such that 
$K\subset T^\ast(Y)$. Let $H(t,x)\in C_0^\infty(\mathbb{R} \times \mathbb{R}^{n-1})$
and set
\begin{equation}\label{smallh} h_\tau (t,x) = \tau^{-N} H ( \tau(t-t_0),\tau(x-y(t)) ),
\end{equation}
where $N$ is the positive integer given by Definition \ref{defrange}
for the operators $R$ and $P$ in Theorem \ref{mainthm1}.
Then $h_\tau \in H_{(N)}(\mathbb{R}^n)$ for all $\tau\geq 1$ and
$\|h_\tau\|_{(N)}\leq C$ where the constant depends on $H$ but not on $\tau$. Furthermore,
if $M$ is the integer given by the definition of $v_\tau$
in \eqref{apsol}
so that
\eqref{P^*est1} holds, and $I_\tau$ is the integral
\begin{equation}\label{intdef}I_\tau = ( R^\ast v_\tau , \overline{h_\tau} )
\end{equation}
where $R^\ast$ is the adjoint of $R(t,x,D)$,
then for any positive integer $\kappa$ there exists a constant $C$ such that
$|I_\tau| \leq C\tau^{-\kappa}$ if $M=M(\kappa)$ is sufficiently large.
\end{prop}
\begin{proof}
In Section \ref{sec:proofmainthm2}, one easily obtains a formula
for the Fourier transform of the corresponding function $h_\tau$ (see
\eqref{eq:htaudef} on page \pageref{eq:htaudef}) which yields the
estimates needed to show that $h_\tau\in H_{(N)}$. Here we shall instead use
the equality
\[\iint |h_\tau (t,x)|^2 d t \, d x=\tau^{-2N}\iint |H(\tau(t-t_0),\tau (x-y(t)))|^2 dt \, dx
\]
which
shows that if $\tau \geq 1$ then $D_t^j  D_x^\alpha h_\tau \in L^2(\mathbb{R}^n)$ for all
$(j, \alpha)\in \mathbb{N} \times \mathbb{N}^{n-1}$ such that $j+|\alpha|\leq N+[n/2]$.
Hence, by using the equivalent norm on $H_{(N)}(\mathbb{R}^n)$ given by
\[ \|h_\tau\|_{(N)}=\sum_{j+|\alpha|\leq N}\|D_t^j D_x^\alpha h_\tau\|_{(0)},
\]
we find that $\{h_\tau\}_{\tau\ge 1}$ is a bounded one parameter family in
$H_{(N)}(\mathbb{R}^n)$,
which proves the first assertion of the proposition.

To prove the second part, let $\kappa$ be an arbitrary positive integer,
and let $\nu$ be the positive integer given by Lemma \ref{lemrange1}
(applied to the operator $R$ instead of $Q$)
so that \eqref{rangeeq1} holds for the choice of semi-norm $\|P^\ast v\|_{(\nu)}$ in the
right-hand side. If we choose
\begin{equation}\label{eqM}
(1-M)/2 \leq -N -n -\nu -\kappa,
\end{equation}
and recall \eqref{P^*est1}, then 
\begin{equation}\label{P^*est3}
\| P^* v_\tau \|_{(\nu)} \leq C\tau^{-\kappa}.
\end{equation}
Since $\supp H$ is compact, we can find a bounded open ball containing $\supp h_\tau$ for all
$\tau \geq 1$. Hence
$h_\tau \in H_{(N)}(\mathbb{R}^n)$ has compact support
and $v_\tau \in C_0^\infty (Y)$ so the result now follows by the estimate \eqref{rangeeq4}
together with Lemma \ref{lemest1}.
\end{proof}

To shorten the notation we will from now on assume
that $t_0=0$, so that $w(0,y(0))=0$.
As in the proof of Theorem \ref{mainthm2} it suffices to show that
all terms in the asymptotic expansion of the symbol of $R^\ast$,
given by
\begin{equation*}\label{eq:asexpforadjoint}
\sigma_{R^\ast}=q_1(t,x,\xi)+q_{0}(t,x,\xi)+\ldots
\end{equation*}
with $q_j$ homogeneous of degree $j$ in $\xi$, have vanishing Taylor coefficients
at $(0,y(0),\eta(0))$.
The method will be
to argue by contradiction that if not, then Proposition \ref{intest1}
does not hold. Therefore, let us assume that
$\partial_t^{k_0} q_{-j_0 (\alpha_0)}^{(\beta_0)}(0,y(0),\eta(0))$ is the first
nonvanishing Taylor coefficient with respect to the ordering $>_t$
given by Definition
\ref{defordering}, and let
\begin{equation}\label{eq2:orderfirsttaylorcoef1}
m=j_0+k_0+|\alpha_0|+|\beta_0|.
\end{equation}
Now let $\kappa$ be a positive integer such that $m<\kappa$, and
sort the terms in $I_\tau$, given by \eqref{intdef}, with respect to homogeneity degree in $\tau$.
We can use Lemma \ref{speciallemaction}
and the classicality of
the symbol $\sigma_{R^\ast}$ to write
\begin{align*}
R^\ast(t,x,D)v_\tau
& = \sum_{j=-1}^{M'} q_{-j}(t,x,D) v_\tau  +\mathcal{O}(\tau^{N+n-M'-1}) \\
& = \sum_{j=-1}^{M'} \sum_{l=0}^M \tau^{N+n-l} q_{-j}(t,x,D) (e^{i\tau w } \phi_l)  +\mathcal{O}(\tau^{N+n-M'-1})
\end{align*}
for some large number $M'$. Note that \eqref{eqM} implies a lower bound on $M$,
but as we shall see below, we must also make sure to pick $M> 2M'+1$.
For each $j$ we then estimate $q_{-j}(t,x,D)(e^{i\tau w } \phi_l)$ using \eqref{specialeqaction}
with $k=M-1-2j$, so that
\begin{align*}q_{-j}(t,x,D)(e^{i\tau w } \phi_l) =\sum_{|\alpha|<M-1-2j} 
& q_{-j}^{(\alpha )}
(t,x,\tau \eta) (D-\tau \eta )^\alpha (\phi_l e^{i\tau w })/\alpha ! 
\end{align*}
with an error of order $\mathcal{O}(\tau^{(1-M)/2})$.
Recalling \eqref{eqM} and the discussion following Lemma \ref{speciallemaction}
regarding the homogeneity of the terms in \eqref{specialeqaction}, this yields
\begin{align}\label{eqhomo1}
R^\ast(t,x,D)v_\tau
& = \sum_{j=-1}^{M'} \sum_{l=0}^M \tau^{N+n-l} e^{i\tau w} \notag \\
& \phantom{=} \ \times  \sum_{|\alpha|<M-1-2j} 
q_{-j}^{(\alpha )}
(t,x,\tau w_x') D^\alpha \phi_l + \mathcal{O}(\tau^{-\kappa-1}) \notag \\
& = \tau^{N+n} e^{i\tau w} \sum_{j=-1}^{M'} \sum_{l=0}^M \sum_{|\alpha|<M-1-2j} \tau^{-j-|\alpha|-l} \notag \\
& \phantom{=} \ \times  q_{-j}^{(\alpha )}
(t,x, w_x') D^\alpha \phi_l 
+ \mathcal{O}(\tau^{-\kappa-1})
\end{align}
if $M'$ is sufficiently large. Note that
$\tau^{-j-|\alpha|-l} q_{-j}^{(\alpha )}
(t,x, w_x') D^\alpha \phi_l$ is now homogeneous of order $-j-|\alpha| - l$ in $\tau$, and that as before,
$q_{-j}^{(\alpha)}(t,x, w_x')$ should be replaced
by a finite Taylor expansion at $\eta$ of sufficiently high order. For each $-1 \leq J\leq \kappa$,
collect all terms of the form $\tau^{-j-|\alpha|-l} q_{-j}^{(\alpha )}
(t,x, w_x') D^\alpha \phi_l$ in \eqref{eqhomo1}
that are homogeneous of order $-J$ in $\tau$, that is,
all terms that satisfy $j+|\alpha|+l=J$ for $j\geq -1$, and $|\alpha|, l \geq 0$.
If
\begin{equation*}
\lambda_J(t,x) 
= \sum_{j+|\alpha|+l=J}
q_{-j}^{ (\alpha )}(t,x,w_x'(t,x)) 
D^\alpha \phi_l (t,x)
\end{equation*}
for the permitted values of $j$ and $l$, then
\begin{equation*}
\begin{aligned}
I_\tau = \tau^n \iint & H(\tau t, \tau(x-y(t)) ) \\
& \times \Big( e^{i\tau w(t,x)} 
\sum_{J=-1}^\kappa
\tau^{-J} \lambda_J(t,x) 
+ \mathcal{O}(\tau^{-\kappa-1}) \Big) dt \, dx.
\end{aligned}
\end{equation*}
After the change of variables $(\tau t,\tau(x-y(t)))\mapsto (t,x)$ we obtain
\begin{equation}\label{eq:intlambda}
\begin{aligned}
I_\tau = \iint & H(t, x) \Big( e^{i\tau w(t/\tau ,x/\tau + y(t/\tau ))} \sum_{J=-1}^\kappa
\tau^{-J}  \\
&\times \lambda_J(t/\tau,x/\tau+y(t/\tau))
+ \mathcal{O}(\tau^{-\kappa-1}) \Big) dt \, dx,
\end{aligned}
\end{equation}
where
\begin{equation}\label{deflambda}
\begin{aligned}
\lambda_J&(t/\tau,x/\tau+y(t/\tau)) 
= \sum_{j+|\alpha|+l=J} 
D^\alpha \phi_l (t/\tau,x/\tau+y(t/\tau))\\
& \quad \quad \quad \times q_{-j}^{ (\alpha )}(t/\tau,x/\tau+y(t/\tau),w_x'(t/\tau,x/\tau+y(t/\tau))).
\end{aligned}
\end{equation}
Recall that
$w_0(0)=0$, which together with \eqref{eq:defw} implies
\begin{align*}
i \tau 
w(t/\tau, x/\tau+y(t/\tau)) 
=it w_0'(0)
+i\langle x, \eta(t/\tau) \rangle + \mathcal{O}(\tau^{-1}).
\end{align*}
Hence
{\setlength\arraycolsep{2pt}
\begin{equation}\label{int1}
\lim_{\tau \to \infty} e^{i\tau w(t/\tau ,x/\tau + y(t/\tau ))}
=  e^{itw_0'(0)+ i\langle x, \eta (0) \rangle}.
\end{equation}}
In the sequel we shall also need
\begin{equation}\label{ordow}
\begin{aligned}
\partial w / \partial x_j (t/\tau,x/\tau & +y(t/\tau)) - \eta_j (t/\tau)\\
& =\sum_{k=1}^{n-1} w_{j,k}(t/\tau)(x_k/\tau) +\mathcal{O}(\tau^{-2}),
\end{aligned}
\end{equation}
which follows from the definition of $w$ and the fact that $w_\alpha$ is symmetric in these special indices
$\alpha$. In particular, $w_{j,k}(t)=w_{k,j}(t)$ for all $j$, $k\in [1, n-1]$.

Recall that we chose the integer $\kappa$ such that $m<\kappa$.
By Proposition \ref{intest1} there is
a constant $C$ such that
\begin{equation}\label{eq:integralest1}
|I_\tau|\leq C\tau^{-\kappa},
\end{equation}
and we shall now
show that if $\partial_t^{k_0} q_{-j_0 (\alpha_0)}^{(\beta_0)}(0,y(0),\eta(0))$ is the first
nonvanishing Taylor coefficient with respect to the ordering $>_t$, where
$m=j_0+k_0+|\alpha_0|+|\beta_0|$, then \eqref{eq:integralest1}
cannot hold. (Since we are denoting the variables by $(t,x)$ now,
the index $\alpha$ in Definition \ref{defordering} will be replaced by the pair $(k,\alpha)\in\mathbb{N}\times
\mathbb{N}^{n-1}$.) We will do this by determining the limit of $\tau^m I_\tau$ as $\tau\to\infty$.
To see what is needed, consider $\lambda_{-1}(t/\tau,x/\tau+y(t/\tau))$ and
recall that this is
\begin{align*}
q_1(t/\tau,x/\tau + y(t/\tau), w_x'(t/\tau,x/\tau +y(t/\tau))) 
\phi_0 ( t/\tau,x/\tau +y(t/\tau) )
\end{align*}
which should be regarded as
a Taylor expansion in $\xi$ of $q_1$ at $\eta(t/\tau)$ of finite order. The same applies to all the
other terms of the form
$q_{-j}^{(\alpha)}$. Note that for given
$j$ and $\alpha$, we only ever need to
consider Taylor expansions of $q_{-j}^{(\alpha)}$ of order $\kappa-j-|\alpha|$
in view of \eqref{eq:intlambda} and \eqref{ordow}. To keep things
simple, we shall first only consider $q_1$; it will be clear by symmetry what
the corresponding expressions for the other terms should be.
Thus,
\begin{equation}\label{eqtaylorexp01}
\begin{aligned}
q_1  (& t/\tau,x/\tau +y(t/\tau),w_x'(t/\tau,x/\tau+y(t/\tau))) \\
& =\sum_{|\beta|\leq \kappa+1} q_1^{(\beta)}(t/\tau,x/\tau +y(t/\tau),\eta(t/\tau))  \\ 
&\phantom{=} \times ( w_x'(t/\tau,x/\tau+y(t/\tau) )-\eta(t/\tau))^\beta / |\beta| !  +\mathcal{O}(\tau^{-\kappa-2}),
\end{aligned}
\end{equation}
which shows that to use our assumption regarding the Taylor
coefficient $\partial_t^{k_0} q_{-j_0 (\alpha_0)}^{(\beta_0)}(0,y(0),\eta(0))$,
we have to for each $\beta$
write $q_1^{(\beta)}(t/\tau,x/\tau +y(t/\tau),\eta(t/\tau))$ as a Taylor series 
at $\eta(0)$, in addition to having to expand each term as a Taylor series in $t$ and $x$.
However, it is immediate from \eqref{ordow} that
if $\beta$ is an $(n-1)-$tuple corresponding to a given differential operator $D_\xi^{\beta}$, then there is a
sequence $\tilde{\beta}=(\tilde{\beta}_1, \ldots , \tilde{\beta}_s)$ of $s=|\beta|$ indices
between $1$ and the dimension $n-1$ of the $x$ variable such that
\begin{equation}\label{eq:defgbeta}
g_\tau^\beta(t,x)=(w_x'(t/\tau,x/\tau+y(t/\tau) )-\eta(t/\tau))^\beta,
\end{equation}
as it
appears in \eqref{eqtaylorexp01}, satisfies
\[
\begin{aligned}
g_\tau^\beta(t,x) & = 
c_\beta (t/\tau, x/\tau)+\mathcal{O}(\tau^{-|\beta|-1}),
\end{aligned}
\]
where
\[
c_\beta(t/\tau,x/\tau) =
\prod_{j=1}^s \Big(
\sum_{k=1}^{n-1} w_{k,\tilde{\beta}_j}(t/\tau) x_k/\tau \Big)
\]
and $c_\beta(0,x/\tau)=\tau^{-|\beta|}c_\beta(0,x)$.
These expressions make sense if we choose the sequence $\tilde{\beta}$ to be increasing,
for then it is uniquely determined by $\beta$.
If for instance $D_\xi^\beta = -\partial^2 / \partial \xi_i\partial \xi_j$,
then $\tilde{\beta}=(i,j)$ if $i\le j$ (see the indices $\alpha$
used in connection with $w_\alpha$ in \eqref{eq:defw}). Thus \eqref{eqtaylorexp01} takes the form
\begin{equation*}
\begin{aligned}
q_1  (& t/\tau,x/\tau +y(t/\tau),w_x'(t/\tau,x/\tau+y(t/\tau))) \\
& =\sum_{|\beta|\leq \kappa+1} q_1^{(\beta)}(t/\tau,x/\tau +y(t/\tau),\eta(t/\tau))  
g_\tau^\beta(t,x) / |\beta| !  +\mathcal{O}(\tau^{-\kappa-2}),
\end{aligned}
\end{equation*}
and if we expand each term in this expression as a Taylor series at $\eta(0)$ we
obtain
\begin{equation}\label{eqtaylorexp02}
\begin{aligned}
q_1  (& t/\tau,x/\tau +y(t/\tau),w_x'(t/\tau,x/\tau+y(t/\tau))) \\
& =\sum_{|\beta|\leq \kappa+1} \ \sum_{|\gamma|\leq \kappa+1-|\beta|}
q_1^{(\beta+\gamma)}(t/\tau,x/\tau +y(t/\tau),\eta(0))  \\ 
&\phantom{=} \times  g_\tau^\beta(t,x)
( \eta(t/\tau)-\eta(0) )^\gamma / (|\beta| ! |\gamma|!)
+\mathcal{O}(\tau^{-\kappa-2})
\end{aligned}
\end{equation}
where we regard $\eta(t/\tau)-\eta(0)$ as a finite Taylor series
\[\eta'(0)t/\tau+\eta''(0)t^2/(2\tau^2)+\ldots
\]
of sufficiently high order to
maintain control of the error term in \eqref{eqtaylorexp02}.
If we for each multi-index $\beta$ let $G_\tau^\beta(t,x)$ be given by
\[
G_\tau^\beta(t,x)=\! \sum_{\gamma_1+\gamma_2=\beta} \!
( \eta(t/\tau)-\eta(0) )^{\gamma_1} g_\tau^{\gamma_2}(t,x)/(|\gamma_1|! |\gamma_2|!)
\]
for $\gamma_j\in\mathbb{N}^{n-1}$, then the required order of the Taylor expansion
$\eta(t/\tau)-\eta(0)$ will ultimately depend on $\beta$, so we can write
\begin{equation}\label{eqtaylorexp03}
\begin{aligned}
q_1  (&t/\tau,x/\tau +y(t/\tau),w_x'(t/\tau,x/\tau+y(t/\tau))) \\
& =\sum_{|\beta|\leq \kappa+1}
q_1^{(\beta)}(t/\tau,x/\tau +y(t/\tau),\eta(0)) 
G_\tau^\beta(t,x)
+\mathcal{O}(\tau^{-\kappa-2})
\end{aligned}
\end{equation}
and
we can always bound $G_\tau^\beta(t,x)$ by a constant times $\tau^{-|\beta|}$.
As it turns out,
the value of $G_\tau^\beta(t,x)$ for $|\beta|>0$ will not be important which will be evident in a moment.
For notational purposes, denote by $G_0^\beta(t,x)$ the limit of
$\tau^{|\beta|}G_\tau^\beta(t,x)$ as $\tau\to\infty$. Since $G_\tau^\beta(t,x)=1$ when $\beta=0$
it is clear that $G_0^0(t,x)=1$.

For each $\beta$ we must now write $q_1^{(\beta)}(t/\tau,x/\tau +y(t/\tau), \eta(0))$ as a Taylor expansion
in $t$ and $x$ at $0$ and $y(0)$, respectively. As before,
for given $j$ and $\alpha$, we will only have to consider Taylor expansions of $q_{-j}^{(\alpha)}$ of order
$\kappa-j-|\alpha|$. By \eqref{deflambda} and \eqref{eqtaylorexp03} we have
\begin{equation}\label{deflambda1}
\begin{aligned}
\lambda_{-1}&
(t/\tau,x/\tau +y(t/\tau))=
\sum_{k+|\alpha|+|\beta|\leq \kappa+1} \phi_0(t/\tau,x/\tau+y(t/\tau))\\
& \quad \quad \times \Big\{(t/\tau)^k (x/\tau+y(t/\tau)-y(0))^\alpha G_\tau^\beta(t,x) \\
& \quad \quad \times \partial_t^k q_{1 (\alpha)}^{(\beta)}(0,y(0),\eta(0))/(k! |\alpha|!) 
+\mathcal{O}(\tau^{-\kappa-2})\Big\}
\end{aligned}
\end{equation}
where we in $(x/\tau+y(t/\tau)-y(0))^\alpha$ regard $y(t/\tau)-y(0)$
as a finite Taylor series of sufficiently high order to
maintain control of the error terms.

In the same way as we obtained the expression \eqref{deflambda1}
for the term $q_{1}(t/\tau,x/\tau+y(t/\tau),w_x'(t/\tau,x/\tau+y(t/\tau)))$,
we can now obtain similar expressions of appropriate order for all the terms
$q_{-j}^{ (\gamma )}(t/\tau,x/\tau+y(t/\tau),w_x'(t/\tau,x/\tau+y(t/\tau)))$
that appear in \eqref{deflambda}. For each $j$ and $\gamma$ we have
\begin{equation}\label{eq:qjs}
\begin{aligned}
q_{-j}^{ (\gamma )}&
(t/\tau,x/\tau+y(t/\tau),w_x'(t/\tau,x/\tau+y(t/\tau)))\\
&=\sum_{k+|\alpha|+|\beta|\leq \kappa-j-|\gamma|} (t/\tau)^k 
(x/\tau+y(t/\tau)-y(0))^\alpha G_\tau^\beta(t,x) \\
& \quad \quad \times \partial_t^k q_{-j (\alpha)}^{(\beta+\gamma)}(0,y(0),\eta(0))/(k! |\alpha|!) 
+\mathcal{O}(\tau^{-\kappa-1+j+|\gamma|}).
\end{aligned}
\end{equation}
This together with \eqref{deflambda} gives
\begin{equation}\label{eq:lambdaJs}
\begin{aligned}
\lambda_{J}&
(t/\tau,x/\tau+y(t/\tau) ) = \sum_{j+l+|\gamma|=J} \
\sum_{k+|\alpha|+|\beta|\leq \kappa-j-|\gamma|} (t/\tau)^k \\
& \times (x/\tau+y(t/\tau)-y(0))^\alpha G_\tau^\beta(t,x) D_x^\gamma \phi_l(t/\tau,x/\tau+y(t/\tau)) \\
& \times \partial_t^k q_{-j (\alpha)}^{(\beta+\gamma)}(0,y(0),\eta(0))/(k! |\alpha|!) 
+\mathcal{O}(\tau^{-\kappa-1+j+|\gamma|})
\end{aligned}
\end{equation}
where $-1\leq j\leq J$ and $l\geq 0$.
Using the fact that
by assumption the Taylor coefficients
$\partial_t^k q_{-j (\alpha)}^{(\beta+\gamma)}(0,y(0),\eta(0))$ vanish for all
$-1\leq j+k+|\alpha|+|\beta|+|\gamma|<m$, and
\[
\tau^{-J-k-|\alpha|}=\tau^{|\beta|}\tau^{-j-k-|\alpha|-|\beta|-|\gamma|-l}
\]
when $J=j+l+|\gamma|$, \eqref{eq:lambdaJs} yields
\begin{equation*}\label{eq:lambdaJs2}
\begin{aligned}
\sum_{J=-1}^m\tau^{-J}\lambda_{J}&
(t/\tau,x/\tau+y(t/\tau) ) =  \sum_{j+l+|\gamma|=-1}^m \
\sum_{j+k+|\alpha|+|\beta|+|\gamma|=m} \tau^{-m-l}\\
&\quad  \times t^k(x+y'(0)t)^\alpha \tau^{|\beta|}G_\tau^\beta(t,x)
D_x^\gamma \phi_l(t/\tau,x/\tau+y(t/\tau))\\
& \quad \times \partial_t^k q_{-j (\alpha)}^{(\beta+\gamma)}(0,y(0),\eta(0))/(k! |\alpha|!) 
+\mathcal{O}(\tau^{-m-1-l}),
\end{aligned}
\end{equation*}
where $\tau^{|\beta|}G_\tau^\beta(t,x)\to G_0^\beta (t,x)$ as $\tau\to\infty$.
As we can see, the expression above is $\mathcal{O}(\tau^{-m-1})$ as soon as
$l>0$, so in view of \eqref{eq:intlambda} and \eqref{int1}
we obtain
\begin{equation}\label{eq:lambdaJs3}
\begin{aligned}
\lim_{\tau\to\infty} \tau^m I_\tau& =\iint H(t,x)e^{itw_0'(0)+ i\langle x, \eta (0) \rangle}
\Big\{
\sum_{j+k+|\alpha|+|\beta|+|\gamma|=m} t^k\\
&\quad \quad \times (x+y'(0)t)^\alpha G_0^\beta(t,x) D_x^\gamma \phi_0(0,y(0))\\
& \quad \quad \times \partial_t^k q_{-j (\alpha)}^{(\beta+\gamma)}(0,y(0),\eta(0))/(k! |\alpha|!) \Big\}
d t \, d x.
\end{aligned}
\end{equation}
Recall \eqref{eq2:orderfirsttaylorcoef1} and choose
$\phi_0$ such that $D_x^{\beta_0}\phi_0(0,y(0))=1$, but so that $D_x^\gamma \phi_0(0,y(0))=0$ for all
other $\gamma$ such that $|\gamma|\leq |\beta_0|$ (see \eqref{eq:constructphis}).
By the choice of our ordering $>_t$ we have $\partial_t^k q_{-j (\alpha)}^{(\beta+\beta_0)}
(0,y(0),\eta(0))=0$ for all $\beta$ such that $|\beta|>0$ as long as $j+k+|\alpha|+|\beta|+|\beta_0|=m$.
Hence, with this choice of $\phi_0$, \eqref{eq:lambdaJs3} takes the form
\begin{equation}\label{eq:lambdaJs4}
\begin{aligned}
\lim_{\tau\to\infty} \tau^m I_\tau& =\iint H(t,x)e^{itw_0'(0)+ i\langle x, \eta (0) \rangle}
\Big\{
\sum_{j+k+|\alpha|+|\beta_0|=m} t^k\\
& \times (x+y'(0)t)^\alpha 
\partial_t^k q_{-j (\alpha)}^{(\beta_0)}(0,y(0),\eta(0))/(k! |\alpha|!) \Big\}
d t \, d x,
\end{aligned}
\end{equation}
so as promised, the value of $G_0^\beta(t,x)$ for $|\beta|>0$ does not matter.
(Note that $G_0^0(t,x)$ is present in \eqref{eq:lambdaJs4} as the constant factor 1.)
As in the proof of Theorem \ref{mainthm2},
some of the Taylor coefficients in \eqref{eq:lambdaJs4}
may be zero. In particular, the expression may
well contain Taylor coefficients that
preceed $\partial_t^{k_0} q_{-j_0 \,(\alpha_0)}^{(\beta_0)}(0,y(0),\eta(0))$
in the ordering, and
those are by assumption zero. In contrast to the proof of Theorem \ref{mainthm2}
we shall have to exploit this fact, since the coefficient of most of the monomials in
\eqref{eq:lambdaJs4} will be
a linear combination of the Taylor coefficients due to the factor
$(x+y'(0)t)^\alpha$. However, the ordering $>_t$ was chosen so that there can be no
nonzero Taylor coefficient $\partial_t^{k} q_{-j \,(\alpha)}^{(\beta_0)}(0,y(0),\eta(0))$
such that $k+|\alpha|>k_0+|\alpha_0|$, or $k+|\alpha|=k_0+|\alpha_0|$ and $k<k_0$.
This follows immediately from the choice of lexiographic order on the $n$-tuple
$(k,\alpha)\in\mathbb{N}^n$. (Recall that in the definition of the ordering
$>_t$, $x$ denoted all the variables in $\mathbb{R}^n$, while here we denote those variables
by $(t,x)$.) Hence, the only coefficient of the monomial
$t^{k_0}x^{\alpha_0}$ in \eqref{eq:lambdaJs4}
is $\partial_t^{k_0} q_{-j_0 \,(\alpha_0)}^{(\beta_0)}(0,y(0),\eta(0))$. We may therefore,
as in the proof of Theorem \ref{mainthm2}, choose $H$ so that the limit in
\eqref{eq:lambdaJs4} is nonzero. Since this contradicts \eqref{eq:integralest1},
Theorem \ref{mainthm1} follows
in view of the discussion following
Lemma \ref{lem26.4.14}.

\appendix

\section{}\label{appendix1}
\noindent Here we prove a few results used in the main text, related to how the property
that all terms in the asymptotic expansion of the total symbol have vanishing
Taylor coefficients is affected by various operations.

\begin{lem}\label{symbolvanishafterconjugation}
Suppose $X$ and $Y$ are two $C^\infty$ manifolds of the same dimension $n$.
Let $K\subset T^\ast(X)\smallsetminus 0$ and
$K'\subset T^\ast(Y)\smallsetminus 0$ be compactly based cones and let $\chi$ be a
homogeneous symplectomorphism from a conic neighborhood of $K'$ to one of $K$ such that
$\chi(K')=K$. Let $A\in I^{m'}(X\times Y, \varGamma')$ and $B\in I^{m''}(Y\times X,(\varGamma^{-1})')$
where $\varGamma$ is the graph of $\chi$, and assume that $A$ and $B$ are properly supported and
non-characteristic at the restriction of the graphs of $\chi$ and $\chi^{-1}$ to $K'$ and to $K$ respectively,
while $W\! F'(A)$ and $W\! F'(B)$ are contained in small conic neighborhoods. If $R$ is a
properly supported classical pseudo-differential
operator in $Y$, then each term in the asymptotic expansion of the total (left) symbol of $R$
has vanishing Taylor coefficients at a point $(y,\eta)\in K'$ if and only if
each term in the asymptotic expansion of the total (left) symbol of the pseudo-differential operator
$ARB$ in $X$ has vanishing Taylor coefficients at $\chi(y,\eta)\in K$.
\end{lem}
\begin{proof}
We may assume that we have a homogeneous generating function $\varphi\in C^\infty$ for the symplectomorphism
$\chi$ (see ~\cite[pp. $101-103$]{grisjo}). Then $\chi$ is locally of the form
\[
(\partial \varphi (x,\eta)/\partial \eta , \eta)\mapsto (x, \partial \varphi (x,\eta)/\partial x),
\]
and $A$ and $B$ are given by
\[
Au(x)=\frac{1}{(2\pi)^n}\iint e^{i(\varphi(x,\zeta)-z\cdot\zeta)}a(x,z,\zeta)u(z) \, dz \, d\zeta,
\]
\[
Bv(y)=\frac{1}{(2\pi)^n}\iint e^{i(y\cdot\theta-\varphi(s,\theta))}b(y,s,\theta)v(s) \, ds \, d\theta.
\]
Since $R$ is properly supported we may assume that
\begin{equation}\label{eq:symvancon0}
Ru(z)=\frac{1}{(2\pi)^n}\int e^{iz\cdot\eta}r(z,\eta)\hat{u}(\eta)d \, \eta, \quad u\in C_0^\infty(Y),
\end{equation}
where $r(z,\eta)=\sigma_R$ is the total symbol of $R$. Hence
\begin{equation}\label{eq:symvancon1}
\begin{aligned}
ARBu(x)&=\frac{1}{(2\pi)^{3n}}\int e^{i(\varphi(x,\zeta)-z\cdot\zeta+(z-y)\cdot\sigma+y\cdot\theta-\varphi(s,\theta))}\\
& \phantom{=} \, \times a(x,z,\zeta)r(z,\sigma)b(y,s,\theta)u(s) \, ds \, d\theta \, dy \, d\sigma \, dz \, d\zeta,
\end{aligned}
\end{equation}
since $B$ being properly supported implies that $Bu\in C_0^\infty(Y)$
when $u\in C_0^\infty(Y)$.
Using integration by parts in $z$, we see that we can insert
a cutoff $\phi((\zeta-\sigma)/|\sigma|)$ in the last integral without changing the operator
$ARB$
mod $\varPsi^{-\infty}$. If we make the change of variables
$\tau=\zeta-\sigma$, then \eqref{eq:symvancon1} takes the form
\begin{equation*}
\begin{aligned}
A&RBu(x)=\frac{1}{(2\pi)^{3n}}\int \phi(\tau/|\sigma|)e^{i(\varphi(x,\tau+\sigma)-z\cdot(\tau+\sigma)
+(z-y)\cdot\sigma
+y\cdot\theta-\varphi(s,\theta))}\\
& \phantom{=} \, \times a(x,z,\tau+\sigma)r(z,\sigma)b(y,s,\theta)u(s) \, ds \, d\theta \, dy \, d\sigma \, dz \, d\tau
+Lu,
\end{aligned}
\end{equation*}
with $L\in\varPsi^{-\infty}$. If $\varOmega\subset \mathbb{R}^{2n}$ is open and
$\tilde{\varphi}\in C^\infty(\varOmega,\mathbb{R})$ is a phase function with a non-degenerate critical point
$x_0\in \varOmega$ such that $d \tilde{\varphi}\ne 0$ everywhere else, then ~\cite[Proposition $2.3$]{grisjo}
states, in particular, that
for every compact $M\subset \varOmega$ and every $u\in C^\infty (\varOmega)\cap \mathscr{E}'(M)$ we have
\begin{equation}\label{eq:symvancon2}
\begin{aligned}
\big|\int e^{i\lambda \tilde{\varphi}(x)} & u(x)dx-e^{i\lambda \tilde{\varphi}(x_0)}A_0u(x_0)\lambda^{-n}\big| \\
& \leq C_M\lambda^{-n-1}\sum_{|\alpha|\leq 2n+3} \sup |\partial^\alpha u(x)|, \quad \lambda \geq 1,
\end{aligned}
\end{equation}
where
\begin{equation}\label{eq:symvancon2.5}
A_0=\frac{(2\pi)^{n}\cdot e^{i\pi \sgn  \tilde{\varphi}''(x_0)/4}}{|\det \tilde{\varphi}''(x_0)|^{1/2}}.
\end{equation}
It is clear that the result extends to the setting $\varOmega=T^\ast(\mathcal{N})\smallsetminus 0$ where
$\mathcal{N}$ is a $C^\infty$ manifold of dimension $n$. 
In order to apply the result, we put
$\sigma=\lambda\omega$,
and make the change of variables $\tau=\lambda\tilde{\tau}$.
After dropping the $\tilde{\phantom{\tau}}$ we obtain
\begin{equation*}
\begin{aligned}
A&RBu(x)=\frac{\lambda^{2n}}{(2\pi)^{3n}}\int \phi(\tau/|\omega|)
e^{i\lambda(\varphi(x,\tau+\omega)-z\cdot(\tau+\omega)
+y\cdot\theta/\lambda+(z-y)\cdot\omega-\varphi(s,\theta)/\lambda)}\\
& \phantom{=} \, \times a(x,z,\lambda(\tau+\omega))r(z,\lambda\omega)
b(y,s,\theta)u(s) \, ds \, d\theta \, dy \, d\omega \, dz \, d\tau
+Lu,
\end{aligned}
\end{equation*}
where we have used the fact that $\varphi$ is homogeneous of degree $1$ in the fiber.
For the $z,\tau$-integration we have the non-degenerate critical point given by
$\tau=0, z=\varphi_{\zeta}'(x,\tau+\omega)$. Note that since
$\varphi_\zeta'$
is homogeneous of degree $0$ in the fiber we have $\varphi_\zeta'(x,\sigma/\lambda)
=\varphi_\zeta'(x,\sigma)$, so this critical point corresponds to the critical point
for the $z, \zeta$-integration given by $\zeta=\sigma, z=\varphi_\zeta'(x,\sigma)$.
Hence the above expression together with
\eqref{eq:symvancon2}
imply that
\begin{equation*}
\begin{aligned}
ARBu(x)&=C\lambda^{2n}\int
e^{i(\varphi(x,\lambda\omega)
+y\cdot\theta-y\cdot\lambda\omega-\varphi(s,\theta) )}\\
& \phantom{=} \, \times w(x,y,s,\omega,\theta) u(s) \, ds \, d\theta \, dy \, d\omega
+Lu,
\end{aligned}
\end{equation*}
where
\begin{align*}
w(x,y,s,\omega,\theta) & = \frac{A_0}{\lambda^{n}} a(x,z,\lambda(\tau+\omega))r(z,\lambda\omega)
b(y,s,\theta) \phi(\tau/|\omega|) \Big|_{ \begin{subarray}{l}\tau=0\\z=\varphi_\zeta'(x,\omega) \end{subarray}} \\
&=\frac{A_0}{\lambda^{n}} a(x,\varphi_\zeta'(x,\omega),\lambda\omega)r(\varphi_\zeta'(x,\omega),\lambda\omega)
b(y,s,\theta)
\end{align*}
with an error of order $\mathcal{O}(\lambda^{-n-1})$. Note that $A_0$ is now a function of
$x$ and $\omega$, since the matrix corresponding to $\tilde{\varphi}''(x_0)$ in \eqref{eq:symvancon2.5}
is given by the block matrix
\begin{equation}\label{eq:symvancon2.6}
F=\left( \begin{array}{cc}
0 & -\mathrm{Id}_n \\
-\mathrm{Id}_n & \varphi_{\zeta \zeta}''(x,\omega)
\end{array}\right),
\end{equation}
where $\mathrm{Id}_n$ is the identity matrix on $\mathbb{R}^n$.
Clearly the determinant of $F$ is either $1$ or $-1$, so $F$ is non-singular.
Furthermore, $F$ depends smoothly on the parameters $x$ and $\omega$
since $\varphi\in C^\infty$,
so the eigenvalues of $F$ are continuous in $x$ and $\omega$. Hence it follows that the signature of $F$
is constant, for if not there has to exist an eigenvalue
vanishing at some point $(x,\omega)$, contradicting the non-singularity of $F$.
Reverting to the variable
$\sigma=\lambda\omega$ we thus obtain
\begin{equation*}
\begin{aligned}
ARBu(x)&=C\int
e^{i(\varphi(x,\sigma)
+y\cdot(\theta-\sigma)-\varphi(s,\theta) )}\\
& \phantom{=} \, \times \tilde{w}(x,y,s,\sigma,\theta) u(s) \, ds \, d\theta \, dy \, d\sigma
+Lu,
\end{aligned}
\end{equation*}
where
\begin{equation*}
\tilde{w}(x,y,s,\sigma,\theta) =
a(x,\varphi_\zeta'(x,\sigma),\sigma)r(\varphi_\zeta'(x,\sigma),\sigma)
b(y,s,\theta)
\end{equation*}
with an error of order $\mathcal{O}(\lambda^{-1})$. Taking the limit as $\lambda\to\infty$ yields
\begin{equation*}
\begin{aligned}
ARBu(x)&=C\int
e^{i(\varphi(x,\sigma)
+y\cdot(\theta-\sigma)-\varphi(s,\theta) )}a(x,\varphi_\zeta'(x,\sigma),\sigma)\\
& \phantom{=} \, \times r(\varphi_\zeta'(x,\sigma),\sigma)
b(y,s,\theta)  u(s) \, ds \, d\theta \, dy \, d\sigma
+Lu.
\end{aligned}
\end{equation*}

We can now repeat the procedure. Indeed, we can insert a cutoff
$\phi((\sigma-\theta)/|\theta|)$ without changing the operator mod
$\varPsi^{-\infty}$, and after making the corresponding changes of variables 
in order to apply ~\cite[Proposition $2.3$]{grisjo}
we find that for the $y,\sigma$-integration we have the non-degenerate critical point given
in the original variables by $\sigma=\theta, y=\varphi_\sigma'(x,\sigma)$. After taking the
limit as $\lambda\to\infty$ we obtain
\begin{equation*}
\begin{aligned}
ARBu(x)&=C\int
e^{i(\varphi(x,\theta)
-\varphi(s,\theta) )} w_1(x,s,\theta) u(s) \, ds \, d\theta
+L_1u,
\end{aligned}
\end{equation*}
where $L_1\in\varPsi^{-\infty}$ and
\begin{equation}\label{eq:symvancon3}
w_1(x,s,\theta) = a(x,\varphi_\theta'(x,\theta),\theta)
r(\varphi_\theta'(x,\theta),\theta)
b(\varphi_\theta'(x,\theta),s,\theta).
\end{equation}
As before we let the factor $A_0$ from \eqref{eq:symvancon2.5}
be included in the constant $C$.
In a conic neighborhood of
$\supp w_1$ we can write
\begin{equation*}
\varphi(x,\theta)-\varphi(s,\theta)=(x-s)\varXi(x,s,\theta).
\end{equation*}
Then $\varXi(x,x,\theta)=\varphi_x'(x,\theta)$ so
$\partial \varXi(x,x,\theta)/\partial \theta=\varphi_{x \theta}''(x,\theta)$
is invertible, since $\varphi_{x \theta}''(x,\theta)\neq 0$
is equivalent to the fact that the graph of $\chi$ is (locally)
the graph of a smooth map.
Hence $\theta\mapsto\varXi(x,s,\theta)$ is $C^\infty$, homogeneous of degree $1$
and with an inverse having the same properties. For $s$ close to $x$,
the equation $\varXi(x,s,\theta)=\xi$ then defines $\theta=\varTheta(x,s,\xi)$.
After a change of variables, the last integral therefore takes the form
\begin{equation}\label{eq:symvancon4}
ARBu(x)=C\int e^{i(x-s)\cdot\xi} 
\tilde{w}_1(x,s,\xi)u(s) \, ds \, d\xi +L_1u,
\end{equation}
where $\tilde{w}_1(x,s,\xi)$ is just $w_1(x,s,\varTheta(x,s,\xi))$
multiplied by a Jacobian. We note in passing that evaluating $\tilde{w}_1$
at a point $(x,x,\xi)$ where $\xi$ is of the form $\xi=\varphi_x'(x,\eta)$
therefore involves evaluating $w_1$ at the point $(x,x,\eta)$.
The integral \eqref{eq:symvancon4} defines a pseudo-differential operator
with total symbol $\rho(x,\xi)$ satisfying
\begin{equation}\label{eq:symvancon5}
\rho(x,\xi)\sim\sum \frac{i^{-|\alpha|}}{\alpha!}(\partial_\xi^\alpha \partial_y^\alpha
\tilde{w}_1(x,y,\xi))\rvert_{y=x}.
\end{equation}
If the total symbol $r=\sigma_R$ of $R$
has vanishing Taylor coefficients at a point $(y,\eta)=(\varphi_\eta'(x,\eta),\eta)$,
then by examining \eqref{eq:symvancon5} in decreasing order of homogeneity we find that
each term of $\rho$ must have
vanishing Taylor coefficients at $(x,\xi)=(x,\varphi_x'(x,\eta))$,
since by what we have shown this would involve evaluating $r(z,\sigma)$
and its derivatives at $(\varphi_\eta'(x,\eta),\eta)$.

To prove the converse, choose $A_1\in I^{-m''}(X\times Y, \varGamma')$
and $B_1\in I^{-m'}(Y\times X, (\varGamma^{-1})')$
properly supported such that
\begin{align*}
K'\cap W\! F(BA_1-I)& = \emptyset,  & K  \cap W\! F(A_1B-I) = \emptyset, \\
K'\cap W\! F(B_1A-I)& = \emptyset,  & K  \cap W\! F(AB_1-I) = \emptyset.
\end{align*}
Then a repetition of the arguments above shows that all the terms in the
asymptotic expansion of the total symbol of $B_1ARBA_1$
has vanishing Taylor coefficients at a point $(y,\eta)=(\varphi_\eta'(x,\eta),\eta)$
if all the terms in the asymptotic expansion of the total symbol of $ARB$
has vanishing Taylor coefficients at $(x,\xi)=(x,\varphi_x'(x,\eta))$.
Since $R$ and $B_1ARBA_1$ have the same total symbol in $K'$ mod $\varPsi^{-\infty}$,
the same must hold for the total symbol of $R$.
This completes the proof.
\end{proof}

Let $\{e_k : k=1,\ldots, n\}$ be a basis for $\mathbb{R}^n$, let $(U,x)$ be local coordinates on
a smooth manifold $M$ of dimension $n$, and let
\[
\Big\{\frac{\partial}{\partial x_k}: k=1,\ldots,n\Big\}
\]
be the induced local frame for the tangent bundle $TM$. 
Since the local frame fields commute,
we can use standard multi-index notation to express
the partial derivatives $\partial_x^\alpha f$ of $f\in C^\infty (U)$.

\begin{lem}\label{appthm19}
Let $M$ be a smooth manifold of dimension $n$, and
for $j\ge 1$ let $p,q_j,g_j\in C^\infty(M)$.
Let $\{\gamma_j\}_{j=1}^\infty$ be a sequence in $M$
such that $\gamma_j\to\gamma$ as $j\to\infty$, and assume that
$p(\gamma)=p(\gamma_j)=0$ for all $j$,
and that
$dp(\gamma)\ne 0$.
Let $(U,x)$ be local coordinates on $M$ near $\gamma$,
and suppose that there exists a smooth function $q\in C^\infty(M)$
such that
\[
\partial_x^\alpha q(\gamma)
=\lim_{j\to\infty} \partial_x^\alpha q_j(\gamma_j)
\]
for all $\alpha\in\mathbb{N}^n$.
If $q_j-pg_j$ vanishes of infinite order at $\gamma_j$ for all $j$, then there
exists a smooth function $g\in C^\infty(M)$
such that $q-pg$ vanishes of infinite order at $\gamma$.
Furthermore,
\begin{equation}\label{eq:app20}
\partial_x^\alpha g(\gamma)
=\lim_{j\to\infty} \partial_x^\alpha g_j(\gamma_j)
\end{equation}
for all $\alpha\in\mathbb{N}^n$.
\end{lem}
\begin{proof}
We have stated the result for a manifold, but since the result
is purely local we may assume that $M\subset\mathbb{R}^n$ in the proof.
It is also clear that we may assume that there exists an open neighborhood
$\mathcal{U}$ of $\gamma$ such that $\gamma_j\in \mathcal{U}$ for $j\ge 1$,
and that $dp\ne 0$ in $\mathcal{U}$.
By shrinking $\mathcal{U}$ if necessary, we can then find a unit vector $\nu\in\mathbb{R}^n$ such that
$\partial_\nu p(w)=\langle \nu, dp(w)\rangle\ne 0$ for $w\in \mathcal{U}$.
(We will identify a tangent vector $\nu\in\mathbb{R}^n$ at $\gamma$ with
$\partial_\nu\in T_\gamma \mathbb{R}^n$
through the usual vector space isomorphism.)
Hence $\partial_\nu p(w)$ is
invertible in $\mathcal{U}$, and we let $(\partial_{\nu} p(w))^{-1}\in C^\infty(\mathcal{U})$ denote its inverse.
By an orthonormal change of coordinates we may even assume that
$\partial_\nu p(w)=\partial_{e_1}p(w)$.
In accordance with the notation used in the statement of the lemma, we shall
write $\partial_{x_k}p(w)$ for the partial derivatives $\partial_{e_k}p(w)$
and denote by $(\partial_{x_1} p(w))^{-1}$ the inverse of $\partial_{\nu} p(w)=\partial_{x_1} p(w)$
in $\mathcal{U}$.

Now
\begin{equation}\label{eq:app20.5}
0=\partial_{x_1} (q_j-pg_j)(\gamma_j)=\partial_{x_1} q_j(\gamma_j) -  \partial_{x_1} p(\gamma_j) g_j(\gamma_j)
\end{equation}
for all $j$ since $p(\gamma_j)=0$. Since $\lim_j \partial_{x_1} q_j(\gamma_j)=\partial_{x_1} q(\gamma)$
by assumption, equation \eqref{eq:app20.5} yields
\begin{equation}\label{eq:app21}
\lim_{j\to\infty} g_j(\gamma_j)=(\partial_{x_1} p(\gamma))^{-1}\partial_{x_1} q(\gamma)=a\in\mathbb{C}.
\end{equation}

We claim that we can in the same way determine
\[
\lim_{j\to\infty} (\partial_x^\alpha g_j)(\gamma_j)=a_{(\alpha)}\in\mathbb{C}
\]
for any $\alpha\in\mathbb{N}^n$.
We start by
determining
\[
\lim_{j\to\infty} \partial g_j(\gamma_j) / \partial x_k=a_{(k)}
\]
for $1\leq k\leq n$. By the hypotheses of the lemma we have
\begin{equation}\label{eq:app22}
\begin{aligned}
0&=\partial_{x_k}\partial_{x_l}(q_j-pg_j)(\gamma_j)\\
&= \partial_{x_k}\partial_{x_l} q_j(\gamma_j)-
\partial_{x_k}\partial_{x_l} p(\gamma_j)g_j(\gamma_j)\\
&\phantom{=}-\partial_{x_k} p(\gamma_j) \partial_{x_l} g_j(\gamma_j)
-\partial_{x_l} p(\gamma_j) \partial_{x_k} g_j(\gamma_j)
\end{aligned}
\end{equation}
since $p(\gamma_j)=0$. For $k=l=1$ we obtain from \eqref{eq:app21} and \eqref{eq:app22}
\begin{equation}\label{eq:app23}
\lim_{j\to\infty} \partial_{x_1} g_j(\gamma_j) 
=(\partial_{x_1} p(\gamma))^{-1}
\big(\partial_{x_1}^2 q(\gamma)
-\partial_{x_1}^2 p(\gamma) a \big)/2.
\end{equation}
This allows us to solve for $\partial_{x_k} g_j(\gamma_j)$ in \eqref{eq:app22}
by choosing $l=1$. If $b\in\mathbb{C}$ denotes the limit in \eqref{eq:app23}
and $a\in\mathbb{C}$ is given by \eqref{eq:app21} we
thus obtain
\begin{equation*}
\begin{aligned}
\lim_{j\to\infty} \partial_{x_k} g_j(\gamma_j)&=(\partial_{x_1} p(\gamma))^{-1}
\big(\partial_{x_1}\partial_{x_k} q(\gamma)\\
&\phantom{=} -\partial_{x_1}\partial_{x_k} p(\gamma) a
- \partial_{x_k}p(\gamma) b \big)
\end{aligned}
\end{equation*}
for $2\leq k \leq n$.

Now assume that for some $m\geq 3$ we have in this way determined
\[
\lim_{j\to\infty} \partial_{x_{k_1}}\ldots\partial_{x_{k_{m-2}}} g_j(\gamma_j),
\]
for $k_i\in [1,n]$, $i\in [1,m-2]$.
To shorten notation, we will use the (non standard) multi-index notation
introduced on page
\pageref{eq:defgbeta};
to every $\alpha\in\mathbb{N}^n$ with $|\alpha|=m$ corresponds precisely
one $m-$tuple
$\beta=(k_1,\ldots,k_m)$ of
non-decreasing numbers
$1\le k_1\le\ldots\le k_m\le n$
such that $\partial_x^\beta$ equals $\partial_x^\alpha$.
Throughout the rest of this proof we shall let $\beta$ represent
such an $m-$tuple, and we let
\[
\hat{\beta}_i=(k_1,\ldots,k_{i-1},k_{i+1},\ldots,k_m).
\]
As before we have
\begin{equation}\label{eq:app24}
\begin{aligned}
0=\partial_x^\beta (q_j-pg_j)(\gamma_j)
&= \partial_x^\beta  q_j(\gamma_j)-
\partial_x^\beta  p(\gamma_j)g_j(\gamma_j) \\
&\phantom{=}
-\ldots -\sum_{i=1}^m \partial_{x_{k_i}} p(\gamma_j)
\partial_x^{\hat{\beta}_i}g_j(\gamma_j)
\end{aligned}
\end{equation}
by assumption. If we choose $k_i=1$ for all $1\leq i \leq m$, the last sum is just
$m \partial_{x_1} p(\gamma_j)\partial_{x_1}^{m-1}g_j(\gamma_j)$, and since
the limit of all other terms on the right-hand side are known by the induction
hypothesis, we thus obtain the value of the limit of $\partial_{x_1}^{m-1}g_j(\gamma_j)$
from \eqref{eq:app24} by first multiplying by
$m^{-1}(\partial_{x_1}p(\gamma_j))^{-1}$ and then letting $j\to\infty$.
Denote this limit by $c\in\mathbb{C}$.
If we choose $k_i\ne1$ for precisely one $i\in [1,m]$, say $k_m=k$, then the last sum
in \eqref{eq:app24} satisfies
\begin{align*}
\sum_{i=1}^m  \partial_{x_{k_i}} p(\gamma_j) \partial_x^{\hat{\beta}_i}g_j(\gamma_j)
& =\partial_{x_{k}}p(\gamma_j)\partial_{x_{1}}^{m-1} g_j(\gamma_j) \\
& \phantom{=}+(m-1)\partial_{x_{1}}p(\gamma_j)\partial_{x_{1}}^{m-2}\partial_{x_{k}}
g_j(\gamma_j),
\end{align*}
so by the same argument as before we can obtain the value of
\[
\lim_{j\to\infty} \partial_{x_{1}}^{m-2}\partial_{x_{k}}g_j(\gamma_j)
\]
for $2\le k\le n$ by multiplying by
$(m-1)^{-1}(\partial_{x_{1}}p(\gamma_j))^{-1}$ and using
$\partial_{x_{1}}^{m-1} g_j(\gamma_j)\to c$
when taking the limit as $j\to\infty$ in \eqref{eq:app24}.
Continuing this way it is clear that we can successively determine
\begin{equation*}
\lim_{j\to\infty} \partial_{x_{k_1}}\ldots\partial_{x_{k_{m-1}}}g_j(\gamma_j)
\end{equation*}
for any $1\le k_1\le\ldots\le k_{m-1}\le n$
which completely determines
\[
\lim_{j\to\infty} \partial_x^\alpha g_j(\gamma_j)=a_{(\alpha)},
\quad \alpha\in\mathbb{N}^n, \ |\alpha|=m-1.
\]
This proves the claim.

By Borel's theorem there exists a
smooth function $g\in C^\infty(M)$ such that
\begin{equation*}\label{eq:app25}
\partial_x^\alpha g(\gamma)=a_{(\alpha)}
=\lim_{j\to\infty} \partial_x^\alpha g_j(\gamma_j)
\end{equation*}
for all $\alpha\in\mathbb{N}^n$. Since $q-pg$ vanishes of infinite order at $\gamma$ by
construction, this completes the proof.
\end{proof}

The lemma will be used to prove the following result for homogeneous smooth functions on the cotangent bundle.
\begin{prop}\label{appthm26}
For $j\ge 1$ let $p,q_j,g_j\in C^\infty(T^\ast (\mathbb{R}^n)\smallsetminus 0)$, where
$p$ and $q_j$ are homogeneous of degree $m$ and the $g_j$ are homogeneous of degree $0$.
Let $\{\gamma_j\}_{j=1}^\infty$ be a sequence in $T^\ast (\mathbb{R}^n)\smallsetminus 0$
such that $\gamma_j\to\gamma$ as $j\to\infty$, and assume that $p(\gamma)=p(\gamma_j)=0$ for all $j$,
and that $dp(\gamma)\ne 0$.
If there exists a smooth function $q\in C^\infty(T^\ast(\mathbb{R})^n\smallsetminus 0)$, homogeneous of degree $m$,
such that
\[
\partial_x^\alpha\partial_\xi^\beta q(\gamma)
=\lim_{j\to\infty} \partial_x^\alpha\partial_\xi^\beta q_j(\gamma_j)
\]
for all $(\alpha,\beta)\in\mathbb{N}^n\times\mathbb{N}^n$,
and if $q_j-pg_j$ vanishes of infinite order at $\gamma_j$ for all $j$, then there
exists a $g\in C^\infty(T^\ast(\mathbb{R}^n)\smallsetminus 0)$, homogeneous
of degree $0$, such that $q-pg$ vanishes of infinite order at $\gamma$.
Furthermore,
\begin{equation}\label{eq:app27}
\partial_x^\alpha\partial_\xi^\beta g(\gamma)
=\lim_{j\to\infty} \partial_x^\alpha\partial_\xi^\beta g_j(\gamma_j)
\end{equation}
for all $(\alpha,\beta)\in\mathbb{N}^n\times\mathbb{N}^n$.
\end{prop}
\begin{proof}
Let $\pi : T^\ast (\mathbb{R}^n)\smallsetminus 0 \rightarrow S^\ast (\mathbb{R}^n)$ be the projection.
Since $dp(\gamma)\ne 0$ it follows from homogeneity that
$dp(\pi(\gamma))\ne 0$.
By using the homogeneity of
$q$, $q_j$ and $g_j$
we may even
assume that $\gamma$ and $\gamma_j$ belong to $S^\ast (\mathbb{R}^n)$
for $j\ge 1$ to begin with.

Now, the radial vector field $\xi\partial_\xi$
applied $k$ times to $a\in C^\infty (T^\ast(\mathbb{R}^n)\smallsetminus 0)$
equals $l^ka$ if $a$ is homogeneous of degree $l$. 
For any point $w\in S^\ast(\mathbb{R}^n)$
with $w=(w_x,w_\xi)$ in local coordinates on $T^\ast(\mathbb{R}^n)$
it is easy to see that
\[
T_w S^\ast (\mathbb{R}^n) =\{ (u,v)\in \mathbb{R}^n\times \mathbb{R}^n :
\langle w_\xi , v \rangle = 0\}.
\]
Therefore a basis for $T_w S^\ast (\mathbb{R}^n)$ together with
the radial vector field $(\xi\partial_\xi)_w$ at $w$ constitutes
a basis for $T_w T^\ast (\mathbb{R}^n)$. This implies that if
we can find a homogeneous function $g$ such
that $q-pg$ vanishes of infinite order in the directions
$T_\gamma S^\ast(\mathbb{R}^n)$, then $q-pg$ vanishes of infinite order
at $\gamma$, for the derivatives involving the radial direction are
determined by lower order derivatives in the directions
$T_\gamma S^\ast(\mathbb{R}^n)$.

By the hypotheses of the proposition
together with an application of Lemma \ref{appthm19},
we find that there exists a function $\tilde{g}\in C^\infty (T^\ast (\mathbb{R}^n))$,
not necessarily homogeneous,
such that $q-p\tilde{g}$ vanishes of infinite order at $\gamma$ and
\eqref{eq:app27} holds for $\tilde{g}$.
The function $g(x,\xi)=\tilde{g}(x,\xi/|\xi|)$
coincides with $\tilde{g}$ on $S^\ast (\mathbb{R}^n)$. In particular,
all derivatives of $g$ and $\tilde{g}$ in the directions $T_\gamma S^\ast(\mathbb{R}^n)$
are equal at $\gamma$.
Thus, by the arguments above
we conclude that $q-pg$ vanishes of infinite order at $\gamma$.
Since $g$ and $g_j$ are homogeneous of degree $0$, the same
arguments also imply that \eqref{eq:app27} holds for $g$,
which completes the proof.
\end{proof}

\end{document}